\documentclass[11pt]{amsart}

\usepackage[top=1in, bottom=1in, left=1in, right=1in]{geometry}
\usepackage[dvipsnames]{xcolor}
\usepackage[colorlinks=true, pdfstartview=FitV, linkcolor=RoyalBlue,citecolor=ForestGreen, urlcolor=blue]{hyperref}
\usepackage{xfrac}

\usepackage[english]{babel}
\usepackage{graphicx} %figure, includegraphics
\usepackage{amsthm} %theoremstyle
\usepackage{amsmath}
\usepackage{amssymb}
\usepackage{amsthm} %proof
\usepackage{mathrsfs} %mathscr
\usepackage{bbm} %mathbbm
\usepackage{enumerate} %(i)
\usepackage[shortlabels]{enumitem}
\usepackage{xcolor}
\usepackage{nicefrac}
\usepackage{subcaption}
\usepackage{comment}
\usepackage{hyperref}

\newcommand{\N}{\mathbb{N}}

\newcommand{\R}{\mathbb{R}}

\newcommand{\T}{\mathbb{T}}

\usepackage{amsmath}
\usepackage{amssymb}
\usepackage{amsthm}
\usepackage{amsfonts}
\usepackage{calligra}
\usepackage{appendix}
\usepackage{graphicx}
\usepackage{enumerate}
\usepackage{mathrsfs}
\usepackage[english]{babel}
\usepackage{mathtools}

\newcommand{\defeq}{\mathrel{:\mkern-0.25mu=}}

\newcommand{\norm}[1]{\left\lVert#1\right\rVert}
\newcommand{\abs}[1]{\left\lvert#1\right\rvert}

\def\R{\mathbb R}

\def\N{\mathbb N}

\def\T{\mathbb T}

\def\tp{\textup}
\usepackage{mathtools}

\newcommand{\dif}{\,\mathrm{d}}

\DeclareMathOperator{\dive}{div}

\def\Xint#1{\mathchoice
	{\XXint\displaystyle\textstyle{#1}}%
	{\XXint\textstyle\scriptstyle{#1}}%
        {\XXint\scriptstyle\scriptscriptstyle{#1}}%
	{\XXint\scriptscriptstyle\scriptscriptstyle{#1}}%
	\!\int}
\def\XXint#1#2#3{{\setbox0=\hbox{$#1{#2#3}{\int}$}
		\vcenter{\hbox{$#2#3$}}\kern-.5\wd0}}
\def\dashint{\Xint-}  %mean integral

\newtheorem{lemma}{Lemma}[section]
\newtheorem{prop}[lemma]{Proposition}
\newtheorem{thm}[lemma]{Theorem}
\newtheorem{defi}[lemma]{Definition}
\newtheorem{cor}[lemma]{Corollary}
\newtheorem{Rem}[lemma]{Remark}

\title[Turbulent dynamos and the geometric transport equation]{Turbulent Dynamos on Bounded Domains and Their Generalization to the Geometric Transport Equation}
\author{Giacomo Del Nin, Daniel Faraco, Sauli Lindberg, and Francisco Mengual}
\date{\today}

\begin{document}

\begin{abstract} 
For any smooth bounded domain $\Omega \subset \mathbb{R}^3$, we construct a divergence-free velocity field $u \in L_t^1 W^{1,p}(\Omega)$ for all $p < \infty$, and magnetic fields $B^\epsilon \in L_t^p C^{m}(\Omega)$ for all $p < \infty$ and $m\in \mathbb{N}$, that solve the kinematic dynamo equation and exhibit arbitrarily fast growth of any magnetic energy mode, uniformly in the vanishing-diffusivity limit $\epsilon \to 0$. 
The construction is based on the convex integration scheme of Modena-Sz\'ekelyhidi and Cheskidov–Luo. The main novelty lies in the introduction of explicit potentials, which allow the solutions to be localized and avoid the need to work with the anti-curl operator. In addition, we present a unified scheme for the geometric transport equation (GTE), which encompasses both the transport and Maxwell equations.
\end{abstract}
	
\maketitle

\setcounter{tocdepth}{1}
\tableofcontents

\section{Introduction and main results}

The aim of this work is threefold: (i) to construct a broad class of magnetic fields exhibiting turbulent dynamo action in smooth bounded domains, (ii) to present a simple convex integration scheme, inspired by \cite{ModenaSzekelyhidi18,BBV20,CheskidovLuo24}, but based on explicit potentials rather than antiderivatives, and (iii) to show its robustness by extending it to the GTE, which applies to both Maxwell and transport systems, thereby recovering earlier results with the added feature of compact spatial support. 

\subsection{Kinematic dynamos}
\textit{Dynamo theory} studies the mechanisms by which electrically conducting fluids generate and sustain magnetic fields in celestial bodies \cite{ArnoldKhesin98,ChildressGilbert03,Gilbert03}. 
Since most planetary, stellar, and galactic flows are much slower than the speed of light,
the electromagnetic field $(E,B)$ is well described by the non-relativistic Maxwell equations. The homogeneous part 
reads
\begin{align}
\partial_t B+\nabla \times E=0, 
\label{Faraday}\tag{Faraday} \\
\nabla\cdot B=0, \label{Gauss}\tag{Gauss}
\end{align}
while the inhomogeneous part relates the charge and current densities to the fields through Gauss’s and Amp\`ere’s laws: $\rho=\nabla\cdot E$, $J=\nabla\times B$.   
In a highly conducting fluid, such as the Sun or the Earth’s core, Amp\`ere’s and Ohm’s laws combine into
\begin{equation}\label{AmpereOhm}
\nabla\times B=
J=\epsilon^{-1}(E+u\times B),
\tag{Amp\`ere-Ohm}
\end{equation}
where $\epsilon > 0$ is the magnetic resistivity and $u$ is the incompressible fluid velocity field ($\nabla\cdot u=0$).
Thus, the electric field $E$ is 
a secondary quantity determined by $(u,B)$ through the \eqref{AmpereOhm} law. 
Substituting this expression into the \eqref{Faraday} law and using the vector calculus identity valid for solenoidal fields, 
$\nabla\times(\nabla\times B)
=-\Delta B$,
yields the well-known \textit{kinematic dynamo equation}
\begin{equation}\label{dynamo1}\tag{dynamo}
\begin{split}
\partial_tB =
\nabla\times(u\times B)+\epsilon\Delta B,\\
\nabla\cdot u=
\nabla\cdot B=0.
\end{split}
\end{equation}

The \textit{kinematic} dynamo model is primarily useful for studying the onset of dynamo action during the early stages of a seed magnetic field, when the Lorentz force $J \times B$ is too weak to noticeably affect the fluid motion \cite[\textsection 1.2.4]{ChildressGilbert03}. Beyond this regime, this model is replaced by the \textit{non-linear} dynamo equation, which extends \eqref{dynamo1} to the full magnetohydrodynamic (MHD) system by incorporating the Cauchy momentum equation with the Lorentz force. In general, we will omit the word \textit{kinematic} in this work and refer to it simply as the \textit{dynamo equation} for brevity. We will also refer to the case $\epsilon=0$ as the \textit{ideal dynamo equation}. 

The central question in dynamo theory is whether there exist velocity fields $u$ whose associated dynamics can generate a large magnetic field $B$ from a given initial condition
\begin{equation}\label{dynamo:Bt=0}
B|_{t=0}=B^\circ.
\end{equation}
The dynamo problem has attracted considerable attention, not only in the physics literature (see e.g.~\cite{ADN02,BouyaDormy13}) but also in mathematics; see, for example, the book of Arnold–Khesin \cite{ArnoldKhesin98}, the treatments of Moffatt–Dormy \cite{MD19}, or Childress and Gilbert \cite{Gilbert03,ChildressGilbert03}.

In recent years, there has been a revival in the rigorous mathematical construction of kinematic dynamos, primarily in $\mathbb{R}^3$ and $\mathbb{T}^3$ (see the discussion below). 
However, the dynamo problem in more general domains, which is of both physical and mathematical interest, remains largely open. 
In this paper, we consider bounded domains and therefore introduce the appropriate boundary conditions, the corresponding function spaces, and the weak formulation of \eqref{dynamo1}.

\subsubsection{Dynamos on bounded domains}
We consider velocity and magnetic fields $u,B:[0,\infty)\times\Omega\to\mathbb{R}^3$ defined on a smooth bounded domain $\Omega\subset\mathbb{R}^3$, together with the corresponding boundary conditions.
First, we consider domains with \textit{impermeable} and \textit{magnetically closed walls}. 
That is, we assume the tangential field conditions
\begin{equation}\label{dynamo:tangentialcondition}
u\cdot n = B\cdot n = 0,
\end{equation}
where $n$ denotes the outward unit normal vector on the boundary $\partial\Omega$.
Second, for domains with \textit{perfectly conducting walls} ($E \times n = 0$), 
Ohm's law together with \eqref{dynamo:tangentialcondition} implies the vanishing of the tangential current, $J\times n=0$, that is,
\begin{equation}\label{dynamo:curlBxn}
(\nabla\times B)\times n = 0.
\end{equation}
From now on, we consider \eqref{dynamo1} together with the initial and boundary conditions
\eqref{dynamo:Bt=0}-\eqref{dynamo:curlBxn}.

The growth of $B$ is typically measured in terms of the \textit{magnetic energy}
$$
\mathcal{E}(B) := \frac{1}{2} \int_{\Omega} |B|^2\dif x.
$$
Using the vector calculus identity valid for solenoidal fields, $\nabla\times(u\times B)=B\cdot\nabla u-u\cdot\nabla B$, a straightforward integration by parts shows that classical solutions of \eqref{dynamo1} satisfy
\begin{equation}\label{eq:energyestimate}
\frac{\dif}{\dif t}\mathcal{E}(B)
=\int_\Omega (B\cdot\nabla u)\cdot B\dif x
-\epsilon\int_\Omega|\nabla\times B|^2\dif x
+\epsilon\int_{\partial\Omega}
(B\times n)\cdot\nabla\times B\dif s.
\end{equation}
The first two integrals in \eqref{eq:energyestimate} correspond to \emph{stretching} and \emph{diffusion}, respectively.  
These effects compete: if stretching dominates, the magnetic energy may grow (dynamo action), whereas if diffusion prevails, the field decays. 

The last integral in \eqref{eq:energyestimate} represents a \emph{boundary flux}, which vanishes whenever the vectors $B$, $\nabla\times B$, and $n$ are linearly dependent on $\partial\Omega$. In particular, this occurs for the perfectly conducting walls introduced above, since condition \eqref{dynamo:curlBxn} implies that $\nabla\times B \parallel n$. Other examples of magnetic fields with vanishing boundary flux include \textit{normal fields} and magnetic fields associated with the Taylor conjecture (\textit{Beltrami fields}), which will be discussed later. This issue is discussed in detail throughout the paper, especially in Section \ref{sec:generalizations}.

We now formulate the dynamo problem in mathematical terms. 
It is well known (see e.g.~\cite[Lemma 3.1]{CSVpp}) that if the velocity $u$ belongs to the Beale–Kato–Majda class $L^1_t W^{1,\infty}$, then there exists a unique Leray–Hopf type solution 
$B^\epsilon \in L^\infty_t L^2 \cap L^2_t H^1$ 
to \eqref{dynamo1}.
Moreover, applying Gr\"onwall's inequality to \eqref{eq:energyestimate}, the solution satisfies the uniform-in-$\epsilon$ energy bound
\begin{equation} \label{eq:Energyestimate}
\mathcal{E}(B^\epsilon(t)) \leq \mathcal{E}(B^\circ) \exp \left( \int_0^t \|\nabla u(s)\|_{L^\infty} \,\mathrm{d}s \right).
\end{equation}
In particular, within these regularity classes, exponential growth is the optimal rate for the dynamo. 
This observation motivates the following definition of a \textit{dynamo}, commonly used in the literature (see e.g.~\cite{ChildressGilbert03}). 
Below, we denote by $L_\sigma^2(\Omega)$ the space of weakly divergence-free vector fields in $L^2(\Omega)$ satisfying the tangential boundary condition \eqref{dynamo:tangentialcondition}. We also consider the Sobolev spaces $H_\sigma^k(\Omega) := H^k(\Omega) \cap L_\sigma^2(\Omega)$.

\begin{defi}[Dynamos]\label{def:dynamo}
Given $u:[0,\infty)\to L_\sigma^2(\Omega)$ and $\epsilon \geq 0$, we define
\begin{equation} \label{e:gamma}
\gamma(\epsilon) := \sup_{B^\circ\in L^2_\sigma(\Omega)} 
\limsup_{t \to \infty} \frac{1}{t}\log \mathcal{E}(B^\epsilon(t)),
\end{equation}
where the supremum is taken over all initial data $B^\circ\in L^2_\sigma(\Omega)$ and the corresponding solutions $B^\epsilon$ of \eqref{dynamo1}. 

\noindent The velocity field $u$ is called a \emph{(kinematic) dynamo} if $\gamma(\epsilon) > 0$ for some $\epsilon > 0$. 
In this case, $u$ is called a \emph{fast dynamo} if $\liminf_{\epsilon \to 0} \gamma(\epsilon) > 0$, and a \emph{slow dynamo} otherwise. 
Likewise, $u$ is called an \emph{ideal dynamo} if $\gamma(0) > 0$. 
\end{defi}

The problem of finding kinematic and fast dynamos dates back to Zeldovich and Sakharov in the 1970s and appears as one of Arnold’s problems \cite[Problem 1994–28]{Arn04}. 
Celebrated antidynamo theorems \cite{Zeldovich57,ZMRS84} show that the velocity field must exhibit a certain degree of complexity. 
For various examples of ideal dynamos in the applied mathematics literature, see e.g.~\cite{ChildressGilbert03}. 
Classical examples of kinematic dynamos include linear flows $u(x,t)=A(t)x$ with $\operatorname{tr} A(t)=0$ \cite{ZMRS84}, the Ponomarenko dynamo \cite{Ponomarenko73}, and the Soward dynamo \cite{Soward87}. 
Related interesting results in the mathematical literature include, for instance, the work of Vishik and Friedlander \cite{VishikFriedlander93} and, more recently, that of Drivas, Elgindi, Iyer, and Jeong \cite{DEIJ22}.

In recent years, renewed interest has emerged in the rigorous mathematical construction of kinematic dynamos. 
Coti Zelati and Navarro-Fern\'andez constructed (universal) ideal dynamos on $\mathbb{T}^3$ \cite{CZNF}, and Coti Zelati, Sorella, and Villringer constructed stationary $W^{1,\infty}$-regular fast dynamos on $\mathbb{R}^3$ \cite{CSVpp}
(see also the recent work \cite{CSVpp2} for the pulsed
kinematic dynamo equation in $\T^3$). 
In fact, the latter authors proved a stronger result, replacing $\limsup_{t \to \infty}$ with $\inf_{t>0}$ in \eqref{e:gamma}. 
In contrast, the linear flows $u(x,t)=A(t)x$ are unbounded, while the Ponomarenko and Soward dynamos are discontinuous. 
These works make elegant use of the $\alpha$-effect, and dynamo action occurs at very large scales.
Later, Rowan constructed \emph{subsequentially fast dynamos}
$u \in W^{1,\infty}([0,\infty)\times \mathbb{T}^3)$,
where the condition $\liminf_{\epsilon \to 0} \gamma(\epsilon) > 0$ is relaxed to $\limsup_{\epsilon \to 0} \gamma(\epsilon) > 0$ \cite{Row25}. 
Building on \cite{Gilbert88}, a rigorous mathematical construction of a smooth Ponomarenko dynamo was finally obtained in \cite{NavarroFernandezVillinger25} by Navarro-Fern\'andez and Villringer. 
In that work, the excited frequencies are of order $\epsilon^{1/3}$, corresponding to slow dynamo action, as predicted by Gilbert.

Another recent development is due to Sorella and Villringer \cite{SorellaVillingerpp}, who constructed a fast dynamo $u \in L^\infty([0,\infty); W^{1,\infty}(\mathbb{T}^3))$. In their paper, such a dynamo is referred to as a \emph{limsup fast dynamo}, and it is noted that Arnold’s original problem \cite{Arn04,ArnoldKhesin98} requires a \emph{liminf} in the definition of $\gamma(\epsilon)$.

In this paper we construct velocity fields for which the corresponding magnetic energy grows (super)exponentially, except on a set of times of asymptotic density zero. 
We refer to this as an \emph{almost sure fast dynamo}. 
Here ``almost sure'' is understood in the asymptotic sense $t\to\infty$, meaning that the proportion of times without dynamo action becomes negligible (see Figure~\ref{fig:energy}).

\begin{defi}[Almost sure dynamos]
We say that $u:[0,\infty)\to L_\sigma^2(\Omega)$ is an \emph{almost sure (a.s.)~dynamo} if there exists $\gamma>0$ such that
$$
\sup_{B^\circ\in L^2_\sigma(\Omega)}
\liminf_{T\to\infty}\frac{1}{T}
|\{t\in(0,T]\,:\,\frac{1}{t}\log\mathcal{E}(B^\epsilon(t))\geq\gamma\}|
=1.
$$
We denote by $\bar{\gamma}(\epsilon)$ the supremum of all such $\gamma$. 
Likewise, we speak of slow, fast, and ideal a.s. dynamos depending on the values of $\bar{\gamma}(\epsilon)$ and $\liminf_{\epsilon \to 0} \bar{\gamma}(\epsilon)$.
\end{defi}

The almost sure property implies exponential growth in any variant of the Ces\`aro sense.
That is, if we adopt the viewpoint of ergodic theory and define
$$
\gamma^{\textup{avg}}(\epsilon) := 
\sup_{B^\circ\in L^2_\sigma(\Omega)} 
\liminf_{T \to \infty} 
\frac{1}{T}\int_1^T \frac{1}{t} \log^+ \mathcal{E}(B^\epsilon(t)) \,\mathrm{d}t,
$$
then one obtains the corresponding notions of slow, fast and ideal dynamos. 
It is straightforward to check that $\gamma^{\textup{avg}}(\epsilon) \ge \bar{\gamma}(\epsilon)$.

\subsubsection{Fast dynamos on bounded domains}

The first contribution of this paper is to show that, if the flow is allowed to be irregular, then on any smooth bounded domain one can construct a.s.~dynamos whose magnetic energy grows arbitrarily fast in time.
More precisely, such dynamos exist if one replaces $\infty$ by any $p<\infty$ in 
$(u,B^\epsilon)\in L_t^1 W^{1,\infty}\times L_t^\infty L^2$. 
For irregular vector fields, the definitions require additional comments. 
First, when working with weak solutions $(u,B^\epsilon)$—for which uniqueness need not hold—we allow the supremum in \eqref{e:gamma} to be taken over all possible solutions. 
Second, since $\mathcal{E}(B^\epsilon)$ is only integrable in time, the $\limsup$ in \eqref{e:gamma} is taken over its Lebesgue points. 
Finally, we require that the initial data be attained in a stronger sense than for general weak solutions. 
Namely, $t=0$ is a Lebesgue–Bochner point of $B^\epsilon$, that is,
$$
\lim_{t\to 0}\dashint_0^t\|B^\epsilon(s)-B^\circ\|_{L^2}\dif s=0,
$$
where $\dashint\equiv$ average integral. Indeed, a stronger technical property holds (see Remark \ref{rem:fastdynamo}).

\begin{thm}[Fast dynamos on bounded domains]\label{thm:Main}
Let $\Omega \subset \R^3$ be a smooth bounded domain. Then, there exists an a.s.~fast dynamo  
$u \in L_t^1W^{1,p}(\Omega)$ for all $p<\infty$.
Furthermore, the corresponding magnetic fields $B^\epsilon$ are spatially smooth up to a null set of times, namely $B^\epsilon\in L_t^pC^m(\Omega)$ for all $p<\infty$ and $m\in\N$.
\end{thm}
\begin{proof}
See Section~\ref{Proof:thm:fastdynamo}, and in particular Subsection~\ref{sec:proof:thm:main}.
\end{proof}

In view of \eqref{eq:Energyestimate}, the result is almost sharp in terms of regularity. It is still open  whether superexponential growth can still be proved, or excluded, when relaxing the $\infty$ only in $u$ or $B^\epsilon$. Note that the solutions $B^\epsilon$ that we construct undergo superexponential growth, but are spatially much more regular than $L^2_\sigma$. This additional smoothness allows us to incorporate the boundary conditions
into the relevant function spaces on bounded domains in a pointwise sense and allows us to describe precisely what weak solutions mean in this context, which we do after Theorem~\ref{thm:turbulent}.

\subsubsection{Turbulent dynamos on bounded domains}

Next, we refine Theorem~\ref{thm:Main} to produce a genuine \emph{turbulent} dynamo. In the applied mathematics literature, so-called turbulent dynamos are frequently discussed, and it is plausible that astrophysical dynamos are indeed turbulent. The precise meaning of ``turbulent'' in this context requires clarification; nevertheless, following Tobias \cite{Tobias21}, we may loosely define a turbulent dynamo as one in which the velocity field is highly irregular and the dynamo action is ``multiscale'', that is, present at both very large and very small scales. Mathematically, this corresponds to rapid growth at both low and high frequencies as time evolves. See \cite[\textsection 4]{Tobias21} for a detailed description of the various turbulent regimes in terms of the magnetic Prandtl number (and, in nonlinear dynamo theory, the Reynolds number), as well as the potential advantages of such flexibility for realizing dynamos in the laboratory. See also \cite[p.~4]{Row25} for a discussion of the mathematical challenges involved in constructing solutions whose energy modes grow simultaneously at large and small scales, uniformly in the vanishing-diffusivity limit $\epsilon\to 0$.
 
In the full space $\R^3$ or on the torus $\T^3$, scales are naturally expressed in terms of Fourier frequencies. 
In a bounded domain, however, it is more natural to formulate this notion in terms of the eigenfrequencies of $\Omega$ associated with electromagnetic waves. 
For magnetic fields, the natural frequencies are the square roots of the eigenvalues $\lambda_k$ of the \textit{magnetostatic operator}, which admits an orthonormal basis of eigenfields $B_k$. 
We briefly review this theory in Appendix \ref{sec:magnetostatic}.

\begin{figure}[htbp]
    \centering
    \includegraphics[width=0.65\textwidth]{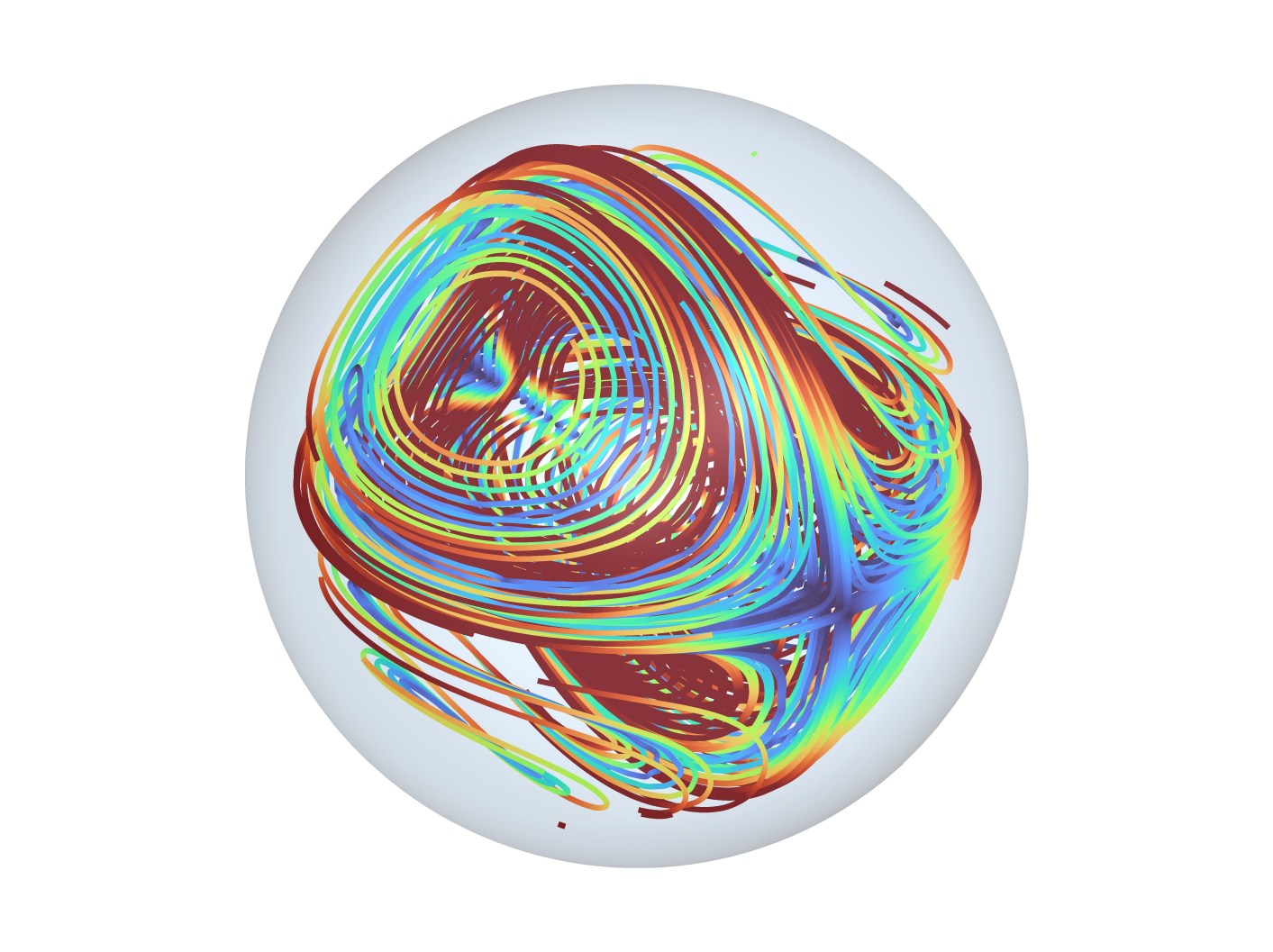}
    \caption{Streamlines of a magnetostatic eigenfield $B_k$ in the unit ball $\Omega$.}
    \label{fig:streamlines_B2}
\end{figure}

In line with this viewpoint, we define the \textit{magnetic energy modes}
$$
\mathcal{E}_k(B):=\frac{1}{2}\langle B,B_k\rangle_{L^2}^2.$$
By Parseval’s identity, the total magnetic energy $\mathcal{E}$ is the sum of all $\mathcal{E}_k$.
In this way, the magnetic energy modes encode the dynamo scales.

\begin{defi}[Turbulent dynamos]
Given $u:[0,\infty)\to L_\sigma^2(\Omega)$, $\epsilon \geq 0$ and $k\in\N$, we define
\begin{equation} \label{e:gamma:k}
\gamma_k(\epsilon) := \sup_{B^\circ\in L^2_\sigma(\Omega)} \limsup_{t \to \infty} \frac{\log \mathcal{E}_k(B^\epsilon(t))}{t}.
\end{equation}
We say that $u$ \emph{exhibits dynamo action at mode $k$} if $\gamma_k(\epsilon)>0$ for some $\epsilon > 0$.
In particular, we say that $u$ is a \emph{turbulent dynamo} if $\gamma_k(\epsilon)>0$ for infinitely many modes $k$ with the same $\epsilon > 0$. These definitions extend naturally to slow, fast, ideal, and a.s.~dynamos.
\end{defi}

In the following theorem, we show that the set of modes $K$ in which dynamo action occurs in Theorem \ref{thm:Main} can be prescribed. Choosing $K$ to consist of low frequencies yields \emph{large-scale} dynamos, while choosing $K$ to consist of high frequencies yields \emph{small-scale} dynamos. Finally, combining low and high frequencies yields \emph{multiscale}, and hence turbulent, dynamos.

\begin{thm}[Turbulent dynamos on bounded domains]\label{thm:turbulent} 
Let $K\subset\mathbb{N}$.
On any smooth bounded domain $\Omega \subset \mathbb{R}^3$, the a.s.~fast dynamos in Theorem \ref{thm:Main} can be constructed to exhibit dynamo action precisely at the modes in $K$. In particular, there exist a.s.~turbulent dynamos on $\Omega$.
\end{thm}
\begin{proof}
See Section~\ref{sec:Turbulent coarse-grained magnetic fields}, and in particular Subsection~\ref{sec:proof:thm:turbulent}.
\end{proof}

As promised, now we set the notation for function spaces and weak solutions.

\begin{defi}
We define the space corresponding to the perfectly conducting walls:
$$
H_{\mathrm{pc}}(\Omega):=\{B\in H_\sigma^2(\Omega)\,:\,(\nabla\times B)\times n=0\}.
$$
When $1 \leq p \leq \infty$, we define $B \in L^p(0,T;H_{\mathrm{pc}}(\Omega))$ to be a \emph{weak solution} if $B \times u \in L^1_{t,x}$ and
\[\int_0^T \int_\Omega (B \cdot \partial_t \varphi - (B \times u + \epsilon  \nabla \times B) \cdot \nabla \times \varphi) \dif x \dif t + \int_\Omega B^0 \cdot \varphi(0) \dif x = 0\]
for every $\varphi \in C^\infty([0,T]\times\bar{\Omega})$.
\end{defi}

\subsection{Convex integration}
Theorems~\ref{thm:Main} and \ref{thm:turbulent} will be obtained as a corollary of an \textit{$H$-principle}. 
Since the pioneering work of De Lellis and Sz\'ekelyhidi \cite{DeLellisSzekelyhidi09}, convex integration has become a powerful framework for constructing weak solutions to the equations
of fluid dynamics that replicate features of turbulent flows. 
By now, a wide variety of convex integration schemes have been developed, some of them genuine pieces of mathematical craftsmanship. 
See, for example, works modelling instabilities \cite{CCF21, CFM19, CFM22, MengualSzekelyhidi23, GKS21, ForsterSzekelyhidi18}, and the celebrated works on the Onsager conjecture and K41 theory \cite{GR23, GKN23, DeLellisSzekelyhidi09, Isett18, BDSV19, BCKpp2, CheskidovLuo22, BuckmasterVicol19, DeLellisSzekelyhidi13}.
In the context of MHD and plasma relaxation see \cite{BLL15, BBV20, FLSz24, LZZpp, MiaoYe24, MYYpp, EPPpp, CZZpp, GiardiSzekelyhidipp}, and for transport equations see \cite{CheskidovLuo24, ModenaSattig20, ModenaSzekelyhidi18, BCKpp}.

Often, convex integration is viewed merely as a mechanism to prove nonuniqueness, and the full strength of the $H$-principle is overlooked. 
This terminology, borrowed from differential geometry, refers to the fact that wild solutions constructed via convex integration often resemble, at large scales, a macroscopic or coarse-grained evolution $(\bar{u},\bar{B})$. 
We now state the $H$-principle in the tradition of convex integration 
\cite{DeLellisSzekelyhidi10, ChiodaroliFeireislKreml2015, DeLellisSzekelyhidi17, DaneriSzekelyhidi17, GKS21, GebhardKolumban22a, Markfelder21, Mengual22, SattigSzekelyhidi2023, FazekasKolumban2025}.

\begin{thm}[H-principle]\label{thm:hprinciple}
Let $\Omega\subset\mathbb{R}^3$ be a smooth bounded domain, and let $\bar{u}$ and $\bar{B}$ be smooth velocity and magnetic fields on $[0,\infty)\times\bar{\Omega}$ satisfying the boundary conditions  \eqref{dynamo:tangentialcondition}-\eqref{dynamo:curlBxn}.
Let $I_j$ be a partition of $(0,\infty)$ into bounded intervals, and consider three sequences: $p_j,m_j\geq 1$ and $\delta_j>0$. 
Then there exists a velocity field 
\begin{equation}\label{thm:hprinciple:u}
u\in \bigcap_{p<\infty}L_t^1 W^{1,p}(\Omega)
\end{equation}
such that, for any magnetic diffusivity $0<\epsilon\leq 1$, there exists a weak solution
\begin{equation}\label{thm:hprinciple:B}
B^\epsilon\in\bigcap_{\substack{p<\infty \\ m\in\N}}L_t^p C^m(\Omega)
\end{equation}
of the \eqref{dynamo1} equation in $[0,\infty)\times\Omega$.
Moreover,  
\begin{equation}\label{thm:hprinciple:barB}
 \dashint_{I_j}\|(u-\bar{u})(t)\|_{W^{1,p_j}}\dif t
 +
\left(\dashint_{I_j}\|(B^\epsilon-\bar{B})(t)\|_{C^{m_j}}^{p_j}\dif t\right)^{\nicefrac{1}{p_j}}
\leq\delta_j,
\end{equation}
for all $j$, uniformly in $0<\epsilon\leq 1$.
\end{thm}
\begin{proof}
See Section \ref{Proof:thm:hprinciple}.
\end{proof}

\begin{figure}[htbp]
\centering
\includegraphics[width=0.30\linewidth]{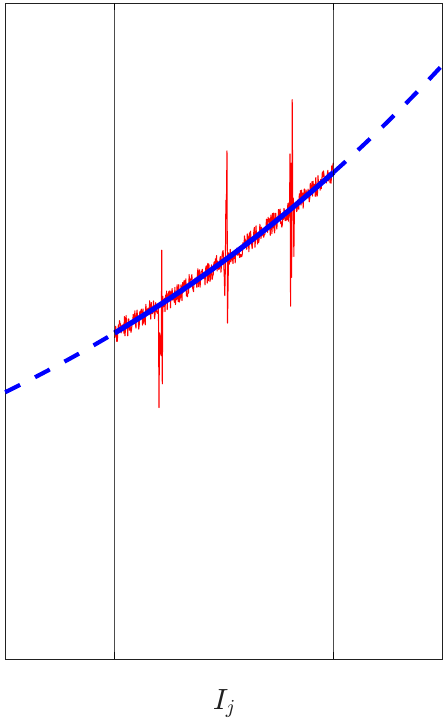}
        \caption{Cartoon of the ``intermittent'' magnetic energy $\mathcal{E}(B^\epsilon(t))$ (red), which remains close to the macroscopic magnetic energy $\mathcal{E}(\bar{B}(t))$ (blue) except on a set of times that can be taken arbitrarily small as the time interval $I_j$ tends to infinity.}
        \label{fig:energy}
\end{figure}

After the $H$-principle, the existence of a fast and turbulent dynamo only requires prescribing a coarse-grained evolution with the desired energy growth and an appropriate choice of the aforementioned sequences, which we do in Theorems \ref{thm:fastdynamo} and \ref{thm:turbulentdynamo}, respectively.

The convex integration scheme presented here has the advantage of being comparatively simple, and we believe it offers an accessible introduction to intermittent convex integration. Readers primarily interested in this aspect, the second goal of the paper, are referred to Sections~\ref{sec:localized}--\ref{sec:idealdynamos}.

\begin{Rem}\label{rem:Times} We record some remarks concerning Theorem \ref{thm:hprinciple}.
\begin{enumerate}[(i)]
    \item The solutions exist globally in time, but their time integrability is only local, namely $u \in L^1([0,T], W^{1,p})$ and $B^\epsilon \in L^p([0,T], C^m)$ for all $T > 0$, $p < \infty$, and $m \in \mathbb{N}$. In Section \ref{sec:idealdynamos}, we present variants that are uniformly bounded in time in the ideal case, by bargaining with spatial regularity (see Theorems \ref{Thm:lebesgue} and \ref{Thm:Sobolev}).
\item The flexibility in choosing the partition $I_j$ is crucial both near $0$ and $\infty$, as well as for optimizing the set of times exhibiting dynamo action in Theorem~\ref{thm:Main}.
By taking $p_j,m_j \nearrow \infty$, we obtain better control of the time integrability and spatial regularity as $t \to 0$ and $t \to \infty$. Nevertheless, one may also consider the constant case $p_j=p$ and $m_j=m$, since it still contains relevant information and is easier to understand on a first reading.
\end{enumerate}
\end{Rem}

The next remark explains why we can get approximate continuity at $t=0$.
\begin{Rem}[Attainment of the initial data]\label{rem:fastdynamo}
By adding and subtracting $B^\circ:=\bar{B}(0)$ in  Theorem \ref{thm:fastdynamo}\ref{thm:fastdynamo:(a)}, the triangle inequality yields
    \begin{equation}\label{rem:fastdynamo:(a)}
    \left(\dashint_0^{a_j}
    \|B^\epsilon(t)-B^\circ\|_{C^{m_j}}^{p_j} \, dt 
    \right)^{\nicefrac{1}{p_j}}\leq\delta_j
    +a_j\|\bar{B}\|_{C^1([0,a_j],C^{m_j})}.
    \end{equation}
    Since $\bar{B}$ depends only on $\Omega$ and $\bar{\mathcal{E}}=\mathcal{E}(\bar{B})$, one can choose $a_j$ in advance so that \eqref{rem:fastdynamo:(a)} tends to zero as $j\to\infty$. Indeed, by refining the partition if necessary, it is possible to replace $a_j$ in the averaged integral by any $t\in[a_{j+1},a_j]$. This shows that $B^\epsilon$ is approximately continuous at $t=0$ in the Lebesgue–Bochner sense, with $B^\epsilon(0)=B^\circ$ independently of $0<\epsilon\leq 1$.
\end{Rem}

\subsubsection{Generalizations}\label{sec:generalizations}
In this section we explain some variants  which might be of interest.\\

%\smallskip

\paragraph{\bf Domains and boundary conditions}

Theorems \ref{thm:hprinciple}, \ref{thm:fastdynamo}, and \ref{thm:turbulentdynamo} also apply in the periodic case $\Omega=\mathbb{T}^3$, where no boundary conditions are required. They also apply to $\Omega=\mathbb{R}^3$, provided one imposes suitable decay at infinity. In fact, for Theorem \ref{thm:fastdynamo} one can construct a fast dynamo with spatially compact support. For Theorem \ref{thm:turbulentdynamo}, since there is no eigenbasis, we instead consider $\hat{B}_k(\xi)=\xi\times\hat{A}_k(\xi)$ where $\hat{A}_k$ is a bump function concentrated around mode $k$. The corresponding $B_k$ is then a Schwartz-class curl field.

As mentioned earlier, Theorems \ref{thm:hprinciple}, \ref{thm:fastdynamo}, and \ref{thm:turbulentdynamo} also apply if $H_{\text{pc}}(\Omega)$ is replaced by normal fields
$$
H_{\mathrm{nf}}(\Omega):=\{B\in H_\sigma^1(\Omega)\,:\,B\times n=0\}.
$$
Notice that  $H_{\mathrm{nf}}(\Omega)$ is simply the space of solenoidal fields in $H^{1}(\Omega)$ with zero trace.

More generally, since $B$ is obtained as the limit of a Cauchy sequence $(B_n)$ in $L_t^p C^m$, for any $p<\infty$ and $m\in\mathbb{N}$, with $B_n - \bar{B}$ supported in the interior of $(0,\infty)\times\Omega$, it follows that $B$ inherits the boundary conditions of $\bar{B}$.\\

\paragraph{\bf Fast and furious dynamos}
The flexibility of the $H$-principle allows the construction of a wide variety of dynamos at the level of weak solutions. Indeed, one can prescribe exponential growth (fast dynamo), or even superexponential growth. Notice that this does not contradict the inequality \eqref{eq:Energyestimate}, since the velocity lies below the Beale–Kato–Majda regularity class. Regularity is also why our results are not  in contradiction with  the explicit quantitative relaxation of the dynamo equation obtained in \cite{HitruhinLindberg24}. See also \cite{FazekasKolumban2025} for estimates on the hull in full MHD.\\

\paragraph{\bf Saturation of the dynamo}
As mentioned earlier in the introduction, in practice the kinematic dynamo typically saturates at a finite time $T$, as a consequence of the nontrivial Lorentz force acting on the fluid through the Navier–Stokes equation \cite{Tobias21}. For this reason, some authors (see e.g.~\cite{NunezSanz00}) replace Definition \ref{def:dynamo} by taking the limit as $t \to 0$ in \eqref{e:gamma}, instead of as $t \to \infty$, namely
    $$
    \tilde{\gamma}(\epsilon) := \sup_{B^\circ\in L^2_\sigma(\Omega) \setminus \{0\}} \limsup_{t \to 0} \frac{1}{t}\log\left(\frac{\mathcal{E}(B^\epsilon(t))}{\mathcal{E}(B^\circ)}\right).
    $$
    Although we have stated our results according to the standard definition of the dynamo, it follows analogously that one may, if desired, adjust the statements for $t \to 0$.\\

\paragraph{\bf Taylor's conjecture}\label{sec:Taylorconjecture}
In the context of plasma relaxation and the Taylor conjecture~\cite{Tay74,Ber84,FL20,FLMV21}, which concerns the long-time behavior of solutions to MHD, we can also prescribe the long-time behavior of the weak magnetic field. For example, there exist weak solutions to \eqref{dynamo1} that converge to force-free (Beltrami) fields, as predicted by Taylor's relaxation.

The method is flexible and also allows one to prescribe convergence to other, possibly nonphysical, magnetic configurations. Moreover, our weak solutions lie in a regime where magnetic helicity does not need to be conserved in the ideal case. The main point is that the construction of dynamos can be made compatible with a natural long-time behavior of the magnetic field.

More precisely, we show in Theorem~\ref{thm:Beltrami} that for any smooth solenoidal field $\bar{u}$, any initial datum $B^\circ$, and any Beltrami field $\mathscr{B}$, there exists another velocity field $u$, arbitrarily close to $\bar{u}$ in $L_t^1 W^{1,p}$, such that the associated \eqref{dynamo1} equation admits a solution starting from $B^\circ$ that converges to $\mathscr{B}$ in the long-time limit
(and such that the boundary flux vanishes at almost all times). We prove this result in Section~\ref{sec:Taylor}.

\subsection{Geometric transport equation}
The third aim of this paper is to present dynamo theory in its appropriate geometric context, namely the geometric transport equation. 
This equation was considered in \cite{Kampschulte} in the setting of gradient flows, and studied extensively in \cite{BDR-TAMS}, to which we refer for the basic theory and notation (see also Section~\ref{sec:GTE}). It represents an evolution equation for a class of generalized surfaces, that can be used for instance to describe the movement of dislocations in materials, see e.g. \cite{Bonicatto-Rindler} which is based on the space-time approach introduced in \cite{Rindler-Space-time}. This equation was also used in \cite{Bonicatto-Frobenius} to derive a link with Frobenius's theorem on the commutativity of flows. 
Our results are complemented by nonuniqueness theorems in the spirit of convex integration, and they couple sharply with the recent well-posedness result of \cite{BDR-JFA, BonicattoDelNinpp}.

The \textit{Geometric Transport Equation} (GTE) 
\begin{equation}\label{eq:GTE}
\partial_t T+\mathcal{L}_u T=0,
\end{equation}
describes the transport by a given velocity field $u$ of a \textit{$k$-current} $T$, representing a generalized surface in $\R^d$ that may be rough and spread out. More precisely, for any $k=0,\ldots,d$, the space of $k$-currents is the dual of the space of $k$-forms. 
The Lie derivative $\mathcal{L}_u T$ can be expressed in terms of the \textit{boundary operator} $\partial$ (the dual of the exterior derivative) and the wedge product $\wedge$ through the Cartan formula
\[
\mathcal{L}_u T
=-\partial(u\wedge T)-u\wedge\partial T.
\]

In this paper, we focus on \textit{boundaryless} currents, i.e.~$\partial T=0$, for which $\mathcal{L}_u T=-\partial(u\wedge T)$. We further assume that they are absolutely continuous with respect to the $d$-dimensional Lebesgue measure.
Assuming in addition that the flow is incompressible, the GTE takes the form
\begin{equation}\label{GTE}\tag{GTE}
\begin{split}
\partial_t T=\partial(u\wedge T),\\
\nabla\cdot u=0,\,\partial T=0.
\end{split}
\end{equation}

A feature of this equation is that it reduces to well-known PDEs for specific values of $k$. Here we mention only its connection to the dynamo and transport equations, as well as the main results concerning uniqueness for the associated initial value problem. For futher details on the derivation of the specific cases from the GTE, see \cite[\textsection 3.4]{BDR-TAMS}.

\subsubsection{The extreme cases}
When $k=0$ and $k=d$, the GTE reduces to the continuity equation and the transport equation, respectively. 
For divergence-free velocity fields, the two systems are formally equivalent:
\begin{equation}\label{transport}\tag{transport}
\begin{split}
\partial_t \rho+\nabla\cdot(\rho u) &=0,\\
\nabla \cdot u &= 0,
\end{split}
\end{equation}
Above, we have identified $T$ with the density $\rho$. 

\subsubsection{The case $k=1$}
Of particular relevance to the present paper is the case of $1$-currents $T$, which can be identified with vector fields, represented here by the magnetic field $B$.
When $k=1$, writing the explicit coordinate form of equation \eqref{eq:GTE} in $\R^d$, one obtains the following expression
\[
\partial_t B+\dive(u\otimes B-B\otimes u)+u\,\nabla\cdot B=0,
\]
where the divergence is taken row-wise. Assuming that $\nabla\cdot B=0$ (this corresponds to the case of boundaryless currents), and specializing to $d=3$, the equation reduces to
\[
\partial_t B=\nabla\times (u\times B).
\]
From this point of view, the ideal \eqref{dynamo1} equation can be regarded as the three-dimensional case of the GTE for boundaryless $1$-currents.

\subsubsection{Nonuniqueness for GTE}
Uniqueness for the initial value problem
\begin{align*}
    \partial_t T+\mathcal{L}_u T&=0,\\
    T|_{t=0}&=T^\circ,
\end{align*}
was shown to hold in \cite{BDR-JFA} in the class of \textit{normal currents} if the vector field $u$ is Lipschitz and autonomous (i.e., time-independent). More recently, in \cite{BonicattoDelNinpp} the uniqueness was extended also to non-autonomous vector fields that belong to the class $L^1_tW^{1,\infty}$. We point out that in these results the currents can be singular (analogous to considering the continuity equation for measures instead of functions), and the vector field $u$ does not have to satisfy any further constraint on the divergence.
If one requires the currents to be absolutely continuous, then it seems that uniqueness results could be pushed further than Lipschitz vectorfields, similar to what happens in the DiPerna-Lions theory for the transport equation, where one has uniqueness for $W^{1,p}$ vector fields and $L^q$ densities if $p$ and $q$ are suitably related. This is part of a forthcoming separate ongoing research project. We also remark that in \cite{Cortopassipp} the author proposes an approach to the classical DiPerna-Lions uniqueness that is based on the theory of currents.

In this work, we show that the convex integration schemes of \cite{ModenaSzekelyhidi18,CheskidovLuo21} do extend to treat general currents
proving nonuniqueness of weak solutions to \eqref{GTE} with certain Sobolev regularity and compact spatial support, in the class of boundaryless currents. 
We state our results for the case $0\le k\le d-2$.
We remark that the case $k=d$ is trivial, since the only boundaryless $d$-currents in $\R^d$ are constants. The case $k=d-1$ instead presents some difficulties that are discussed in Section~\ref{sec:k=d-1}, and will be left to future investigations. 

\begin{thm}[Nonuniqueness for GTE]\label{thm:GTE}
Let $0\le k\le d-2$ and let $\Omega\subset\R^d$ be an open domain. Let $1<p,q<\infty$ be a pair of conjugate H\"older exponents 
$$
\frac{1}{p}+\frac{1}{q}=1.
$$
Consider $r,s\geq 0$ and $1<\tilde{p},\tilde{q}<\infty$ in the regime
$$
\frac{r}{d-(k+1)}
+\frac{1}{p}
<\frac{1}{\tilde{p}},
\quad\quad
\frac{s}{d-(k+1)}
+\frac{1}{q}
<\frac{1}{\tilde{q}}.
$$
Then there exists a non-trivial weak solution to the \eqref{GTE} in the class
$$
u\in L_t^\infty (L^p\cap W^{r,\tilde{p}}),\quad\quad
T\in L_t^\infty (L^q\cap W^{s,\tilde{q}}).
$$
Moreover, $(u,T)$ is supported inside $(0,1)\times\Omega$. 
\end{thm}
\begin{proof}
See Section~\ref{sec:GTE}, and in particular Subsection~\ref{sec:proof:thm:GTE}.
\end{proof}
The transport case $k=0$, $r=1$ is covered by \cite[Theorem 1.2]{ModenaSzekelyhidi18}, and thus, in this regime, the only novelty of Theorem~\ref{thm:GTE} is the compact support. In fact, for the transport equation, this regime has been further improved in \cite{ModenaSattig20,BCKpp}, and we wonder if a similar approach could improve the range also for the \eqref{GTE}.

Theorem~\ref{thm:GTE} includes the case of an ideal dynamo as a particular case. However, since some readers may be primarily interested in the dynamo equation, we find it more illustrative to begin with the case $k=1$ and $d=3$ in Section~\ref{sec:idealdynamos}. The reason why this case is simpler is that the corresponding intermittent shear flows introduced in Section~\ref{sec:localized} are one-dimensional, which facilitates the construction of potentials. A useful exercise for readers new to convex integration is to repeat this construction for the 2D \eqref{transport} equation.

For general $k$-currents, recall that the boundary operator $\partial$ generalizes both the divergence and curl operators. This observation allows us to give a unified expression for the potential for all $k$, which may be of independent interest. Its construction is more involved than in the dynamo case and is presented in Section~\ref{sec:potential}. 

\subsubsection{The Diffusion GTE}\label{sec:DGTE}
Next, we consider the diffusive version of the GTE
\begin{equation}\label{DGTE}\tag{DGTE}
\begin{split}
\partial_t T=\partial(u\wedge T)+\epsilon\Delta T,\\
\nabla\cdot u=0,\,\partial T=0,
\end{split}
\end{equation}
which we refer to as the \textit{Diffusion Geometric Transport Equation} (DGTE).

In adapting the ideas from \cite{ModenaSzekelyhidi18} to incorporate diffusion into the transport equation for $k$-dimensional currents, one is led to the decomposition for boundaryless currents $T$:
$$
\Delta T=\partial\partial^\dagger T,
$$
where $\partial^\dagger$ denotes the adjoint of $\partial$. Hence, the diffusion term can be incorporated into Theorem~\ref{thm:GTE}, provided that we can control the $L^1$-norm of $\partial^\dagger T$. This essentially requires that $T \in W^{1,1}$, which corresponds to the case $s=\tilde q=1$, and occurs in the regime
\begin{equation}\label{d>k+1+p}
0\leq k <d-1-p.
\end{equation}

\begin{thm}[Nonuniqueness for DGTE]\label{thm:DGTE}
In the setting of Theorem~\ref{thm:GTE}, if in addition \eqref{d>k+1+p} holds, then the pair $(u,T)$ can be constructed to satisfy the \eqref{DGTE}.
\end{thm}
\begin{proof}
See Section~\ref{sec:GTE}, and in particular Subsection~\ref{sec:proof:thm:GTE}.
\end{proof}

The case $k=0$ with $r=1$ in Theorem~\ref{thm:DGTE} recovers the regularity of the nonunique solutions obtained in \cite[Theorem 1.9]{ModenaSzekelyhidi18} for the diffusion–transport equation, while extending the result to arbitrary domains.

Unfortunately, the regime \eqref{d>k+1+p} does not cover, for example, the dynamo case $k=1$ and $d=3$, not even for $p=1$. Indeed, in order to control the magnetic energy, we need to take $p=2$. We circumvent this obstacle by incorporating the time intermittency introduced in \cite{CheskidovLuo24}. Moreover, we implement a scheme in which the velocity does not depend on $\epsilon$, as required in the definition of fast dynamos.

\begin{thm}[Time-Intermittent Nonuniqueness for DGTE]\label{thm:DGTE:intermittent}
Let $0\le k \le d-2$ and let $\Omega\subset\R^d$ be an open domain. 
Then there exists a non-trivial weak solution $(u,T)$ to the \eqref{DGTE} such that
\[
(u,T)\in \bigcap_{\substack{p<\infty \\ m\in\N}} L_t^1 W^{1,p}\times L_t^p C^m.
\]
Moreover, $(u,T)$ is supported inside $(0,1)\times\Omega$.
\end{thm}
\begin{proof}
See Section~\ref{sec:GTE}, and in particular Subsection~\ref{sec:proof:thm:DGTE}.
\end{proof}

The case $k=0$ recovers the regularity of the nonunique solutions to the transport equation obtained in \cite[Theorem 1.2]{CheskidovLuo24}, and extends that result to arbitrary domains. 

All the theorems for the \eqref{GTE} also admit a form of the $H$-principle in terms of a coarse-grained flow $(\bar{u}, \bar{T})$, as we show in Section~\ref{sec:GTE}. Moreover, for the last theorem for the \eqref{DGTE}, we can choose $\bar{u}$ to be independent of the diffusion parameter $\epsilon$. As noted in Section~\ref{sec:generalizations}, the solution is obtained as the limit of a Cauchy sequence that preserves the boundary conditions.\\

\textbf{Organization of the paper.} In Section \ref{sec:localized}, we sketch the localized convex integration scheme and introduce the main building blocks and potentials. In Section \ref{sec:idealdynamos}, we present the simplest version of the construction for ideal dynamos. In Section \ref{sec:fastdynamos}, we introduce time intermittency to address the case with diffusion. In Section \ref{sec:GTE}, we generalize the construction to the GTE. We conclude with some remarks in Section \ref{sec:conclusion}.

\section{Localized convex integration}\label{sec:localized}

This section is divided into two parts. In Section~\ref{sec:buildingblocks}, we first define the main building blocks and potentials. In Section~\ref{sec:scheme}, we introduce the localized convex integration scheme, which will be treated in detail in Sections~\ref{sec:idealdynamos} and~\ref{sec:fastdynamos}.

\subsection{Building blocks and potentials}\label{sec:buildingblocks}

Let us begin by introducing some notation involving three parameters ($\ell,\mu,\lambda$): the side length $\ell>0$ of the periodic domain, the concentration $\mu\geq \ell^{-1}$, and the oscillation $\lambda\in\mathbb{N}$.
\begin{itemize}
\item $\T_\ell^d := [0,\ell)^d$ denotes the $d$-dimensional flat torus of side length $\ell>0$.
	\item Given $f\in C_c^\infty(\R^d)$ with $\operatorname{supp}f\subset (0,1)^d$ and $\mu\geq\ell^{-1}$, we denote
	$$
	f_\mu(y):=f(\mu y),\quad y\in\R^d,
	$$
	which is supported on $(0,\ell)^d$.
	Let $g\in C^\infty(\T_\ell^d)$ be the $\ell$-periodic extension of $f_\mu$. 
    \item Given $g\in C^\infty(\T_\ell^d)$ and $\lambda\in\N$, we denote
	$$
	g_\lambda(y):=g(\lambda y),\quad y\in\T_\ell^d.
	$$
    We abuse the notation $g=f_\mu$ and so $g_\lambda=(f_\mu)_\lambda$.
\end{itemize}

A key insight of the intermittent variant of convex integration \cite{BuckmasterVicol19} is that adding concentration to localized oscillations yields improved estimates, as illustrated by the following lemma, which will be used repeatedly throughout Sections~\ref{sec:idealdynamos}--\ref{sec:GTE}.

\begin{lemma}\label{lemma:buildingblock}
For any $f\in C_c^\infty(\R^d)$  with $\operatorname{supp}f\subset (0,1)^d$, $m\geq 0$ and $1\leq r\leq\infty$,
we have
$$
\|((f_\mu)_\lambda)^{(m)}\|_{L^r}
=(\lambda\mu)^m\mu^{-\frac{d}{r}}\|f^{(m)}\|_{L^r}.
$$
\end{lemma}
\begin{proof}
The factor $(\lambda\mu)^m$ comes from the derivative, whereas the term $\mu^{-d/r}$ arises from the Jacobian of the change of variables $y'=\mu y$.
\end{proof}

Next, we fix a cutoff $\Phi\in C_c^\infty(\R)$ with $\operatorname{supp}\Phi\subset  (0,1)$ and satisfying
\begin{equation}\label{cutoff:mean1}
\int_{\R}\phi^2 \dif s = 1,
\quad\quad
\phi:=\dot{\Phi}.
\end{equation}
Given $\mu\geq \ell^{-1}$, we define the concentrated cutoffs $\phi_\mu(s)
:=\phi(\mu s)$ and $\Phi_\mu(s)
:=\Phi(\mu s)$, and we extend both functions  periodically into $\T_\ell$. By the mean property \eqref{cutoff:mean1}, we can define the smooth potential $\mathbf{H}\in C^\infty(\T_\ell)$:
\begin{equation}\label{def:potentialH}
\mathbf{H}(s):=\int_{0}^{s}(1-\ell\mu\phi_\mu^2)\dif s'.
\end{equation}
We do not make explicit the dependence of $\mathbf{H}$ on $\mu$ in order to alleviate the notation, since it does not affect the estimates; namely, $\mathbf{H}$ is uniformly bounded in $\mu$
\begin{equation}\label{eq:Hbounded}
\|\mathbf{H}\|_{L^\infty} \leq 2\ell.
\end{equation}

\begin{figure}[htbp]
    \centering
    \begin{subfigure}{0.45\textwidth}
        \centering
        \includegraphics[width=\linewidth]{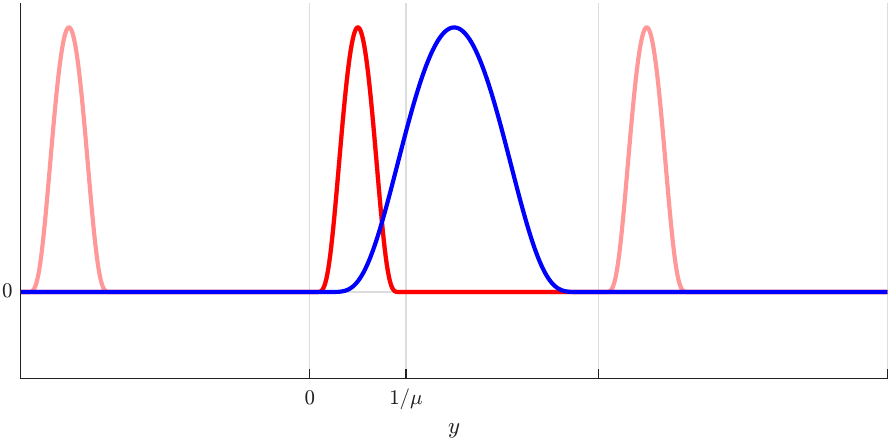}
        \caption{Concentration}
    \end{subfigure}
    \hfill
    \begin{subfigure}{0.45\textwidth}
        \centering
        \includegraphics[width=\linewidth]{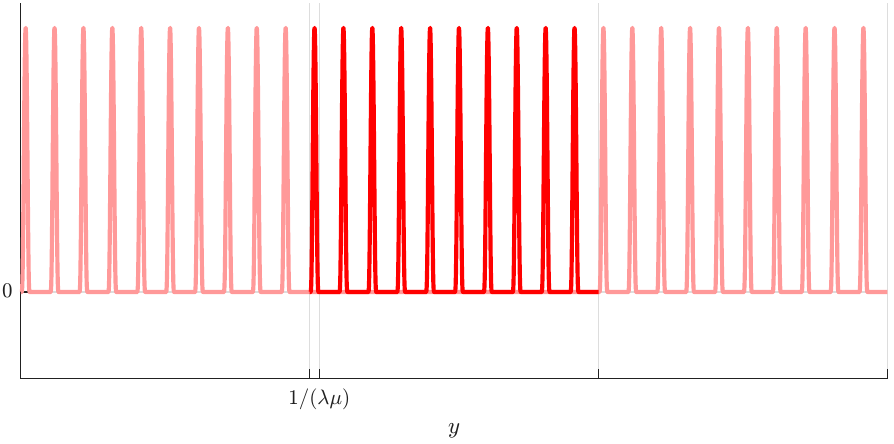}
        \caption{Oscillation}
    \end{subfigure}
    \caption{\textbf{Left:} The cutoff $\Phi$ (blue), its concentrated version $\Phi_\mu$ (red), and its periodic extension (light red). \textbf{Right:} The oscillatory version $(\Phi_\mu)_\lambda$. The combination of concentration and oscillation leads to intermittency.}
    \label{fig:mikado_flows}
\end{figure}

Next, we define the spatial and temporal building blocks in terms of $\phi_\mu$, $\Phi_\mu$, and $\mathbf{H}$ in Sections~\ref{sec:spatialintermittency} and~\ref{sec:temporalintermittency}, respectively.
We adopt the notation $(\ell,\mu,\lambda)$ for the spatial building blocks and $(\ell_0,\mu_0,\lambda_0)$ for the temporal ones. Unlike the concentration and oscillation parameters, the side lengths $\ell$ and $\ell_0$ play no essential role in the estimates and are introduced only to make the construction more explicit. For a first reading, one may simply set $\ell=\ell_0=1$.

\subsubsection{Spatial intermittency}\label{sec:spatialintermittency} 
In Section \ref{sec:idealdynamos}, we consider a smooth domain $\Omega\subset(0,\ell)^3$, 
a concentration parameter $\mu\geq\ell^{-1}$, an oscillation parameter $\lambda \in \N$, and a pair of conjugate H\"older exponents
$$
\frac{1}{p}+\frac{1}{q}=1.
$$
For any index $k=1,2,3$, we define the spatial-intermittent velocity and magnetic fields
\begin{subequations}\label{def:intermittentshearflows}
\begin{align}
	\mathbf{u}_k(x)
	&:=(\ell\mu)^{1/p}\phi_{\mu}(x_k)e_i,\\
	\mathbf{B}_k(x)
	&:=(\ell\mu)^{1/q}\phi_{\mu}(x_k)e_j,
\end{align}
\end{subequations}
where  $\{e_1,e_2,e_3\}$ is the standard basis of $\R^3$, and the indices $(i,j)$ are uniquely determined by $k$ through the identity
$$e_i\times e_j=e_k.$$
Notice that the only difference between $\mathbf{u}_{k}$ and $\mathbf{B}_{k}$ is the direction ($e_i$ and $e_j$) and the exponent of $\mu$ ($1/p$ and $1/q$). 
Similarly, we define the smooth velocity and magnetic potentials
\begin{subequations}\label{def:magneticvectorpotentials}
\begin{align}
{\boldsymbol{\psi}}_k(x)
&:=-(\ell\mu)^{1/p}\Phi_{\mu}(x_k) e_j,\\
\mathbf{A}_k(x)
&:=(\ell\mu)^{1/q}\Phi_{\mu}(x_k) e_i.
\end{align}
\end{subequations}
In addition, we define
$$
\mathbf{H}_k(x):=\mathbf{H}(x_k).
$$

Next, we derive some basic properties of these intermittent shear flows that will be crucial in the construction of the dynamos.

\begin{lemma}\label{Lemma:steady}
We have
$$\nabla\mathbf{H}_k=e_k-\mathbf{u}_k\times\mathbf{B}_k,$$ 
and also
\begin{subequations}\label{eq:magneticvectorpotentials}
\begin{align*}
\nabla\times\boldsymbol{\psi}_k&=\mu\mathbf{u}_k,\\
\nabla\times\mathbf{A}_k&=\mu\mathbf{B}_k.
\end{align*}
\end{subequations} 
\end{lemma}
\begin{proof} 
The claim follows directly from the definitions.
\end{proof}

As an immediate corollary, we deduce the following property.

\begin{prop}\label{prop:steadydynamo}
The pair $(\mathbf{u}_{k},\mathbf{B}_{k})$ is a steady solution to the ideal \eqref{dynamo1} equation in $\T_\ell^3$
\begin{subequations}
\begin{align}
\nabla\times(\mathbf{u}_{k}\times \mathbf{B}_{k})
&=0,\label{steady:1}\\
\nabla\cdot \mathbf{u}_{k}=
\nabla\cdot \mathbf{B}_{k}
&=0.\label{steady:2}
\end{align}
\end{subequations}
Moreover, it satisfies
$$
\dashint_{\T_\ell^3}\mathbf{u}_{k}\dif x =
\dashint_{\T_\ell^3}\mathbf{B}_{k}\dif x 
=0,\quad\quad
\dashint_{\T_\ell^3}\mathbf{u}_{k}\times \mathbf{B}_{k}\dif x
=e_k.
$$
\end{prop}

\subsubsection{Temporal intermittency}\label{sec:temporalintermittency} 
In Section \ref{sec:fastdynamos}, we consider a time interval $I=(0,\ell_0)$,
a concentration parameter $\mu_0\geq\ell_0^{-1}$, an oscillation parameter $\lambda_0 \in \N$, and a pair of conjugate H\"older exponents
\begin{equation}\label{eq:timeexponents}
\frac{1}{p_0}+\frac{1}{q_0}=1.
\end{equation} 
With them we define the time-intermittent velocity and magnetic coefficients
\begin{subequations}\label{def:timeintermittency:Bu}
\begin{align}
\mathbf{u}_0(t):=(\ell_0\mu_0)^{1/p_0}\phi_{\mu_0}(t),\\
\mathbf{B}_0(t):=(\ell_0\mu_0)^{1/q_0}\phi_{\mu_0}(t).
\end{align}
\end{subequations}
Similarly, we define 
\begin{equation}\label{def:timeintermittency:H0}
\mathbf{H}_0(t):=\int_{0}^{t}(1-\ell_0\mu_0\phi_{\mu_0}^2)\dif s',
\end{equation}
which satisfies
$
\mathbf{H}_0'
=1-\mathbf{u}_0\mathbf{B}_0.
$

\begin{Rem}\label{Rem:subindex0}
Notice that the spatial intermittent shear flows $(\mathbf{u}_k, \mathbf{B}_k)$ for $k = 1, 2, 3$ are vector fields on $\T_\ell^3$, whereas the temporal intermittent coefficients $(\mathbf{u}_0, \mathbf{B}_0)$ are scalar functions on $\T_{\ell_0}$.
For ease of notation, we omit the dependence on the parameters $(\ell,\mu,p)$ and $(\ell_0,\mu_0,p_0)$.
\end{Rem}

\subsection{Localized convex integration scheme}\label{sec:scheme}
In turbulent regimes, it is common to apply the \textit{Reynolds decomposition}, writing the flow $(u,B)$ as the sum of a coarse-grained part $(\bar{u},\bar{B})$ and a fluctuating part $(u',B')$:
$$
u=\bar{u}+u',
\quad\quad
B=\bar{B}+B'.
$$ 

On the one hand, the coarse-grained flow satisfies the \textit{mean-field dynamo equation} 
\begin{equation}\label{relaxeddynamo}\tag{MF-dynamo}
	\begin{split}
		(\partial_t -\epsilon\Delta)\bar{B} - \nabla\times(\bar{u}\times \bar{B})&=  \nabla\times R,\\
		\nabla\cdot\bar{B} =
		\nabla\cdot\bar{u}&=0,
	\end{split}
\end{equation}
where $R$ is the \textit{subscale electromotive force}. This term measures the deviation of the coarse-grained flow from being an exact solution, and we therefore also refer to $R$ as the \textit{defect}. 
Indeed, if $(\bar{E},\bar{B})$ is the coarse-grained electromagnetic field, then by \eqref{Faraday} law we have $\bar{E}=\epsilon\nabla\times\bar{B}-\overline{u\times B}$, and hence $R=\overline{u\times B}-\bar{u} \times \bar{B}$.
For a detailed discussion of \eqref{relaxeddynamo}, see e.g.~\cite{Eyi15,Kap25,KR80}.

On the other hand,
the fluctuation part must solve 
$$
\underbrace{\nabla\times(R-u'\times B')}_{\text{quadratic}}
+\underbrace{(\partial_t-\epsilon\Delta)B'
-\nabla\times(u'\times\bar{B}+\bar{u}\times B')}_{\text{linear}}
=0.
$$

A central problem in turbulence is the reconstruction of the full solutions $(u,B)$ from the coarse-grained fields $(\bar{u},\bar{B})$, that is, the recovery of the unresolved fluctuations $(u',B')$. One might attempt to determine $(u',B')$ either by canceling the quadratic error $u'\times B' = R$ or by solving the corresponding linear problem. However, given the complexity of turbulent dynamics, it is unrealistic to expect both conditions to hold at the first step. A more reasonable strategy is to choose $(u',B')$ so that the quadratic and linear errors are significantly smaller than the original one, with the hope of reducing the error to zero through an iterative procedure.

The strategy outlined above is precisely the one implemented by the convex integration method. This allows recovering weak solutions $(u,B)$ of the \eqref{dynamo1} equation as the limit, in a suitable sense, of a sequence $(u_n,B_n,R_n)$ solving the \eqref{relaxeddynamo} equation, starting from $(u_0,B_0,R_0)=(\bar{u},\bar{B},R)$. At each step, the defect $R_n$ is reduced by adding highly oscillatory and concentrated perturbations to $(u_n,B_n)$, carefully designed to diminish the error while preserving incompressibility. Iterating this procedure and driving $R_n \to 0$ yields a weak solution $(u,B)$ of \eqref{dynamo1} 
$$
u=\underbrace{\underbrace{u_0+u_0'}_{u_1}
+u_1'}_{u_2}
+u_2'+\cdots
\quad\quad
B=\underbrace{\underbrace{B_0+B_0'}_{B_1}
+B_1'}_{B_2}
+B_2'+\cdots
\quad\quad
$$

This is, in fact, the original approach of Nash, later brought into fluid mechanics by De Lellis and Sz\'ekelyhidi \cite{DeLellisSzekelyhidi09,DeLellisSzekelyhidi13}. The method is highly flexible, since the oscillatory (and, in the case of intermittency, concentrated) perturbations can be chosen in many different ways. Consequently, a wide variety of convex integration schemes now exist, tailored to the specific equations and desired properties.

Our scheme is directly inspired by \cite{ModenaSzekelyhidi18,CheskidovLuo24} and uses building blocks from \cite{BBV20}. The main technical innovation is the avoidance of inverse div/curl operators, which both simplifies the construction and makes it more local and explicit. We now summarize the construction for the reader’s convenience; its components will be treated in detail in the following sections. See also \cite{Isett18,EPP25} for localized versions of convex integration for the 3D Euler equations.

There exists a smooth magnetic potential $\bar{A}$ with $\nabla\times\bar{A} = \bar{B}$, $\nabla \cdot \bar{A} = 0$ in $\Omega$ and $\bar{A} \cdot n = 0$ on $\partial \Omega$ (see e.g.~\cite[Theorem 3.12 and Corollary 2.15]{ABDG}), and the electromotive force $R$ is given by
$$
R = (\partial_t -\epsilon\Delta)\bar{A} - \bar{u}\times\bar{B}.
$$
Notice that the support of $R$ is contained in the support of $(\bar{u},\bar{A})$. 
To preserve the support, we set
$$
u'=\nabla\times\psi',
\quad\quad
B'=\nabla\times A',
$$
for suitable potentials $(\psi',A')$, which will be described below. We begin with the case $\epsilon=0$.
 
\subsubsection{Ideal case} 
For simplicity, we assume here that $(\bar{u},\bar{A})$, and hence $R$, have compact spatial support in $\Omega$. In particular, we impose vanishing boundary conditions.
As a first attempt, we set
\begin{equation}\label{ideal:fluctuation}
\psi'
=\frac{1}{\lambda\mu}\sum_{k=1}^3
\tilde{u}_k(\boldsymbol{\psi}_k)_\lambda,
\quad\quad
A'
=\frac{1}{\lambda\mu}\sum_{k=1}^3
\tilde{B}_k
(\mathbf{A}_k)_\lambda,
\end{equation}
in terms of the oscillation and concentration parameters $\lambda,\mu>0$, to be chosen sufficiently large, the velocity and magnetic potentials $(\boldsymbol{\psi}_k,\mathbf{A}_k)$ of the spatial intermittent shear flows $(\mathbf{u}_k,\mathbf{B}_k)$, which satisfy the crucial Lemma~\ref{Lemma:steady}, and certain coefficients $(\tilde{u}_k,\tilde{B}_k)$ to be determined.
Indeed, these coefficients are taken so that $\tilde{u}_k\tilde{B}_k=R_k=R\cdot e_k$. 
By applying a vector calculus identity, we split
$$
u'
=\underbrace{\sum_{k=1}^3
\tilde{u}_k(\mathbf{u}_k)_\lambda}_{\tilde{u}}
+
\underbrace{\frac{1}{\lambda\mu}\sum_{k=1}^3
(\nabla\tilde{u}_k)\times (\boldsymbol{\psi}_k)_\lambda}_{\tilde{u}_c},
\quad\quad
B'
=\underbrace{\sum_{k=1}^3
\tilde{B}_k(\mathbf{B}_k)_\lambda}_{\tilde{B}}
+
\underbrace{\frac{1}{\lambda\mu}\sum_{k=1}^3
(\nabla\tilde{B}_k)\times(\mathbf{A}_k)_\lambda}_{\tilde{B}_c},
$$
into the mean term $(\tilde{u},\tilde{B})$ and the incompressibility corrector $(\tilde{u}_c,\tilde{B}_c)$. The latter can be made arbitrarily small by taking $\lambda\mu$ large enough, and we therefore neglect them for the remainder of this presentation. 
Thus, the mean contribution to the quadratic error is given by
$$
R-\tilde{u}\times\tilde{B}
=\underbrace{\sum_{k=1}^3R_k(e_k-\mathbf{u}_k\times\mathbf{B}_k)_\lambda}_{R_0^{\text{osc}}}
+\underbrace{\sum_{k\neq k'}\tilde{u}_k\tilde{B}_{k'}(\mathbf{u}_k\times\mathbf{B}_{k'})_\lambda}_{R^{\text{int}}}.
$$
The second term arises from the interaction of building blocks with different indices and, as observed in \cite{ModenaSzekelyhidi18}, can be made small by taking the concentration parameter $\mu$ sufficiently large. The first term is the oscillation error, and thanks to Lemma~\ref{Lemma:steady} it can be replaced by a smaller oscillation error by taking $\lambda$ sufficiently large, namely (see Lemma~\ref{def:Rosc})
$$
R_1^{\text{osc}}
=-\frac{1}{\lambda}\sum_{k=1}^3(\nabla R_k)(\mathbf{H}_k)_\lambda,
$$
which has the same curl as the original oscillation error, i.e.\ $\nabla\times R_0^{\text{osc}}=\nabla\times R_1^{\text{osc}}$. This is the main ingredient that allows us to localize our solutions.

We follow this approach in Section \ref{sec:idealdynamos} to construct weak solutions to the ideal \eqref{dynamo1} equation (see Theorems \ref{Thm:lebesgue} and \ref{Thm:Sobolev}). The only difference is that we introduce a cutoff, as in \cite{ModenaSzekelyhidi18}, to prevent a possible loss of regularity near the points where $R$ vanishes.

As mentioned earlier, the main advantage of the \eqref{dynamo1} equation compared to the \eqref{transport} equation is that the corresponding building blocks are one-dimensional, which makes it easy to construct explicit potentials. The drawback, however, is that diffusion cannot be incorporated for free, in contrast with the transport–diffusion setting. To overcome this, in the next step we incorporate the time intermittency introduced in \cite{CheskidovLuo21}. We also refer to \cite{BDS16} for an intermittent in time version of Onsager conjecture.

\subsubsection{Diffusion case}
In order to construct a fast dynamo, we must incorporate diffusion into the scheme while ensuring that the velocity remains independent of the magnetic diffusivity. Since the defect $R$ depends on $\epsilon$, a natural choice of the coefficients would be $\tilde{u}_k=1$ and $\tilde{B}_k=R_k$. The drawback of this choice is that the support is no longer controlled.
We therefore introduce a spatial cutoff $\chi_\sigma$, where $\sigma>0$ measures the distance to the boundary $\partial\Omega$, and set $\tilde{u}_k=\chi_\sigma$ and $\tilde{B}_k=\chi_\sigma R_k$. Notice that this choice preserves the boundary conditions. Finally, we add the time-intermittent coefficients $(\mathbf{u}_0,\mathbf{B}_0)$ to absorb the diffusion into the defect by compensating it with the time integrability. With these ingredients, we define the new potentials
$$
\psi'
=(\mathbf{u}_0)_{\lambda_0}\frac{\chi_\sigma}{\lambda\mu}\sum_{k=1}^3
(\boldsymbol{\psi}_k)_\lambda,
\quad\quad
A'
=(\mathbf{B}_0)_{\lambda_0}\frac{\chi_\sigma}{\lambda\mu}\sum_{k=1}^3
 R_k (\mathbf{A}_k)_\lambda
+\frac{1}{\lambda_0}(\mathbf{H}_0)_{\lambda_0}\chi_\sigma^2 R,
$$
where $\mathbf{H}_0$ is introduced to compensate for a new term appearing in the quadratic error, as in \cite{CheskidovLuo24}. 
Finally, we need to check that all the parameters defining the velocity—namely, the distance to the boundary, the oscillation and concentration parameters as well as the H\"older exponents—can be chosen independently of $\epsilon$.

We follow this approach in Section~\ref{sec:fastdynamos} to construct weak solutions to the \eqref{dynamo1} equation satisfying the $H$-principle (Theorem~\ref{thm:hprinciple}), which ultimately allows us to prove the existence of a.s.\ turbulent fast dynamos (Theorems~\ref{thm:Main} and~\ref{thm:turbulent}). 

\subsubsection{Extension to the GTE}\label{sec:localized:GTE}
Since the generalization to currents appears in Section~\ref{sec:GTE}, we prefer to present both the building blocks and the convex integration scheme there, in order to avoid forcing the reader to jump back and forth between sections. Here we only explain the key points.

In this setting, we consider the Reynolds \eqref{GTE} equation
\begin{equation}\label{RGTE}\tag{RGTE}
\begin{split}
\partial_t T - \partial(u \wedge T) = \partial R,\\
\nabla \cdot u = 0,\quad \partial T = 0,
\end{split}
\end{equation}
where $R$ is a $(k+1)$-current, which we will also refer to as the defect. Likewise, we will also consider its version with a diffusion term, namely the Reynolds \eqref{DGTE}.

For the \eqref{GTE}, the spatial building blocks $(\mathbf{u}_\alpha,\mathbf{T}_\alpha)$ depend on $x^\alpha$, which denotes the $d-(k+1)$ components of $x \in \mathbb{R}^d$ complementary to the multi-index $\alpha$ of size $k+1$. Using these, we construct the main perturbations 
$$
\tilde{u}
=\sum_{\alpha}\tilde{u}_\alpha(\mathbf{u}_{\alpha})_\lambda,
\quad\quad
\tilde{T}
=\sum_{\alpha}\tilde{T}_\alpha(\mathbf{T}_{\alpha})_\lambda,
$$
which produces the corresponding oscillatory error
$$
R_0^{\text{osc}}
=\sum_{\alpha}R_{\alpha}(e_\alpha-\mathbf{u}_{\alpha}\wedge\mathbf{T}_{\alpha})_\lambda.
$$

On the one hand, we seek a small corrector $\tilde{T}_c$ such that
$
\partial (\tilde{T}+\tilde{T}_c)=0.
$
In fact, we need to find a small potential $\tilde{S}$ such that
$$
\tilde{T}+\tilde{T}_c=\partial\tilde{S}.
$$
On the other hand, we aim to construct a smaller oscillatory error $R_1^{\text{osc}}$ such that
$$
\partial R_1^{\text{osc}}=\partial R_0^{\text{osc}}.
$$

We realize that both problems can be formulated as follows. Let $U$ be of the form
$$
U=f (g_\alpha)_\lambda e_{\beta},
$$
where $\beta \subset \alpha$, and $g_\alpha(x) = g(x^\alpha)$, with $g$ a zero-mean scalar function. We then seek $V$ and $W$, supported in the same region as $f$ and of size $\sim \lambda^{-1}$, such that
$$
U+V=\partial W.
$$

The main contribution of Section~\ref{sec:GTE} is the construction of such a general potential, which appears in Proposition~\ref{prop:potential:GTE}.
The remainder of the convex integration scheme then proceeds similarly to the dynamo case.

\section{Ideal dynamos}\label{sec:idealdynamos}

In this section we apply the ideas from the transport equation \cite{ModenaSzekelyhidi18} and the MHD equation \cite{BBV20} to construct a weak solution $(u,B)\in L_t^\infty(L^p\times L^q)$ to the ideal \eqref{dynamo1} equation ($\epsilon=0$) on  an open domain $\Omega\subset\R^3$, where $1<p,q<\infty$ is a pair of conjugate H\"older exponents
$$\frac{1}{p}+\frac{1}{q}=1.$$
This solution $(u,B)$ is obtained as the limit of a sequence $(u_n,B_n,R_n)$ of solutions to the ideal \eqref{relaxeddynamo} equation, namely 
$$(u_n,B_n,R_n)
\to
(u,B,0)
\quad\quad
\text{in}
\quad\quad
L_t^\infty(L^p\times L^q\times L^1).$$ 
This sequence is constructed iteratively by means of the next proposition. 

\begin{prop}\label{Prop:Lebesgue}
There is a constant $M>0$ such that the following holds. 
For any $\delta,\eta >0$ and any smooth solution $(u_0,B_0,R_0)$ of the ideal \eqref{relaxeddynamo} equation in $[0,T]\times\Omega$, there is another smooth solution $(u_1,B_1,R_1)$ which fulfills the estimates
\begin{subequations}\label{inductiveestimate:ideal}
\begin{align}
\|u_1(t)-u_0(t)\|_{L^p}&\leq \frac{M}{\eta}\|R_0(t)\|_{L^1}^{1/p},\label{inductiveestimate:ideal:1}\\
\|B_1(t)-B_0(t)\|_{L^{q}}&\leq M\eta\|R_0(t)\|_{L^1}^{1/{q}},\label{inductiveestimate:ideal:2}\\
\|R_1(t)\|_{L^1}&\leq\delta,\label{inductiveestimate:ideal:3}
\end{align} 
\end{subequations}
for all $t\in [0,T]$.
Moreover, 
\begin{equation}\label{prop:Lebesgue:support}
u_1(t,x)=u_0(t,x),
\quad B_1(t,x)=B_0(t,x),
\quad R_1(t,x)=0
\quad\text{outside}\quad
\operatorname{supp}R_0.
\end{equation}
\end{prop}
\begin{proof}
See Section \ref{sec:proof:Lebesgue}.
\end{proof}

\begin{thm}\label{Thm:lebesgue}
For any $\delta>0$ and any smooth solution $(\bar{u},\bar{B},R)$ of the ideal \eqref{relaxeddynamo} equation in $[0,T]\times\Omega$, there exists a weak solution $(u,B)\in L_t^\infty(L^p\times L^q)$ of the ideal \eqref{dynamo1} equation satisfying
$$
\|(B-\bar{B})(t)\|_{L^q}\leq\delta,
$$
for all $t\in[0,T]$.
Moreover, 
\begin{equation}\label{thm:Lebesgue:support}
u(t,x)=\bar{u}(t,x),\quad B(t,x)=\bar{B}(t,x)
\quad\text{outside}\quad
\operatorname{supp}R.
\end{equation}
\end{thm}
\begin{proof}
Set $(u_0,B_0,R_0):=(\bar u,\bar B,R)$. 
Let $\delta_n>0$ be a sequence satisfying
\begin{equation}\label{prop:Lebesgue:deltan}
\sum_{n\geq 1}\delta_n^{1/\max\{p,q\}}<\infty,
\end{equation}
and fix a parameter $\eta>0$, to be chosen below. By Proposition \ref{Prop:Lebesgue} there exists a sequence $(u_n,B_n,R_n)$ of smooth solutions to the \eqref{relaxeddynamo} equation fulfilling, for all $n\geq 1$, the inductive estimates
\begin{align*}
\|u_n(t)-u_{n-1}(t)\|_{L^p}&\le \frac{M}{\eta}\,\|R_{n-1}(t)\|_{L^1}^{1/p},\\
\|B_n(t)-B_{n-1}(t)\|_{L^{q}}&\le M\eta\,\|R_{n-1}(t)\|_{L^1}^{1/q},\\
\|R_n(t)\|_{L^1}&\le \delta_n,
\end{align*}
together with the support condition
\begin{equation}\label{prop:Lebesgue:support:n}
u_n(t,x)=u_{n-1}(t,x),
\quad B_n(t,x)=B_{n-1}(t,x),
\quad R_n(t,x)=0
\quad\text{outside}\quad
\operatorname{supp}R_{n-1}.
\end{equation}
Hence, by \eqref{prop:Lebesgue:deltan}, the sequence $(u_n,B_n,R_n)$ is Cauchy in $L_t^\infty\!\big(L^p\times L^q\times L^1\big)$, and therefore converges to some $(u,B,0)$.
Moreover,
\[
u_n\times B_n \to u\times B
\qquad\text{in}\qquad
L_t^\infty L^1,
\]
so that $(u,B)$ is a weak solution of the ideal \eqref{dynamo1} equation.

By choosing $\eta$ at the outset, in terms of $(M,R_0,(\delta_n))$, we can ensure that
\begin{equation}\label{B-barB:Lebesgue}
\|B-\bar B\|_{L^q}
\le \sum_{n\geq 1}\|B_n-B_{n-1}\|_{L^q}
\le M\eta\!\left(\|R_0\|_{L^1}^{\nicefrac{1}{q}}
+\sum_{n\geq 1}\delta_n^{\nicefrac{1}{q}}\right)
\le \delta,
\end{equation}
uniformly in time.

Finally, since \eqref{prop:Lebesgue:support:n} implies that
$$
u_n(t,x)=u_0(t,x),
\quad B_n(t,x)=B_0(t,x)
\quad\text{outside}\quad
\operatorname{supp}R_0,
$$
for all $n\geq 1$, the support condition \eqref{thm:Lebesgue:support} is satisfied.
\end{proof}

\begin{Rem}
As observed in \eqref{B-barB:Lebesgue}, the role of the parameter $\eta$ is to compensate for the first term in the iteration, $\|R_0\|_{L^1}$, which is not necessarily small. Since this issue does not arise in Section~\ref{sec:fastdynamos}, we set $\eta = 1$.
\end{Rem}

\begin{Rem}
Let us make some comments on the space and time domains:
\begin{itemize}
    \item By the translation invariance of the \eqref{dynamo1} equation, we may assume without loss of generality that $\Omega\subset(0,\ell)^3$. Moreover, by the scaling invariance, we may also assume that $\ell=1$ for simplicity.
    \item We restrict to the time interval $[0,T]$. To extend the result to the entire interval $[0,\infty)$, one can follow \cite{ModenaSzekelyhidi18} by incorporating a time cutoff into the iteration scheme. This cutoff can also ensure that the initial datum is preserved. Since these issues will be addressed in Section~\ref{sec:fastdynamos} by incorporating the temporal intermittency from \cite{CheskidovLuo24}, we omit this part of the construction here to keep the presentation as simple as possible.
\end{itemize}
\end{Rem}

In Sections \ref{sec:newB1}–\ref{sec:Rcor} we establish Proposition \ref{Prop:Lebesgue}, which is summarized in Section \ref{sec:proof:Lebesgue} for the reader's convenience. Section \ref{sec:Sobolevversion} is then devoted to improving the Sobolev regularity of $(u,B)$, but we begin with simpler preliminary estimates.

\subsection{The new $(u_1,B_1,R_1)$}\label{sec:newB1}
Given $(u_0,B_0,R_0)$ we declare
\begin{equation}\label{def:newB1u1}
u_1:=u_0 + \underbrace{\tilde{u} + \tilde{u}_c}_{\nabla\times\tilde{\psi}},
\quad\quad
B_1:=B_0 + \underbrace{\tilde{B} + \tilde{B}_c}_{\nabla\times\tilde{A}}.
\end{equation}
The mean terms $(\tilde{u}, \tilde{B})$ are introduced in Section \ref{sec:mean} to cancel the lower frequencies of the old defect $R_0$. 
The incompressibility correctors $(\tilde{u}_c, \tilde{B}_c)$ are introduced in Section~\ref{sec:correctors} to ensure that both $u_1$ and $B_1$ are divergence-free; namely, we construct suitable velocity and magnetic potentials $\tilde{\psi}$ and $\tilde{A}$ as indicated above. 
These will be defined in terms of the spatial intermittent shear flows introduced in Section~\ref{sec:buildingblocks}, the conjugate H\"older exponents $p,q$, and the parameters
\begin{align*}
\textit{concentration}&\quad\quad\mu\geq\ell^{-1}\\
\textit{oscillation}&\quad\quad\lambda\in\N
\end{align*}
both taken sufficiently large, to be specified later. The new defect $R_1$ must satisfy the ideal \eqref{relaxeddynamo} equation, which can be written in terms of the decomposition \eqref{def:newB1u1} as follows:
\begin{subequations}\label{R1}
\begin{align}
\nabla\times R_1
&=\partial_t B_1 - \nabla\times(u_1\times B_1)\nonumber\\
&= \partial_t B_0 - \nabla\times(u_0\times B_0) - \nabla\times(\tilde{u}\times\tilde{B})\label{term:quad}\\
&+ \nabla\times(\partial_t\tilde{A}-\tilde{u}\times B_0 - u_0\times\tilde{B})\label{term:lin}\\
&-\nabla\times(\tilde{u}_c\times B_1
+u_1\times \tilde{B}_c
-\tilde{u}_c\times \tilde{B}_c)\label{term:cor}.
\end{align}
\end{subequations}
We decompose $R_1$ into three terms depending on each line of \eqref{R1}
\begin{equation}\label{R1:decomposition}
R_1:=R^{\text{quad}}
+R^{\text{lin}}
+R^{\text{cor}}.
\end{equation}

On the one hand, since $(u_0,B_0,R_0)$ is a solution to the ideal \eqref{relaxeddynamo} equation, the first line \eqref{term:quad} equals
$$\eqref{term:quad}=\nabla\times(R_0-\tilde{u}\times\tilde{B}).$$
As we will explain in Section~\ref{sec:Rquad}, the choice of $(\tilde{u},\tilde{B})$ guarantees the existence of a small \textit{quadratic error} satisfying
$$\nabla\times R^{\text{quad}}=\nabla\times(R_0-\tilde{u}\times\tilde{B}).$$
As we mentioned in Section \ref{sec:scheme}, the main difference compared to previous works is that here we localize the quadratic error by means of a potential (see Lemma~\ref{def:Rosc}). 

On the other hand, for the second and third line, \eqref{term:lin} and \eqref{term:cor},  we can take the \textit{linear error} and the \textit{corrector error} simply as
\begin{equation}\label{eq:Rlincor}
\begin{split}
R^{\text{lin}}
&:=\partial_t\tilde{A}-\tilde{u}\times B_0 - u_0\times\tilde{B},\\
R^{\text{cor}}
&:=-\tilde{u}_c\times B_1
-u_1\times \tilde{B}_c
+\tilde{u}_c\times \tilde{B}_c.
\end{split}
\end{equation}
As we will explain in Sections~\ref{sec:Rlin} and~\ref{sec:Rcor}, these terms can be made arbitrarily small by taking $\lambda$ and $\mu$ sufficiently large

\begin{Rem}
From this point onward, we define and estimate all terms in the construction at a fixed time $t$, omitting the time dependence whenever it is not essential for ease of notation. We use the letters $C$ and $M$ to denote positive constants: $C$ may depend on the given inputs of the statement, while $M$ depends only on the domain $[0,T]\times\Omega$. Their values may change from line to line, but the notation remains fixed for simplicity.
\end{Rem}

\subsection{The mean term $(\tilde{u},\tilde{B})$}\label{sec:mean}

For technical reasons, we start by fixing a smooth cutoff function $\chi:[0,\infty)\to [0,1]$ satisfying
$$\chi(s)
=\left\lbrace
\begin{array}{rl}
0, & s\leq 1/2,\\
1, & s\geq 1.
\end{array}\right.$$
For any index $k=1,2,3$, we define 
$$\chi_{k}(t,x)
:=\chi\left(\frac{|R_{0,k}(t,x)|}{\delta/9}\right),$$
where $R_{0,k}=R_0\cdot e_k$, that is
\begin{equation}\label{Rdecomposition} 
R_0(t,x)=\sum_{k=1}^3R_{0,k}(t,x)e_k.
\end{equation}

The main term $(\tilde{u},\tilde{B})$ is defined in terms of the old defect $R_0$, the spatial intermittent shear flows $(\mathbf{u}_k,\mathbf{B}_{k})$ given in \eqref{def:intermittentshearflows}, and the cutoff $\chi$ as follows (recall $f_\lambda(x):=f(\lambda x)$)
\begin{equation}\label{eq:defTildeB}
\tilde{u}
:=\frac{1}{\eta}\sum_{k=1}^{3}\underbrace{\chi_{k}|R_{0,k}|^{1/p}}_{\tilde{u}_{k}}(\mathbf{u}_k)_\lambda,
\quad\quad
\tilde{B}
:=\eta\sum_{k=1}^{3}\underbrace{\chi_{k}\operatorname{sgn}(R_{0,k})|R_{0,k}|^{1/q}}_{\tilde{B}_k}(\mathbf{B}_{k})_\lambda.
\end{equation}
Notice that the coefficients $(\tilde{u}_k,\tilde{B}_k)$ are scalar functions. 
The role of $\chi_k$
is to guarantee the smoothness of $(\tilde{u}_k,\tilde{B}_k)$ for $p,q\neq\infty$ by removing the zero set
of $R_{0,k}$. Indeed, in Section \ref{sec:fastdynamos}, this cutoff will not be necessary.

We now turn to the derivation of the corresponding integrability of the main term. The key tool is the following \textit{improved H\"older inequality} (\cite[Lemma 2.1]{ModenaSzekelyhidi18}).

\begin{prop}\label{prop:IHI}
For any $\lambda\in\N$, smooth functions $f,g:\T_\ell^d\to\R$, and $1\leq r\leq\infty$,
$$
	\bigg|\|fg_\lambda\|_{L^r}
	-\|f\|_{L^r}\|g\|_{L^r}\bigg|
	\leq\frac{C_{r,\ell}}{\lambda^{1/r}}\|f\|_{C^1}\|g\|_{L^r}.
$$
\end{prop}

\begin{lemma}[Lebesgue bounds on $(\tilde{u},\tilde{B})$] \label{lem:Lestimates}
There exist $M>0$ and $C=C(p,\delta,\|R_0\|_{C^1})>0$ such that
\begin{align*}
\|\tilde{u}\|_{L^p}
&\leq \frac{M}{\eta}\left(\|R_{0}\|_{L^1}^{1/p} + \frac{C}{\lambda^{1/p}}\right),\\
\|\tilde{B}\|_{L^q}
&\leq M\eta\left(\|R_{0}\|_{L^1}^{1/q} + \frac{C}{\lambda^{1/q}}\right).
\end{align*}
\end{lemma}
\begin{proof} 
A direct application of Proposition \ref{prop:IHI} to each term of $\tilde{u}$ in \eqref{eq:defTildeB}
yields that
$$
\|\tilde{u}_{k}(\mathbf{u}_k)_\lambda\|_{L^p}
\leq\frac{1}{\eta}\left(\|\tilde{u}_k\|_{L^p}
+\frac{C}{\lambda^{1/p}}\|\tilde{u}_k\|_{C^1}\right)\|\mathbf{u}_k\|_{L^p}.
$$
Now, directly from the definition of $\tilde{u}_k$,
\begin{align*}
\|\tilde{u}_k\|_{L^p}&\leq\|R_{0,k}\|_{L^1}^{1/p},\\
\|\tilde{u}_k\|_{C^1}&\leq C(\delta,\|R_{0,k}\|_{C^1}).
\end{align*}
Since 
$\|\mathbf{u}_k\|_{L^p}=\ell^{1/p}\|\phi\|_{L^p}$ by definition \eqref{def:intermittentshearflows} and Lemma \ref{lemma:buildingblock}, the $L^p$ estimate for $\tilde{u}$ is completed. The $L^q$ estimate for $\tilde{B}$ is completely analogous.
\end{proof}

All in all, Lemma~\ref{lem:Lestimates} allows to  take $\lambda$ large enough, in terms of $(p,\delta,\|R_0\|_{C^1})$, to guarantee that 
\begin{equation}\label{eq:uB:Lebesgue}
\begin{split}
\|\tilde{u}\|_{L^p}&\leq\frac{M}{\eta}\|R_0\|_{L^1}^{1/p},\\
\|\tilde{B}\|_{L^q}&\leq M\eta\|R_0\|_{L^1}^{1/q},
\end{split}  
\end{equation}
for some universal constant $M>0$.

\subsection{The quadratic error}
\label{sec:Rquad}
By the definition \eqref{eq:defTildeB} of $(\tilde{u},\tilde{B})$ we have
\begin{align*}
\tilde{u}\times\tilde{B}
&=\sum_{k=1}^{3}
\tilde{u}_{k}\tilde{B}_{k} (\mathbf{u}_{k}\times \mathbf{B}_{k})_\lambda
+\sum_{k\neq k'}\tilde{u}_{k}\tilde{B}_{k'} (\mathbf{u}_{k}\times \mathbf{B}_{k'})_\lambda.
\end{align*}
The first term comes from the products of shear flows with the same index, and the second term from the intersection between different indices. 
By applying that
$$\tilde{u}_{k}\tilde{B}_{k}
=\chi_{k}^2 R_{0,k},$$
we split
\begin{subequations}\label{eq:R0-uxB}
\begin{align}
R_0-\tilde{u}\times\tilde{B}
&=\sum_k
\chi_{k}^2R_{0,k}(e_k-\mathbf{u}_{k}\times \mathbf{B}_{k})_\lambda
&&=:R^{\text{osc}}_0\label{eq:R0-uxB:osc}\\
&+\sum_k(1-\chi_{k}^2)
R_{0,k}e_k
&& =: R^\chi\\
&-\sum_{k\neq k'}\tilde{u}_{k}\tilde{B}_{k'} (\mathbf{u}_{k}\times \mathbf{B}_{k'})_\lambda
&&=:R^{\text{int}}
\end{align}
\end{subequations}
Firstly, we will find $R^{\text{osc}}_1$ with small $L^1$-norm and the same curl as $R^{\text{osc}}_0$. Secondly, we will bound the $L^1$-norm of $R^\chi$ and $R^{\text{int}}$ directly. With them we define
\begin{equation}\label{def:Rquad}
R^{\text{quad}}
:=R^{\text{osc}}_1
+R^\chi
+R^{\text{int}}.
\end{equation}

\subsubsection{The oscillation error}
In the following lemma we find the small $R_1^{osc}$. The crucial observation is that, thanks to Lemma \ref{Lemma:steady}, we can express \eqref{eq:R0-uxB:osc} as
$$
R^{\text{osc}}_0
=\sum_{k=1}^{3}
\chi_{k}^2R_{0,k}(\nabla\mathbf{H}_k)_\lambda.
$$
Thus, we 
can use the following vector calculus identity for smooth and periodic functions $f,h$
\begin{equation}\label{eq:vectorcalculus:1}
\nabla \times (f (\nabla h)_\lambda)=-\frac{1}{\lambda} \nabla \times (h_\lambda \nabla f),
\end{equation}
to gain decay by selecting a big oscillation parameter $\lambda$.

\begin{lemma}\label{def:Rosc}
We define
$$R^{\text{osc}}_1
:=-\frac{1}{\lambda}\sum_{k=1}^{3}
\nabla(\chi_k^2R_{0,k})(\mathbf{H}_k)_\lambda.$$
Then, we have
$$\nabla\times R^{\text{osc}}_0
=\nabla\times R^{\text{osc}}_1.$$
\end{lemma}
\begin{proof}
It follows by applying the curl operator to each term of $R_0^{\text{osc}}$ and $R_1^{\text{osc}}$, and then using the identity \eqref{eq:vectorcalculus:1} to $f=\chi_k^2 R_{0,k}$ and $h=\mathbf{H}_k$.
\end{proof}

\begin{lemma}[$L^1$-bound of $R^{\text{osc}}_1$]\label{lemma:Rosc} There exists $C=C(\delta,\|R_0\|_{C^1})>0$ such that
$$
\|R^{\text{osc}}_1\|_{L^1}
\leq\frac{C}{\lambda}.
$$
\end{lemma}
\begin{proof}
	The claim follows from
	$$\|\nabla(\chi_k^2R_{0,k})\|_{L^\infty}
	\leq C(\delta,\|R_0\|_{C^1}),$$
	and \eqref{eq:Hbounded}, namely
$$\|(\mathbf{H}_k)_\lambda\|_{L^1}
=\|\mathbf{H}_k\|_{L^1}
\leq\ell\|\mathbf{H}_k\|_{L^\infty}\leq 2\ell^2.
$$
\end{proof}

\subsubsection{The cutoff error}
In this section we bound $R^{\chi}$. We recall that this is a technical term for isolating the regions of non-smoothness of $|R|$.
\begin{lemma}[$L^1$-bound of $R^\chi$]\label{lemma:Rcutoff} We have
	$$\|R^\chi\|_{L^1}\leq\delta/3.$$
\end{lemma}
\begin{proof} Since $|R_{0,k}|\leq\delta/9$ on the support of $(1-\chi_{k}^2)$, we have
	$$\|R^\chi\|_{L^1}
	\leq\sum_{k=1}^{3}\int_{\operatorname{supp}(1-\chi_{k}^2)}
	|R_{0,k}(t,x)|\dif x
	\leq\delta/3.$$
\end{proof}

\subsubsection{The intersection error}
For the cross terms, we observe that they depend on different variables. This allows us to use the concentration parameter $\mu$ to make the intersection small. 

\begin{lemma}[$L^1$-bound of $R^{\text{int}}$]\label{lemma:Rint}
There exists $M>0$ such that
$$\|R^{\text{int}}\|_{L^1}
\leq
\frac{M}{\mu}\|R_0\|_{C^0}.$$
\end{lemma}
\begin{proof}
On the one hand, we have the pointwise bound
$$|\tilde{u}_{k}\tilde{B}_{k'}|
\leq|R_0|.$$
On the other hand, since
$$|(\mathbf{u}_{k}\times \mathbf{B}_{k'})(x)|
=\mu|\phi_{\mu}(x_k)\phi_{\mu}(x_{k'})|,$$
and $k\neq k'$, we have
$$
\|(\mathbf{u}_{k}\times \mathbf{B}_{k'})_\lambda\|_{L^1}
=\mu\|\phi_\mu\|_{L^1}^2
=\mu^{-1}\|\phi\|_{L^1}^2.
$$ 
\end{proof}

By combining the decomposition \eqref{def:Rquad} of $R^{\text{quad}}$ with Lemmas \ref{lemma:Rosc}-\ref{lemma:Rint}, we can take $(\lambda,\mu)$ large enough, in terms of $(\delta,\|R_0\|_{C^1})$, such that
\begin{equation}\label{Rquad:estimate}
\|R^{\text{quad}}\|_{L^1}\leq\delta/2.
\end{equation}

\subsection{The incompressibility correctors $(\tilde{u}_c,\tilde{B}_c)$}\label{sec:correctors}

In this section, we define the incompressibility correctors of the mean term $(\tilde{u}, \tilde{B})$ constructed in Section~\ref{sec:mean}, by means of the velocity and magnetic potentials $(\boldsymbol{\psi}_k,\mathbf{A}_k)$ defined in \eqref{def:magneticvectorpotentials}. 

\begin{lemma}\label{lemma:incompressibilitycorrectors}
We define 
$$
\tilde{u}_c
:=\frac{1}{\eta}\frac{1}{\lambda\mu}\sum_{k=1}^{3}\nabla \tilde{u}_k\times (\boldsymbol{\psi}_k)_\lambda,
\quad\quad
\tilde{B}_c
:=\eta\frac{1}{\lambda\mu}\sum_{k=1}^{3}\nabla \tilde{B}_k\times (\mathbf{A}_k)_\lambda.
$$
Then, we have $\tilde{u}+\tilde{u}_c=\nabla\times\tilde{\psi}$ and $\tilde{B}+\tilde{B}_c
=\nabla\times \tilde{A}$, where 
$$
\tilde{\psi}
:=\frac{1}{\eta}\frac{1}{\lambda\mu}\sum_{k=1}^{3}\tilde{u}_k({\boldsymbol{\psi}}_k)_\lambda,
\quad\quad
\tilde{A}
:=\eta\frac{1}{\lambda\mu}\sum_{k=1}^{3}\tilde{B}_k({\mathbf{A}}_k)_\lambda.
$$
In particular,
$\nabla\cdot 
(\tilde{u}+\tilde{u}_c)
=\nabla\cdot 
(\tilde{B}+\tilde{B}_c)=0$.
\end{lemma}
\begin{proof}
It follows from the vector calculus identity for smooth and periodic scalar $f$ and vector $A$ functions
$$
\frac{1}{\lambda}\nabla\times(fA_\lambda)
=f(\nabla\times A)_\lambda
+\frac{1}{\lambda}\nabla f\times A_\lambda,
$$
and Lemma \ref{Lemma:steady}.
\end{proof}

\begin{lemma}[Lebesgue estimates on $(\tilde{u}_c,\tilde{B}_c)$]\label{lem:Lestimates:cor}
There exists $C(\delta,\|R_0\|_{C^1})>0$ such that
\begin{align*}
\|\tilde{u}_c\|_{L^p}
&\leq\eta\frac{C}{\lambda\mu},\\
\|\tilde{B}_c\|_{L^q}
&\leq\frac{1}{\eta} \frac{C}{\lambda\mu}.
\end{align*}
\end{lemma}
\begin{proof} 
Since 
$\|\mathbf{u}_k\|_{L^p}=\ell^{1/p}\|\phi\|_{L^p}$ and $\|\mathbf{B}_k\|_{L^q}=\ell^{1/q}\|\phi\|_{L^q}$ by definition \eqref{def:intermittentshearflows} and Lemma \ref{lemma:buildingblock}, the claim follows by applying the H\"older inequality.
\end{proof}

Thus, we can take $\lambda\mu$ large enough, in terms of $(\delta,\|R_0\|_{C^1})$, such that
\begin{equation}\label{eq:uBcor:Lebesgue}
\begin{split}
\|\tilde{u}_c\|_{L^p}&\leq\frac{M}{\eta}\|R_0\|_{L^1}^{1/p},\\
\|\tilde{B}_c\|_{L^q}&\leq M\eta\|R_0\|_{L^1}^{1/q}.
\end{split}
\end{equation}

\subsection{The linear error}\label{sec:Rlin}
As we anticipated in Section \ref{sec:newB1}, we define 
$$R^{\text{lin}}
:=\partial_t\tilde{A} - \tilde{u}\times B_0 - u_0\times\tilde{B}.$$

\begin{lemma}[$L^1$-Bound of $R^{\text{lin}}$]\label{lemma:Rlin}
There exists $C=C(\delta,\|R_0\|_{C_{t,x}^1})>0$ such that
$$\|\partial_t \tilde{A}\|_{L^1}
\leq \eta \frac{C}{\lambda\mu^{1+1/q}}.$$
In addition, there exists $M>0$ such that
\begin{align*}
\|\tilde{u}\times B_0\|_{L^1}
&\leq\frac{M}{\eta}\frac{\|R_0\|_{C^0}^{1/p}\|B_0\|_{C^0}}{\mu^{1/q}},\\
\|u_0\times\tilde{B}\|_{L^1}
& \leq M\eta\frac{\|R_0\|_{C^0}^{1/q}\|u_0\|_{C^0}}{\mu^{1/p}}.
\end{align*}
\end{lemma}
\begin{proof}
Since $({\mathbf{A}}_k)_\lambda$ is time independent, we have
$$\partial_t \tilde{A}
=\frac{\eta}{\lambda\mu}\sum_{k=1}^{3}\partial_t\tilde{B}_k({\mathbf{A}}_k)_\lambda.$$
Thus, the first estimate follows from the H\"older inequality recalling that $\|{\mathbf{A}}_k\|_{L^1}=\mu^{1/p-1}\|\Phi\|_{L^1}$ from definition \eqref{def:magneticvectorpotentials} and Lemma \ref{lemma:buildingblock}. The second estimate follows from
$$\|\tilde{u}_k(\mathbf{u}_{k})_\lambda\times  B_0\|_{L^1}
\leq
\|\tilde{u}_k\|_{C^0}\|B_0\|_{C^0}\|\mathbf{u}_{k}\|_{L^1},$$
and $\|\mathbf{u}_{k}\|_{L^1}
=\mu^{1/p-1}\|\phi\|_{L^1}$ from definition \eqref{def:intermittentshearflows} Lemma \ref{lemma:buildingblock}. The estimate for the third term  $u_0 \times\tilde{B}$ is analogous.
\end{proof}

Thus, we can take $(\lambda,\mu)$ large enough, in terms of $(\delta,\eta,\|R_0\|_{C_{t,x}^1},\|B_0\|_{C^0},\|u_0\|_{C^0})$, such that
\begin{equation}\label{Rlin:estimate}
\|R^{\text{lin}}\|_{L^1}\leq\delta/8.
\end{equation}

\subsection{The corrector error}\label{sec:Rcor}
As we anticipated in Section \ref{sec:newB1}, we define 
$$R^{\text{cor}}
=-\tilde{u}_c\times B_1
-u_1\times \tilde{B}_c
+\tilde{u}_c\times \tilde{B}_c.$$

\begin{lemma}[$L^1$-bound of $R^{\text{cor}}$]\label{lemma:Rcor}
There exists $C=C(\delta,\eta,\|R_0\|_{C^2},\|B_0\|_{C^0},\|u_0\|_{C^0})>0$ such that
$$
\|\tilde{u}_c\times B_1\|_{L^1},\,
\|u_1\times\tilde{B}_c\|_{L^1}
\leq\frac{C}{\lambda\mu},
$$
and also
$$
\|\tilde{u}_c\times\tilde{B}_c\|_{L^1}\leq\frac{C}{(\lambda\mu)^2}.
$$
\end{lemma}
\begin{proof}
Since $u_1=u_0+\tilde{u}+\tilde{u}_c$ and $B_1=B_0+\tilde{B}+\tilde{B}_c$, the claim by applying the H\"older inequality and Lemmas \ref{lem:Lestimates} and \ref{lem:Lestimates:cor}.
\end{proof}

Thus, we can take $(\lambda,\mu)$ large enough, in terms of $(\delta,\eta,\|R_0\|_{C^2},\|B_0\|_{C^0},\|u_0\|_{C^0})$, such that
\begin{equation}\label{Rcor:estimate}
\|R^{\text{cor}}\|_{L^1}\leq\delta/8.
\end{equation}

\subsection{Proof of Proposition \ref{Prop:Lebesgue}}\label{sec:proof:Lebesgue}
We now summarize the results obtained in the previous sections, leading to the proof of \eqref{inductiveestimate:ideal} and \eqref{prop:Lebesgue:support}, for the convenience of the reader.

Let $(u_1,B_1)$ be as defined in \eqref{def:newB1u1}. As established in \eqref{eq:uB:Lebesgue} and \eqref{eq:uBcor:Lebesgue}, it is possible to choose $\lambda$ and $\mu$ sufficiently large such that
\eqref{inductiveestimate:ideal:1} and \eqref{inductiveestimate:ideal:2} hold:
\begin{align*}
\|u_1-u_0\|_{L^p}
&\leq\|\tilde{u}\|_{L^p}+\|\tilde{u}_c\|_{L^p}
\leq \frac{M}{\eta}\|R_0\|_{L^1}^{\nicefrac{1}{p}},\\
\|B_1-B_0\|_{L^q}
&\leq\|\tilde{B}\|_{L^q}+\|\tilde{B}_c\|_{L^q}
\leq M\eta\|R_0\|_{L^1}^{\nicefrac{1}{q}},
\end{align*}
for some universal constant $M>0$.

Similarly, let $R_1$ be as in \eqref{R1:decomposition}. As established in \eqref{Rquad:estimate}, \eqref{Rlin:estimate}, and \eqref{Rcor:estimate}, one can take $\lambda$ and $\mu$ sufficiently large to ensure \eqref{inductiveestimate:ideal:3}:
$$
\|R_1\|_{L^1}
\leq\|R^{\text{quad}}\|_{L^1}
+\|R^{\text{lin}}\|_{L^1}
+\|R^{\text{cor}}\|_{L^1}
\leq\delta.
$$

Finally, since $\tilde{u}=\tilde{B}=\tilde{u}_c=\tilde{B}_c=0$
and $R^{\text{quad}}=R^{\text{lin}}=R^{\text{cor}}=0$
outside $\operatorname{supp}R_0$ (indeed, all of them are defined in terms of $R_0$; see \eqref{eq:Rlincor}, \eqref{eq:defTildeB}, \eqref{eq:R0-uxB}, \eqref{def:Rquad}, and Lemmas \ref{def:Rosc} and \ref{lemma:incompressibilitycorrectors}), the support condition \eqref{prop:Lebesgue:support} is satisfied.

\subsection{Sobolev regularity}\label{sec:Sobolevversion}
In this section, we discuss the regularity of the solutions obtained through the previous convex integration scheme. As will be explained in the next section, the ultimate goal is to control the $L^1$-norm of the curl of the magnetic perturbation, which makes it possible to incorporate the diffusion term. In fact, it is straightforward to show that the following bounds hold for the perturbation.

\begin{lemma}[Sobolev bound of the perturbation]\label{lemma:Sobolev} 
For any $m\geq 0$ and $1\leq r\leq\infty$, there exists $C=C(m,\delta,R_0)>0$ such that
\begin{align*}
\|\tilde{\psi}\|_{\dot{W}^{m,r}}
&\leq\frac{C}{\eta}(\lambda\mu)^{m-1}\mu^{\tfrac{1}{p}-\tfrac{1}{r}},\\[0.2cm]
\|\tilde{A}\|_{\dot{W}^{m,r}}
&\leq C\eta(\lambda\mu)^{m-1}\mu^{\tfrac{1}{q}-\tfrac{1}{r}}.
\end{align*}
\end{lemma}
\begin{proof}
For the intermittent potentials $(\boldsymbol{\psi}_k,\mathbf{A}_k)$, these bounds follows from Lemma \ref{lemma:buildingblock}. For the complete term $(\tilde{\psi},\tilde{A})$ it is enough to apply the Leibniz rule when $m$ is an integer, and finally interpolate for the non-integer exponents.
\end{proof}

Hence, whenever the exponent of $\mu$ is negative in Lemma \ref{lemma:Sobolev}, namely
\eqref{Prop:Sobolev:regime},
we can take $\mu$ large enough, in terms of $(\delta,\eta,\lambda,R_0)$, to obtain the following (fractional) Sobolev regularity.

\begin{prop}\label{Prop:Sobolev}
Consider $r,s\geq 0$ and $1<\tilde{p},\tilde{q}<\infty$ in the regime
\begin{equation}\label{Prop:Sobolev:regime}
r+\frac{1}{p}<\frac{1}{\tilde{p}},
\quad\quad
s+\frac{1}{q}<\frac{1}{\tilde{q}}.
\end{equation}

Then, in the setting of Proposition \ref{Prop:Lebesgue}, we can also guarantee that
\begin{align*}
\|u_1(t)-u_0(t)\|_{W^{r,\tilde{p}}}&\leq \delta,\\
\|B_1(t)-B_0(t)\|_{W^{s,\tilde{q}}}&\leq \delta.
\end{align*}
\end{prop}
\begin{proof}
Since $u_1-u_0=\nabla\times\tilde{\psi}$ and $B_1-B_0=\nabla\times\tilde{A}$, the claim follows from Lemma \ref{lemma:Sobolev}.
\end{proof}

As a consequence, we deduce a Sobolev version of Theorem \ref{Thm:lebesgue}.

\begin{thm}\label{Thm:Sobolev}
Consider $r,s\geq 0$ and $1<\tilde{p},\tilde{q}<\infty$ in the regime \eqref{Prop:Sobolev:regime}. Then, in the setting of Theorem \ref{Thm:lebesgue}, we can also guarantee that
$(u,B)\in L_t^\infty(W^{r,\tilde{p}}\times W^{s,\tilde{q}})$
with
\begin{align*}
\|u-\bar{u}\|_{L_t^\infty W^{r,\tilde{p}}}&\leq\delta,\\
\|B-\bar{B}\|_{L_t^\infty W^{s,\tilde{q}}}&\leq\delta .
\end{align*}
\end{thm}

Unfortunately, at this point we can only guarantee that $s<1$ as $q\to\infty$ and $\tilde{q}\to 1$, which prevents us from considering the diffusion term
$$
\Delta B
=-\nabla\times(\nabla\times B),
$$
by incorporating $\nabla\times B$ to the defect. 
This is in stark contrast with the convex integration scheme from where we were drawing inspiration in the first place,  
namely the (3D) transport-diffusion equation 
\begin{align*}
\partial_t\rho + \nabla\cdot(\rho u) &= \epsilon\Delta\rho,\\
\nabla\cdot u &=0.
\end{align*}
In this case, the spatial-intermittent density and velocity field are given in \cite{ModenaSzekelyhidi18} by
\begin{align*}
\mathbf{v}_{k}(x)&:=\mu^{(3-1)/p}\varphi_{\mu}(x^k)e_k,\\
\boldsymbol{\rho}_{k}(x)&:=\mu^{(3-1)/q}\varphi_{\mu}(x^k),
\end{align*}
where
$
x^k
=(\ldots,x_{k-1},x_{k+1},\ldots),
$
for some cutoff function $\varphi\in C_c^\infty(\R^{3-1})$. Such  $\varphi$, depending on more variables, yields regularity in $W^{1,p}$ for $p<2$, in particular allows to recover a full derivative. Since $\Delta\rho=\nabla\cdot(\nabla\rho)$, one can absorb the gradient into the defect.

In \cite{ModenaSattig20}, Modena and Sattig were able to improve the regularity by adding other parameters to the construction, in particular a ``phase speed'' (like $\phi_\mu(x_k-\alpha t)$), which appear later in other constructions.
Since it is not clear how to incorporate such building blocks in a straightforward way into dynamo theory, nor whether this approach would allow one to construct a fast dynamo, we instead turn to time intermittency as introduced by Cheskidov and Luo \cite{CheskidovLuo24}.

\section{Turbulent fast dynamos}\label{sec:fastdynamos}

The goal of this section is to prove the H-principle (Theorem \ref{thm:hprinciple}) and to show how it implies the existence of turbulent fast dynamos (Theorems \ref{thm:Main} and \ref{thm:turbulent}). As an intermediate step we need to prove that the required coarse-grained magnetic fields exist (Theorems \ref{thm:fastdynamo} and \ref{thm:turbulentdynamo}).

The construction proceeds similarly to the ideal case, except for some modifications that we summarize below. 
\begin{itemize}
    \item Since $\tilde{u}_k \sim |R_{0,k}|^{1/p}$, the velocity field in Section \ref{sec:idealdynamos} depends on the defect (which in turn depends on $\epsilon$), unless we take ``$p = \infty$ \& $q = 1$".
    With this choice of exponents, the spatial building blocks are given by 
    \begin{align*}
    \mathbf{u}_k(x)&:=\phi_{\mu}(x_k)e_i,
    &\boldsymbol{\psi}_k(x)
    &:=-\Phi_{\mu}(x_k) e_j,\\
    \mathbf{B}_k(x)&:=\ell\mu\phi_{\mu}(x_k)e_j,
    &\mathbf{A}_k(x)
    &:=\ell\mu\Phi_{\mu}(x_k) e_i,
    \end{align*}
    while $\mathbf{H}_k$ remains unchanged.
    \item At this point, the issue is that $\tilde{u}_k$ would be equal to $1$, and we would lose control of the spatial support as well as the boundary conditions. To avoid this problem, we apply Urysohn’s lemma to introduce a smooth cutoff $\chi_\sigma :\Omega \to [0,1]$ such that 
    $$
    \chi_\sigma(x)=
    \begin{cases}
    1, & \operatorname{dist}(x,\partial\Omega)>\sigma, \\
    0, & \operatorname{dist}(x,\partial\Omega)<\sigma/2.
    \end{cases}
    $$
    Then, we take $\tilde{u}_k=\chi_\sigma$ and $\tilde{B}_k=\chi_\sigma R_{0,k}$, that is,
    \begin{equation}\label{eq:defTildeB:diff}
    \tilde{u}
    :=\sum_{k=1}^3\chi_\sigma(\mathbf{u}_k)_\lambda,
    \quad\quad
    \tilde{B}
    :=\sum_{k=1}^3 \chi_\sigma R_{0,k}(\mathbf{B}_k)_\lambda.
    \end{equation}
    As mentioned previously, in this case it suffices to take $\chi = \eta = 1$. 
    \item The incompressibility correctors are given by
    $$
    \tilde{u}_c
    :=\frac{1}{\lambda\mu}\sum_{k=1}^{3}\nabla\chi_\sigma\times (\boldsymbol{\psi}_k)_\lambda,
    \quad\quad
    \tilde{B}_c
    :=\frac{1}{\lambda\mu}\sum_{k=1}^{3}\nabla (\chi_\sigma R_{0,k})\times (\mathbf{A}_k)_\lambda,
    $$
    which satisfy $\tilde{u}+\tilde{u}_c
    =\nabla\times\tilde{\psi}$ and $\tilde{B}+\tilde{B}_c
    =\nabla\times\tilde{A}$, where
    $$
    \tilde{\psi}
    :=\frac{1}{\lambda\mu}\sum_{k=1}^{3}\chi_\sigma(\boldsymbol{\psi}_k)_\lambda.
    \quad\quad
    \tilde{A}
    :=\frac{1}{\lambda\mu}\sum_{k=1}^{3}\chi_\sigma R_{0,k}(\mathbf{A}_k)_\lambda.
    $$
    \item The new ingredient is the use of the time-intermittent velocity and magnetic coefficients $(\mathbf{u}_0,\mathbf{B}_0)$ given in \eqref{def:timeintermittency:Bu}. Namely, we define the new mean terms, incompressibility corrector, and potentials as
    \begin{align*}
     \tilde{u}_{\text{dyn}}
    &:=(\mathbf{u}_0)_{\lambda_0}\tilde{u},&
    \tilde{B}_{\text{dyn}},
    &:=(\mathbf{B}_0)_{\lambda_0}\tilde{B}, \\
    \tilde{u}_{\text{dyn},c}
    &:=(\mathbf{u}_0)_{\lambda_0}\tilde{u}_c,&
    \tilde{B}_{\text{dyn},c}
    &:=(\mathbf{B}_0)_{\lambda_0}\tilde{B}_c,\\
    \tilde{\psi}_{\text{dyn}}
    &:=(\mathbf{u}_0)_{\lambda_0}\tilde{\psi},&
    \tilde{A}_{\text{dyn}}
    &:=(\mathbf{B}_0)_{\lambda_0}\tilde{A}.
    \end{align*}
    In most of the estimates we will use the inequality
    $$
    \|cf\|_{L_t^r X}\leq
    \|c\|_{L^r}\|f\|_{L_t^\infty X},
    $$
where $c(t)$ denotes one of the time-intermittent coefficients (so that Lemma \ref{lemma:buildingblock} applies), and $f(t,x)$ is one of the terms estimated in the ideal case (Section \ref{sec:idealdynamos}), in the corresponding space $X$. 
\item Additionally, we incorporate a temporal corrector
$$
H
:=-\frac{1}{\lambda_0}(\mathbf{H}_0)_{\lambda_0}\chi_\sigma^2 R_0,
$$
with $\mathbf{H}_0$ given in \eqref{def:timeintermittency:H0}.
\end{itemize}

Analogously to Section~\ref{sec:idealdynamos}, Theorem~\ref{thm:hprinciple} will be proved by constructing a sequence iteratively using the next proposition. Let us note that, while $u_n \times B_n \to u \times B$ in $L_{t,x}^1$ is immediate in Theorem~\ref{Thm:lebesgue}, here this convergence must be verified, as in \cite{CheskidovLuo24}.

In this proposition, we fix a time interval $I_j$ of the partition in Theorem~\ref{thm:hprinciple} and, for ease of notation, omit the subindex $j$ in Sections~\ref{sec:dynamo:newterm}--\ref{sec:uxB}. This causes no issues, since the time-intermittent coefficients are compactly supported in $I_j$, allowing us to glue the solutions constructed on each time interval of the partition. Thus, the parameters $(p_0,q_0)$, $(\lambda,\mu)$, and $(\lambda_0,\mu_0)$ will later depend on $j$ in the proof of Theorem~\ref{thm:hprinciple}, as our goal is to obtain estimates for $p_j$, $m_j$ and $\delta_j$.

Since the time interval is fixed, we denote $p_j$, $m_j$, and $\delta_j$ simply by $p$, $m$, and $\delta$, respectively. Moreover, we split the exponent $p$ into two roles: $p$, governing the spatial integrability of the velocity, and $\gamma$, governing the time integrability of the magnetic field. This distinction is made solely to emphasize their respective roles in the estimates. Later, in Section~\ref{Proof:prop:dynamo}, we set $\gamma = p$.

We caution the reader that the H\"older exponent $p$ in Proposition \ref{prop:dynamo} is not the same as in Section \ref{sec:idealdynamos}, where we have just fixed it to be infinite. In any case, whenever we invoke the lemmas from Section \ref{sec:idealdynamos}, we will recall the corresponding estimates to avoid any confusion.

\begin{prop}\label{prop:dynamo}
Let $\Omega\subset\R^3$ be a smooth bounded domain, $I\subset (0,\infty)$ be a bounded interval, and let $p,\gamma,m\geq 1$ and $\delta>0$. For any continuous-in-$\epsilon$ family of smooth solutions $(u_0,B_0^\epsilon,R_0^\epsilon)$ of the \eqref{relaxeddynamo} equation on $I\times\Omega$, there is another continuous-in-$\epsilon$ family of smooth solutions $(u_1,B_1^\epsilon,R_1^\epsilon)$ which fulfills the estimates
\begin{align*}
\|u_1-u_0\|_{L^1(I,W^{1,p})}&\leq \delta,\\
\|B_1-B_0\|_{L^\gamma(I,C^m)}&\leq\delta,\\
\|R_1\|_{L^1(I\times\Omega)}&\leq\delta,\\
\|u_1\times B_1-u_0\times B_0\|_{L^1(I\times\Omega)}&\leq M\|R_0\|_{L^1(I\times\Omega)}.
\end{align*}
Moreover, 
\begin{equation}\label{prop:dynamo:boundary}
\operatorname{supp}(u_1-u_0),\,\operatorname{supp}(B_1-B_0),\,\operatorname{supp}R_1\subset\subset I\times\Omega.
\end{equation}
\end{prop}
\begin{proof}
See Section \ref{Proof:prop:dynamo}.
\end{proof}

\subsection{The new $(u_1,B_1,R_1)$}\label{sec:dynamo:newterm}

Given $(u_0,B_0,R_0)$, we
explain how the next step $(u_{\text{dyn}},B_{\text{dyn}},R_{\text{dyn}})$ is constructed. We will use the subindex ``dyn'' instead of 1 to avoid confusion with the ideal case.
 We decompose the new fields similarly to the ideal dynamo situation
$$
u_{\text{dyn}}:=u_0 + \underbrace{\tilde{u}_{\text{dyn}} + \tilde{u}_{\text{dyn},c}}_{\nabla\times\tilde{\psi}_{\text{dyn}}},
\quad\quad
B_{\text{dyn}}:=B_0 + \underbrace{\tilde{B}_{\text{dyn}} + \tilde{B}_{\text{dyn},c}}_{\nabla\times\tilde{A}_{\text{dyn}}} + \nabla\times H.
$$
Recall that these new perturbations are essentially the old ideal ones, tuned with the time cut-offs. 
They are defined in terms of the conjugate H\"older exponents $p_0,q_0$ and the parameters
\[
\renewcommand{\arraystretch}{1.2}
\begin{tabular}{r|c|c}
& \text{space} & \text{time} \\ \hline 
\text{concentration} & $\mu$ & $\mu_0$ \\ \hline
\text{oscillation} & $\lambda$ & $\lambda_0$
\end{tabular}
\]

As we did for $R_1$, we split
\begin{subequations}\label{R1:int;dif}
\begin{align}
\nabla\times R_{\text{dyn}}
&=(\partial_t-\epsilon\Delta) B_{\text{dyn}} - \nabla\times(u_{\text{dyn}}\times B_{\text{dyn}})\nonumber\\
&= \nabla\times(R_0-
\tilde{u}_{\text{dyn}}\times \tilde{B}_{\text{dyn}}+\partial_tH)\label{term:quad:int}\\
&+ \nabla\times(\partial_t\tilde{A}_{\text{dyn}}-\tilde{u}_{\text{dyn}}\times B_0 -u_0\times\tilde{B}_{\text{dyn}})\label{term:lin:int}\\
&-\nabla\times(\tilde{u}_{\text{dyn},c}\times (B_{\text{dyn}} -\nabla\times H)
+u_{\text{dyn}}\times(\tilde{B}_{\text{dyn},c}+\nabla\times H)
-\tilde{u}_{\text{dyn},c}\times \tilde{B}_{\text{dyn},c})
\label{term:cor:int}\\	
&+\epsilon\nabla\times(\nabla\times(\nabla\times(\tilde{A}_{\text{dyn}}+H)))\label{term:vis:int}\end{align}
\end{subequations}
and we define accordingly
\begin{equation}\label{Rdyn}
R_{\text{dyn}}
:=R_{\text{dyn}}^{\text{quad}}
+R_{\text{dyn}}^{\text{lin}}
+R_{\text{dyn}}^{\text{cor}}
+\epsilon R^{\text{dif}}.
\end{equation}

As we saw in Section \ref{sec:idealdynamos}, most of the terms are harmless, since they involve bounded negative powers of the parameters $(\mu,\mu_0,\lambda,\lambda_0)$. However, as discussed in Section \ref{sec:Sobolevversion}, there are also delicate terms, which a priori are bounded by positive powers of some of the parameters and thus must be balanced by those with negative exponents. In the following sections we estimate the delicate terms and, for completeness, briefly recall the bounds for the harmless ones.

\subsection{Choice of $\sigma$}
By continuity in $\epsilon$, the function
$$
g(t,x):=\max_{0\leq\epsilon\leq 1}|R_0^\epsilon(t,x)|
$$
is continuous on the closure of $I\times \Omega$. In particular, $g\in L^1(I\times \Omega)$. Hence, by the dominated convergence theorem, we deduce that there exists $\sigma>0$, independent of $0\leq \epsilon \leq 1$, such that
\begin{equation}\label{eq:choicesigma}
\|(1-\chi_\sigma^2)R_0^\epsilon\|_{L^1(I\times\Omega)}
\leq\|(1-\chi_\sigma^2)g\|_{L^1(I\times\Omega)}
\leq\frac{\delta}{8}.
\end{equation}

In the sequel, we omit the dependence on $\epsilon$ to simplify the notation. The key point is that all the estimates depend on derivatives of $(u_0,B_0^\epsilon,R_0^\epsilon)$. Since these derivatives depend continuously on $\epsilon$, we can ultimately choose all the parameters uniformly in $\epsilon$.

\subsection{Regularity of the velocity and magnetic  fields}

Firstly, we obtain a bound for the perturbation of the velocity field. 

\begin{lemma}[Sobolev bound of the velocity perturbation]\label{lemma:Sobolev:udyn}
There exists $M>0$ such that
$$
\|\nabla\times\tilde{\psi}_{\text{dyn}}\|_{L_t^1 W^{1,p}}
\leq M\frac{\lambda\mu^{1-\tfrac{1}{p}}}{{\mu_0}^{1-\tfrac{1}{p_0}}}.
$$
\end{lemma}
\begin{proof}
Since $\tilde{\psi}_{\text{dyn}}=(\mathbf{u}_0)_{\lambda_0}\tilde{\psi}$, the claim follows from the bounds 
$$
\|(\mathbf{u}_0)_{\lambda_0}\|_{L^1}\leq M\mu_0^{\frac{1}{p_0}-1},
\quad\quad
\|\nabla\times\tilde{\psi}\|_{L_t^\infty W^{1,p}}\leq M\lambda\mu^{1-\frac{1}{p}},
$$
which are obtained from Lemma \ref{lemma:buildingblock} applied to \eqref{def:timeintermittency:Bu}, and Lemma  \ref{lemma:Sobolev}, respectively. Recall that in this section the velocity is chosen to be independent of $R_0$, so the constant in Lemma~\ref{lemma:Sobolev} for the estimate of $\tilde{\psi}$ is universal.
\end{proof}
In view of Lemma \ref{lemma:Sobolev:udyn}, we can make the $L_t^1 W^{1,p}$-norm of $(u_1-u_0)$ small provided that
\begin{equation}\label{fastdynamo:cond2}
\lambda \mu^{1-\tfrac{1}{p}} \ll \mu_0^{\tfrac{1}{q_0}}.
\end{equation}
In particular, this condition requires $q_0<\infty$.

Similarly, we obtain a bound for the perturbation of the magnetic field. 

\begin{lemma}[Sobolev bound of the magnetic perturbation]\label{lemma:Sobolev:Bdyn}
There exists $C(m,\|R_0\|_{C^1})>0$ such that
$$
\|\nabla\times\tilde{A}_{\text{dyn}}\|_{L_t^\gamma C^m}
\leq C\frac{(\lambda\mu)^{m}\mu}{{\mu_0}^{\tfrac{1}{\gamma}-\tfrac{1}{q_0}}}.
$$
\end{lemma}
\begin{proof}
Since $\tilde{A}_{\text{dyn}}=(\mathbf{B}_0)_{\lambda_0}\tilde{A}$, the claim follows from the bounds 
$$
\|(\mathbf{B}_0)_{\lambda_0}\|_{L_t^\gamma}\leq M\mu_0^{\frac{1}{q_0}-\frac{1}{\gamma}},
\quad\quad
\|\nabla\times\tilde{A}\|_{L_t^\infty W^{m,\infty}}\leq C(\lambda\mu)^m \mu,
$$
which are obtained from Lemma \ref{lemma:buildingblock} applied to \eqref{def:timeintermittency:Bu}, and Lemma \ref{lemma:Sobolev}, respectively. 
\end{proof}
In view of Lemma \ref{lemma:Sobolev:Bdyn}, we can make the $L_t^\gamma C^m$-norm of $\nabla\times\tilde{A}_{\text{dyn}}$ small provided that
\begin{equation}\label{fastdynamo:cond1}
(\lambda\mu)^{m}\mu
\ll
{\mu_0}^{\tfrac{1}{\gamma}-\tfrac{1}{q_0}}.
\end{equation}
In particular, this condition requires $q_0>\gamma$.

The temporal corrector satisfies
$$
\|\nabla\times H\|_{L_t^\gamma C^m}
\leq\frac{M}{\lambda_0}\|\chi_\sigma^2R_0\|_{C^{m+1}},
$$
where we have applied that $\mathbf{H}_0$ is bounded (recall \eqref{eq:Hbounded}). Therefore, the temporal corrector is a harmless term and thus the condition \eqref{fastdynamo:cond1} suffices to make $(B_1-B_0)$ arbitrarily small in $L_t^\gamma C^m$.

\subsection{The quadratic error}\label{sec:quadraticerror:dif}
Recall that $R_{\text{dyn}}^{\text{quad}}$ is defined at the line \eqref{term:quad:int}. 
More precisely, since
$$
\tilde{u}_{\text{dyn}}\times\tilde{B}_{\text{dyn}}
=(\mathbf{u}_0\mathbf{B}_0)_{\lambda_0}(\tilde{u}\times\tilde{B}),
$$
and, using the definition of $\mathbf{H}_0$,
$$\partial_tH
=(\mathbf{u}_0\mathbf{B}_0-1)_{\lambda_0}\chi_\sigma^2 R_0
-\frac{1}{\lambda_0}(\mathbf{H}_0)_{\lambda_0}\chi_\sigma^2\partial_tR_0,
$$
we split
\begin{align*}
R_0-\tilde{u}_{\text{dyn}}\times \tilde{B}_{\text{dyn}} +\partial_tH
=
(1-\chi_\sigma^2)R_0
+(\mathbf{u}_0\mathbf{B}_0)_{\lambda_0}\chi_\sigma^2(R_0-\tilde{u}\times\tilde{B}) 
-\frac{1}{\lambda_0}(\mathbf{H}_0)_{\lambda_0}\chi_\sigma^2\partial_tR_0.
\end{align*}
This allows us to define
$$R_{\text{dyn}}^{\text{quad}}
:=(1-\chi_\sigma^2) R_0 + (\mathbf{u}_0\mathbf{B}_0)_{\lambda_0} \chi_\sigma^2 R^{\text{quad}}
-\frac{1}{\lambda_0}(\mathbf{H}_0)_{\lambda_0} \chi_\sigma^2 \partial_tR_0.$$
The first term was bounded in \eqref{eq:choicesigma}. For the second term, note that $R^{\text{quad}}$ was already estimated in \eqref{Rquad:estimate} (recall that $R^{\text{quad}} = R_1^{\text{osc}} + R^{\text{int}}$, where now $R^\chi = 0$). The last term can be made arbitrarily small by taking $\lambda_0$ sufficiently large. Therefore, by choosing the parameters large enough, we can ensure, for instance, that
$$
\|R_{\text{dyn}}^{\text{quad}}\|_{L^1_{t,x}} \leq \frac{3}{4}\delta.
$$  

\subsection{The linear error} Recall that the linear error $R_{\text{dyn}}^{\text{lin}}$ is defined at the second line \eqref{term:lin:int}. More precisely, by splitting
$$\partial_t\tilde{A}_{\text{dyn}}
=\partial_t(\mathbf{B}_0)_{\lambda_0}\tilde{A}
+(\mathbf{B}_0)_{\lambda_0}\partial_t\tilde{A},$$
we define
\begin{align*}
R_{\text{dyn}}^{\text{lin}}
&:=\partial_t(\mathbf{B}_0)_{\lambda_0}\tilde{A}
+(\mathbf{B}_0)_{\lambda_0}(\partial_t\tilde{A}
-u_0\times\tilde{B}) - (\mathbf{u}_0)_{\lambda_0}(\tilde{u}\times B_0).
\end{align*}
The first term is new, and following \cite{CheskidovLuo24}, we refer to it as the \textit{acceleration error}
$$R^{\text{acc}}
:=\partial_t(\mathbf{B}_0)_{\lambda_0}\tilde{A}.$$
\begin{lemma}[$L^1$-bound of $R^{\text{acc}}$]
There exists $C(R_0)>0$ such that
$$\|R^{\text{acc}}\|_{L_{t,x}^1}
\leq C \frac{\lambda_0\mu_0^{1/q_0}}{\lambda\mu}.$$
\end{lemma}
\begin{proof}
The claim follows from the bounds 
$$
\|\partial_t(\mathbf{B}_0)_{\lambda_0}\|_{L^1}\leq M\lambda_0\mu_0^{\tfrac{1}{q_0}},
\quad\quad
\|\tilde{A}\|_{L_t^\infty L^1}\leq\frac{C}{\lambda\mu},
$$
which are obtained from Lemma~\ref{lemma:buildingblock} applied to \eqref{def:timeintermittency:Bu}, and Lemma~\ref{lemma:Sobolev}, respectively.
\end{proof}

Therefore, the acceleration error can be made small in $L_{t,x}^1$ provided that
\begin{equation}\label{acc:cond:1}
\lambda_0 \mu_0^{1/q_0} \ll \lambda \mu .
\end{equation}
The remaining terms of $R_{\text{dyn}}^{\text{lin}}$ are harmless; that is, they can be made arbitrarily small in $L_{t,x}^1$ by first applying Lemma \ref{lemma:Rlin} to control the contributions from the ideal case, and then using Lemma \ref{lemma:buildingblock} applied to \eqref{def:timeintermittency:Bu} to gain an additional factor of $\mu_0^{-1/\max (p_0,q_0)}$.

\subsection{The corrector error}
For the third line in \eqref{term:cor:int}, we set
\begin{align*}
R_{\text{dyn}}^{\text{cor}}
&:=-\tilde{u}_{\text{dyn},c}\times (B_{\text{dyn}}-\nabla\times H)
-u_{\text{dyn}}\times(\tilde{B}_{\text{dyn},c}+\nabla\times H)
+\tilde{u}_{\text{dyn},c}\times \tilde{B}_{\text{dyn},c}.
\end{align*}
Using the extra factors $(\lambda\mu)^{-1}$, $(\lambda\mu)^{-1}$ and  $\lambda_0^{-1}$ of $\tilde{u}_{\text{dyn},c}$, $\tilde{B}_{\text{dyn},c}$ and $H$, respectively, one easily sees that the $L_{t,x}^1$-norm of the corrector error can be made arbitrarily small; in other words, $R_{\text{dyn}}^{\text{cor}}$ is harmless.

\subsection{The diffusion error}

Finally, from the fourth line \eqref{term:vis:int}, we set
$$R^{\text{dif}}:=\nabla\times\nabla\times(\tilde{A}_{\text{dyn}}+ H).$$
Let us bound the harmless term, which is analogous to Lemma \ref{lemma:Sobolev:Bdyn}.

\begin{lemma}[$L^1$-norm of $R^{\text{dif}}$]
There exists $C(R_0)>0$ such that
$$\|\nabla\times\nabla\times\tilde{A}_{\textup{dyn}}\|_{L_ {t,x}^1}
\leq C
\frac{\lambda\mu}{{\mu_0}^{1/p_0}}.$$
\end{lemma}
\begin{proof}
Since $\tilde{A}_{\text{dyn}}=(\mathbf{B}_0)_{\lambda_0}\tilde{A}$, the claim follows from the bounds
$$
\|(\mathbf{B}_0)_{\lambda_0}\|_{L^1}\leq M\mu_0^{\tfrac{1}{q_0}-1},
\quad\quad
\|\nabla\times\nabla\times\tilde{A}\|_{L_t^\infty L^1}\leq C\lambda\mu,
$$
which are obtained from Lemma \ref{lemma:buildingblock} applied to \eqref{def:timeintermittency:Bu}, and \ref{lemma:Sobolev}, respectively.
\end{proof}
As in the previous sections, 
the term $H$ can be made arbitrarily small in $L_{t,x}^1$ thanks to the factor $\lambda_0^{-1}$.
All in all, the $L_{t,x}^1$-norm of the viscous error can be made small provided that
\begin{equation}\label{viscous:cond:1}
\lambda \mu \ll \mu_0^{1/p_0}.
\end{equation}
As expected, one can easily check that \eqref{viscous:cond:1} is implied by the much stronger condition \eqref{fastdynamo:cond1}. Therefore, it will suffice to verify \eqref{fastdynamo:cond1}.

\subsection{The cross product}\label{sec:uxB}
In this section we control the $L_{t,x}^1$-norm of
$$
u_1\times B_1-u_0\times B_0
= (\tilde{u}_{\text{dyn}} + \tilde{u}_{\text{dyn}}^c) \times
(\tilde{B}_{\text{dyn}}+\tilde{B}_{\text{dyn},c}+\nabla\times H).
$$
As in the previous sections, all the terms except $\tilde{u}_{\text{dyn}}\times\tilde{B}_{\text{dyn}}$ can be made arbitrarily small by using the extra factors $(\lambda\mu)^{-1}$, $(\lambda\mu)^{-1}$ and  $\lambda_0^{-1}$ of $\tilde{u}_{\text{dyn},c}$, $\tilde{B}_{\text{dyn},c}$ and $H$, respectively. 

The remaining term can be estimated by applying the improved H\"older inequality (Proposition~\ref{prop:IHI}). 
This was already used in \eqref{eq:uB:Lebesgue}; namely, we can take the spatial oscillation $\lambda$ sufficiently large so that
$$
\|(\tilde{u}\times\tilde{B})(t)\|_{L^1}
\leq\|\tilde{u}(t)\|_{L^\infty}\|\tilde{B}(t)\|_{L^1}
\leq M\|R_0(t)\|_{L^1}.
$$
Then, using the smoothness of $t \mapsto \|R_0(t)\|_{L^1}$, we can argue as before and take the temporal oscillation $\lambda_0$ sufficiently large to ensure that
$$
\|\tilde{u}_{\text{dyn}}\times\tilde{B}_{\text{dyn}}\|_{L^1_{t,x}}
\leq M\|(\mathbf{u}_0\mathbf{B}_0)_{\lambda_0}\|R_0\|_{L^1}\|_{L_t^1}
\leq M\|R_0\|_{L_{t,x}^1},
$$
where we have applied that $\|\mathbf{u}_0\mathbf{B}_0\|_{L^1}=1$.

\subsection{Proof of Proposition \ref{prop:dynamo}}
\label{Proof:prop:dynamo}
In the previous sections, we have seen that proving Proposition \ref{prop:dynamo} reduces to showing that one can choose the parameters $(p_0,q_0)$, $(\lambda,\mu)$, and $(\lambda_0,\mu_0)$ satisfying the three conditions \eqref{fastdynamo:cond2}, \eqref{fastdynamo:cond1}, and \eqref{acc:cond:1}. To this end, we begin by comparing the concentration/oscillation parameters with $\lambda$ in terms of the corresponding exponents
$$
\lambda=\lambda^1,
\quad\quad
\mu=\lambda^{\gamma_\mu}, 
\quad\quad
\lambda_0=\lambda^{\gamma_{\lambda_0}},
\quad\quad
{\mu_0}=\lambda^{\gamma_{\mu_0}}.
$$
This allows us to write our three conditions in terms of algebraic relations between these exponents:
\begin{enumerate}[(1)]
    \item\label{cond:u} The $L_t^1W^{1,p}$-condition \eqref{fastdynamo:cond2} reads as
$$q_0(1+(1-1/p)\gamma_\mu) < \gamma_{\mu_0}.$$
    \item\label{cond:B} The $L_t^\gamma C^m$-condition \eqref{fastdynamo:cond1} reads as ($q_0>\gamma$)
$$\frac{q_0\gamma}{q_0-\gamma}(m+(m+1)\gamma_\mu)<\gamma_{\mu_0}.$$
    \item\label{cond:acc} The acceleration condition \eqref{acc:cond:1} reads as
$$\gamma_{\mu_0} <q_0(1 + \gamma_\mu - \gamma_{\lambda_0}).$$
\end{enumerate}

Next, we investigate whether there is a way to make all these conditions compatible.
Notice that there is no incompatibility between \eqref{fastdynamo:cond2} and \eqref{fastdynamo:cond1}, they only require taking $\gamma_{\mu_0}$ large enough.
Thus, we only need to check that these two conditions are compatible with \eqref{acc:cond:1}.

Firstly, 
\eqref{fastdynamo:cond2} and \eqref{acc:cond:1} are compatible as soon as
$$
1+(1-1/p)\gamma_\mu
<1 + \gamma_\mu - \gamma_{\lambda_0},
$$
or equivalently,
\begin{equation}\label{fastdynamo:compatibility23}
p\gamma_{\lambda_0} < \gamma_\mu.
\end{equation}

Secondly, \eqref{fastdynamo:cond1} and \eqref{acc:cond:1} are compatible if
$$
\frac{\gamma}{q_0-\gamma}(m+(m+1)\gamma_\mu)
<1 + \gamma_\mu - \gamma_{\lambda_0},
$$
or equivalently,
\begin{equation}\label{fastdynamo:compatibility13}
q_0>
\gamma\left(1+\frac{m+(m+1)\gamma_\mu}{1 + \gamma_\mu - \gamma_{\lambda_0}}\right).
\end{equation}

Therefore, for every $p,\gamma,m\geq 1$, by choosing $\gamma_\mu$ such that
$$
p\gamma_{\lambda_0}<\gamma_\mu<\infty,
$$
for some arbitrary $\gamma_{\lambda_0}\in\N$,
and then taking $q_0$ in the regime
$$
\gamma\left(1+\frac{m+(m+1)\gamma_\mu}{1 + \gamma_\mu - \gamma_{\lambda_0}}\right)
<q_0<\infty,
$$
there is room to take $\gamma_{\mu_0}$ satisfying the three conditions \ref{cond:u}, \ref{cond:B} and \ref{cond:acc}. 

Finally, recall that $\epsilon$ only appears in the definition of $R_{\text{dyn}}$ (equation \eqref{Rdyn}). Since $\epsilon \leq 1$, we can make the $L^1_{t,x}$-norm of $R_{\text{dyn}}$ arbitrarily small uniformly in $\epsilon$. Moreover, thanks to the continuity in $\epsilon$ of the family $(u_0,B_0^\epsilon,R_0^\epsilon)$, all the estimates are uniform in $\epsilon$, that is, the parameters can be taken independent of $\epsilon$. Thus, $u_1$ is also independent of the magnetic diffusivity.

\subsection{Proof of Theorem \ref{thm:hprinciple}}\label{Proof:thm:hprinciple}

Set 
$u_0:=\bar{u}$, $B_0:=\bar{B}=\nabla\times\bar{A}$ and $R:=(\partial_t-\epsilon\Delta)\bar{A}-\bar{u}\times\bar{B}$. Then, $(u_0,B_0,R_0)$ is a smooth solution to the \eqref{relaxeddynamo} equation in $[0,\infty)\times\Omega$. Given the partition $I_j$ and the sequences $p_j,m_j\geq 1$ and $\delta_j>0$, we define
$$
p_{j,n}:=np_j,
\quad\quad
m_{j,n}:=n m_j,
\quad\quad
\delta_{j,n}:=2^{-(n+1)}\min(|I_j|,|I_j|^{\nicefrac{1}{p_{j,n}}})\delta_j.
$$
We apply Proposition \ref{prop:dynamo} iteratively on each interval $I_j$ to construct a sequence of smooth solutions $(u_n,B_n,R_n)$ satisfying the inductive estimates
\begin{align*}
\|u_n-u_{n-1}\|_{L^1(I_j,W^{1,p_{j,n}})}&\leq\delta_{j,n},\\
\|B_n-B_{n-1}\|_{L^{p_{j,n}}(I_j,C^{m_{j,n}})}&\leq\delta_{j,n},\\
\|R_n\|_{L^1(I\times\Omega)}&\leq\delta_{j,n},\\
\|u_n\times B_n-u_{n-1}\times B_{n-1}\|_{L^1(I\times\Omega)}&\leq M\|R_{n-1}\|_{L^1(I\times\Omega)},
\end{align*}
for all indices $j$ and $n\geq 1$. 
Recall that the solutions are well-defined at the boundaries of the partition thanks to the support condition \eqref{prop:dynamo:boundary}. It follows that, for any given $p,m \geq 1$, the sequence $(u_n,B_n,R_n,u_n\times B_n)$ is Cauchy in $L_t^1W^{1,p}\times L_t^p C^m\times L_{t,x}^1\times L_{t,x}^1$, and therefore converges to some $(u,B,0,P)$. 
By taking a subsequence and dropping indices if necessary, we obtain $u_n \to u$, $B_n \to B$, and $u_n \times B_n \to P$ pointwise almost everywhere. Hence $P = u \times B$, and therefore $(u,B)$ is a weak solution to the \eqref{dynamo1} equation.

By applying the triangle inequality, using that the sequences $p_j \leq p_{j,n}$ and $m_j \leq m_{j,n}$ are increasing in $n$, and recalling the choice of $\delta_{j,n}$, we deduce that
\begin{align*}
\left(\dashint_{I_j}\|B^\epsilon-\bar{B}\|_{C^{m_j}}^{p_j}\dif t\right)^{\nicefrac{1}{p_j}}
&\leq\sum_{n\geq 1}
\left(\dashint_{I_j}\|B_n-B_{n-1}\|_{C^{m_j}}^{p_j}\dif t\right)^{\nicefrac{1}{p_j}}\\
&\leq\sum_{n\geq 1}
\left(\dashint_{I_j}\|B_n-B_{n-1}\|_{C^{m_{j,n}}}^{p_{j,n}}\dif t\right)^{\nicefrac{1}{p_{j,n}}}\\
&\leq\sum_{n\geq 1}\delta_{j,n}
|I_j|^{-\nicefrac{1}{p_{j,n}}}=\frac{\delta_j}{2}.
\end{align*}
The proof for the velocity is analogous.
This establishes \eqref{thm:hprinciple:barB}.

Finally, we note that the boundary conditions are preserved, since the perturbations are compactly supported in $\Omega$.

\subsection{Coarse-grained magnetic fields}\label{Proof:thm:fastdynamo}
After the $H$-principle the existence of dynamos boils down to the existence of coarse-grained evolutions. We label differently the intervals $I_j$ near $0$ and $\infty$. That is,  we set the notation
\begin{equation}\label{eq:partition}
0<\ldots<a_1<a_0=b_0<b_1<\ldots<\infty,
    \end{equation}
    where the intervals $(a_{j+1},a_j)$ and $(b_j,b_{j+1})$ accumulate at $0$ and $\infty$, respectively. 
  
\begin{thm}[Coarse-grained magnetic fields]\label{thm:fastdynamo}
Let $\Omega\subset\R^3$ be a smooth bounded domain, and let $\bar{\mathcal{E}}:[0,\infty)\to[0,\infty)$ be a smooth energy profile. Consider a partition of $(0,\infty)$ of the form \eqref{eq:partition}. Fix two non-decreasing sequences $p_j,m_j \geq 1$, and another sequence $\delta_j\searrow 0$. Then there exists
$u$ and $B^\epsilon$ as in \eqref{thm:hprinciple:u} and \eqref{thm:hprinciple:B}, such that:
\begin{enumerate}[(a)]
    \item\label{thm:fastdynamo:(a)} There exists a smooth solenoidal field $\bar{B}$ on $[0,\infty)\times\bar{\Omega}$ such that $\mathcal{E}(\bar{B})=\bar{\mathcal{E}}$ and
$$
\left(\dashint_0^{a_j}
\|(B^\epsilon-\bar{B})(t)\|_{C^{m_j}}^{p_j} \, dt 
\right)^{\nicefrac{1}{p_j}}\leq\delta_j,
$$
for all $j$, uniformly in $0<\epsilon\leq 1$. 
    \item\label{thm:fastdynamo:(b)} The magnetic energies converge as $t\to\infty$ in the sense that
$$
\left(\dashint_{b_{j}}^{b_{j+1}} |\mathcal{E}(B^\epsilon(t))-\bar{\mathcal{E}}(t)|^{p_j}\dif t\right)^{\nicefrac{1}{p_j}}
\leq\delta_j,
$$
for all $j$, uniformly in $0<\epsilon\leq 1$.
\begin{comment}
In particular, if 
$$
\liminf_{t\to\infty}\frac{\log\bar{\mathcal{E}}(t)}{t}>0,
$$
then $u$ is a fast dynamo and if \[ \lim_{T\to \infty}\frac{1}{T}  \int_1^T\frac{\log\bar{\mathcal{E}}(t)}{t}dt =\infty \]
$u$ is an \it{ergodic} fast dynamo.
\end{comment}
\end{enumerate}
\end{thm}
\begin{proof}
Let $B \in C_c^\infty(\Omega)$ be a solenoidal vector field with $\mathcal{E}(B)=1$.
We apply Theorem \ref{thm:hprinciple} to $\bar{u}=0$ and 
$\bar{B}(t,x) = \sqrt{2\bar{\mathcal{E}}(t)} B(x)$,
using the same partition $I_j$, and for some sequences $p_j',m_j'\geq 1$, and $\delta_j' > 0$, to be determined.
In fact, we set $p_j' = 2p_j$ and $m_j'=m_j$.

\textit{Proof of \ref{thm:fastdynamo:(a)}}.
Let us abbreviate $f_j=\|B^\epsilon-\bar{B}\|_{C^{m_j}}$. Since $f_j$ is non-decreasing, we deduce that 
\begin{equation}\label{thm:fastdynamo:1}
\dashint_0^{a_j}f_j^{p_j}\dif t
=\sum_{i\geq j}\frac{a_i-a_{i+1}}{a_j}\dashint_{a_{i+1}}^{a_i}f_j^{p_j}\dif t
\leq
\sum_{i\geq j}\dashint_{a_{i+1}}^{a_i}f_i^{p_j}\dif t,
\end{equation}
where we also used that $a_i - a_{i+1} \leq a_i \leq a_j$ for all $i \geq j$.
Since $p_i'=2p_i$ is non-decreasing, the Jensen inequality and \eqref{thm:hprinciple:barB} yield
\begin{equation}\label{thm:fastdynamo:2}
\eqref{thm:fastdynamo:1}
\leq
\sum_{i\geq j}
\left(\dashint_{a_{i+1}}^{a_i}f_i^{p_i'}\dif t\right)^{\frac{p_j}{p_i'}}
\leq
\sum_{i\geq j}(\delta_i')^{p_j}.
\end{equation}
Finally, by choosing $\delta_i' \leq 1$ in advance, smaller than the telescoping series
$(\delta_i^{p_i} - \delta_{i+1}^{p_{i+1}})$, we obtain
$$
\eqref{thm:fastdynamo:2}
\leq\sum_{i\geq j} \delta_i'
\leq\delta_j^{p_j}.
$$

\textit{Proof of \ref{thm:fastdynamo:(b)}}.
Let us denote $I_j=(b_j,b_{j+1})$.
By applying  the reverse triangle inequality and \eqref{thm:hprinciple:barB}, it follows that 
\begin{equation}\label{thm:fastdynamo:b:1}
\left(\dashint_{I_j}
|\|B^\epsilon\|_{L^2}-\|\bar{B}\|_{L^2}|^{p_j'}\dif 
t\right)^{\nicefrac{1}{p_j'}}
\leq
\left(\dashint_{I_j}
\|B^\epsilon-\bar{B}\|_{L^2}^{p_j'}\dif 
t\right)^{\nicefrac{1}{p_j'}}
\leq|\Omega|^{\nicefrac{1}{2}}\delta_j'.
\end{equation}
This estimate implies that
\begin{align}
\left(\dashint_{I_j}|\mathcal{E}(B^\epsilon)-\mathcal{E}(\bar{B})|^{p_j}\dif t\right)^{\nicefrac{1}{p_j}}
&\leq\frac{1}{2}\left(\dashint_{I_j}|\|B^\epsilon\|_{L^2}
+\|\bar{B}\|_{L^2}|^{2\gamma_j}\right)^{\nicefrac{1}{2p_j}}
\left(\dashint_{I_j}
|\|B^\epsilon\|_{L^2}-\|\bar{B}\|_{L^2}|^{2p_j}\dif 
t\right)^{\nicefrac{1}{2p_j}}\nonumber\\
&\leq
\frac{1}{2}\left(2\left(\dashint_{I_j}\|\bar{B}\|_{L^2}^{2p_j}\right)^{\nicefrac{1}{2p_j}}+|\Omega|^{\nicefrac{1}{2}}\delta_j'\right)|\Omega|^{\nicefrac{1}{2}}\delta_j'.\label{thm:fastdynamo:b:2}
\end{align}
Therefore, we may choose $\delta_j'$ from the outset, depending on $(B,\bar{\mathcal{E}}, I_j, p_j, \delta_j)$, so that $\eqref{thm:fastdynamo:b:2}\leq\delta_j$. Recall that $\mathcal{E}(\bar{B}) = \bar{\mathcal{E}}$.
\end{proof}

\subsubsection{Proof of Theorem~\ref{thm:Main}}
\label{sec:proof:thm:main}
The existence of a limsup fast dynamo is immediate. On each time interval $I_j$, the property \ref{thm:fastdynamo:(b)} ensures the existence of Lebesgue points $t_j \in I_j$ such that $\mathcal{E}(B^\epsilon(t_j)) \geq \bar{\mathcal{E}}(t_j) - \delta_j$. 
Since $\delta_j \to 0$, if $\bar{\mathcal{E}}$ grows (super)exponentially, then $u$ satisfies the dynamo condition \eqref{e:gamma} uniformly in $0 < \epsilon \leq 1$.

The existence of an a.s.~dynamo follows by choosing, for instance, $\bar{\mathcal{E}}(t) = e^{\bar{\gamma} t}$ for some $\bar{\gamma} > 0$. For simplicity, we take $b_j = j$, that is, $I_j = [j,j+1)$. The sequence $\delta_j$ will be specified at the end.
For brevity, we write $\mathcal{E}(B^\epsilon(t)) = \mathcal{E}(t)$. By Chebyshev’s inequality and property~\ref{thm:fastdynamo:(b)},
\[
|\{ t\in I_j:|\mathcal{E}(t)-\bar{\mathcal{E}}(t)|\ge \lambda_j|\}|\le \frac{\delta_j}{\lambda_j},\]
for some $\lambda_j>0$ to be chosen. Splitting $\mathcal{E}(t)=e^{\bar{\gamma}t}+(\mathcal{E}(t)-\bar{\mathcal{E}}(t))$, we deduce that
\[
|\{ t\in I_j:\mathcal{E}(t)\geq e^{\bar{\gamma}t}-\lambda_j\}|
\geq 1-\frac{\delta_j}{\lambda_j}.\]
Given an arbitrary $0 < \gamma < \bar{\gamma}$, we choose $0 < \lambda_j \le e^{\bar{\gamma} j} - e^{\gamma j}$. Then $\lambda_j \le e^{\bar{\gamma} t} - e^{\gamma t}$ for all $t \ge j$, which implies $\frac{1}{t}\log(e^{\bar{\gamma}t}-\lambda_j)\geq\gamma$. Hence
\[
|\{ t\in I_j:\frac{1}{t}\log(\mathcal{E}(t))\geq \gamma\}|
\geq 1-\frac{\delta_j}{\lambda_j}.\]
Summing over all intervals, we obtain
\begin{align*}
\frac{1}{T}|\{t\in(0,T]\,:\,
\frac{1}{t}\log(\mathcal{E}(t))\geq\gamma\}|
&\geq\frac{1}{T}\sum_{j=1}^{\lfloor T \rfloor}
|\{t\in I_j\,:\,
\frac{1}{t}\log(\mathcal{E}(t))\geq\gamma\}|\\
&\geq 
\frac{\lfloor T\rfloor}{T}
\left(1-\frac{1}{\lfloor T \rfloor}\sum_{j=1}^{\lfloor T \rfloor}\frac{\delta_j}{\lambda_j}\right).
\end{align*}
Since $\frac{\lfloor T \rfloor}{T} \to 1$ as $T\to\infty$, by choosing $\delta_j$ from the outset so that $\frac{1}{\lfloor T\rfloor}\sum^{\lfloor T\rfloor}_{j=1}\frac{\delta_j}{\lambda_j}\to 0$, the claim follows.

\begin{comment}
Our fast dynamo is based on intermittency and thus it is an intermittent fast dynamo. As such it is a limsup fast dynamo in the terminology of \cite{SorellaVillingerpp}. Our dynamo is
    ergodic in the sense that the Ces\`aro means in time, that is ergodic average of the type 

\[\lim_{T \to \infty} \frac{1}{T}\int_0^T \frac{\log \mathcal{E}(B^\epsilon(t))}{t} dt =\infty \]
Namely, we could take intervals $I_j=[j,j+1]$,  $\delta_j=\frac{1}{j}$ and $p_j=1$. Then, \textcolor{red}{(some further comment seems to be needed here; the enemy would be possible times where $0 < \mathcal{E}(B^\epsilon(t)) \ll 1$,)}
\[ \limsup_{T\to \infty}\frac{1}{T} \int^T_{1}|\frac{\log \mathcal{E}(B^\epsilon(t))}{t} - \frac{\log  \bar{\mathcal{E}}(t)}{t}|dt \le \limsup_{T\to \infty} \frac{1}{T}\sum^{[T]+1}_{j=1} \frac{1}{j}=0 \]}
\end{comment}

\subsection{Turbulent coarse-grained magnetic fields} \label{sec:Turbulent coarse-grained magnetic fields}
As in the case of fast dynamos, multiscale dynamos are obtained by constructing coarse-grained evolutions with appropriate growth of their magnetic energy modes.

\begin{thm}[Turbulent coarse-grained magnetic fields]\label{thm:turbulentdynamo}
Let $\Omega \subset \mathbb{R}^3$ be a smooth bounded domain. Let $\bar{\mathcal{E}}_k : [0,\infty) \to [0,\infty)$ be a sequence of smooth energy profiles satisfying
\begin{equation}\label{thm:turbulentdynamo:energies}
\sum_{k=1}^\infty \partial_t^n \sqrt{\bar{\mathcal{E}}_k(t)}\, k^{\nicefrac{m}{3}} < \infty,
\end{equation}
for all $n,m \geq 0$ and $t \geq 0$.
Then there exist $u$ and $B^\epsilon$ as in \eqref{thm:hprinciple:u} and \eqref{thm:hprinciple:B}, and a smooth solenoidal field $\bar{B}$ on $[0,\infty)\times\bar{\Omega}$ with $\mathcal{E}_k(\bar{B}) = \bar{\mathcal{E}}_k$, satisfying \ref{thm:fastdynamo:(a)} and the following refinement of \ref{thm:fastdynamo:(b)}:
$$
\left(\dashint_{b_j}^{b_{j+1}} |\mathcal{E}_k(B^\epsilon(t)) - \bar{\mathcal{E}}_k(t)|^{p_j}\, \dif t\right)^{\nicefrac{1}{p_j}}
\leq \delta_j,
$$
for all $j$ and $k$, uniformly in $0 < \epsilon \leq 1$.
\end{thm}

\begin{proof}
Condition \eqref{thm:turbulentdynamo:energies} is imposed to ensure the smoothness of the coarse-grained magnetic field. Namely, for the ad hoc defined field
$$
\bar{B}(t,x)=\sum_k \sqrt{2\bar{\mathcal{E}}_k(t)}\, B_k(x),
$$
the regularity and bounds of the eigenfields established in Corollary~\ref{corappendix:magneticfields} yield
$$
\|\partial_t^n \bar{B}(t)\|_{H^m}
\leq \sqrt{2}\sum_k \partial_t^n \sqrt{\bar{\mathcal{E}}_k(t)} \|B_k\|_{H^m}
\leq C_{\Omega,m} \sum_k \partial_t^n \sqrt{\bar{\mathcal{E}}_k(t)} k^{\nicefrac{m}{3}} < \infty,
$$
for all $n,m \geq 0$ and $t \geq 0$. Hence, $\bar{B}$ is a smooth on $[0,\infty)\times\bar{\Omega}$. 
Moreover, $\bar{B}$ satisfies the same boundary conditions as the eigenfields $B_k$, and thus fulfills the assumptions of Theorem~\ref{thm:hprinciple}. 

Let $(u,B^\epsilon)$ denote the solutions to the \eqref{dynamo1} equation such that $\mathcal{E}(B^\epsilon)$ approximates $\mathcal{E}(\bar{B})$. By applying the Cauchy–Schwarz inequality,
$$
|\langle B^\epsilon - \bar{B}, B_k \rangle_{L^2}|
\leq \|B^\epsilon - \bar{B}\|_{L^2},
$$
we can repeat the argument in \eqref{thm:fastdynamo:b:1}–\eqref{thm:fastdynamo:b:2} to conclude the proof.
\end{proof}

\subsubsection{Proof of Theorem~\ref{thm:turbulent}}
\label{sec:proof:thm:turbulent}

Notice that \eqref{thm:turbulentdynamo:energies} is automatically satisfied if only finitely many of the energy profiles are nonzero. More generally, given $K\subset\mathbb{N}$ (not necessarily finite), we may fix $\bar{\gamma}>0$ and define
$$
\bar{\mathcal{E}}_k(t)=e^{\bar{\gamma}t-k},
$$
for $k\in K$,
and zero otherwise. This choice satisfies \eqref{thm:turbulentdynamo:energies}. Since $\langle \bar{B}, B_k \rangle_{L^2}^2 = 2\bar{\mathcal{E}}_k$, we have
$$
\lim_{t\to\infty}\frac{\log\mathcal{E}_k(\bar{B}(t))}{t}=\bar{\gamma}>0,$$
for all $k\in K$,
and zero otherwise.
By selecting $K$ appropriately, Theorem \ref{thm:turbulentdynamo} yields the existence of both small-scale and large-scale (fast) dynamos, as well as turbulent (fast) dynamos. 

Finally, notice that the multiscale dynamo action could also be prescribed at almost all times as in the proof of Theorem~\ref{thm:Main} in Section \ref{sec:proof:thm:main}.

\subsection{Taylor conjecture}\label{sec:Taylor} 

Recall that, in the context of plasma relaxation, one prescribes initial data $B^\circ$ together with the long-time behavior, while requiring the boundary flux to vanish at all times. Since the boundary conditions differ at the initial and final times, a compatibility condition must be imposed to ensure that the corresponding solutions satisfy that the vectors $B^\epsilon$, $\nabla\times B^\epsilon$, and $n$ are linearly dependent on $\partial\Omega$. This condition guarantees the vanishing of the boundary flux in the energy estimate \eqref{eq:energyestimate}. Note that all these conditions are trivial in the periodic setting $\mathbb{T}^3$.

\begin{thm}\label{thm:Beltrami}
Let $\Omega\subset\R^3$ be a smooth bounded domain.  
Let $\bar{u}$ be a smooth solenoidal field, $B^\circ$ be a smooth initial field in $H_{\mathrm{pc}}(\Omega)$, and let $\mathscr{B}$ denote a Beltrami field.
Assume that $B^\circ$, $\mathscr{B}$ and $n$ are linearly dependent on $\partial\Omega$.
Consider a partition of $(0,\infty)$ of the form \eqref{eq:partition}. Fix two non-decreasing sequences $p_j,m_j \geq 1$, and another sequence $\delta_j\searrow 0$. Then there exist
$u$ and $B^\epsilon$ as in \eqref{thm:hprinciple:u} and \eqref{thm:hprinciple:B}, such that:
\begin{itemize}
    \item {$u$ and $B^\epsilon$ solve \eqref{dynamo1} and $(\nabla \times B^\epsilon) \times B^\epsilon \cdot n = 0$ on $\partial \Omega$}.

    \item The velocity field $u$ remains close to $\bar{u}$ in the sense that
    $$
    \dashint_{I_j}\|(u-\bar{u})(t)\|_{W^{1,p_j}}\dif t\leq\delta_j,
    $$
    for all $j$.
    \item The magnetic field $B$ satisfies that $B|_{t=0}=B^\circ$ in the Lebesgue-Bochner sense and it approaches the Beltrami field $\mathscr{B}$ as $t\to\infty$ in the sense that
    $$
    \left(\dashint_{b_j}^{b_{j+1}}\|B(t)-\mathscr{B}\|_{C^{m_j}}^{p_j}\dif t\right)^{\nicefrac{1}{p_j}}
    \leq\delta_j,
    $$
    for all $j$.
\end{itemize}
\end{thm}
\begin{proof}
As explained in Section~\ref{sec:generalizations}, the $H$-principle ensures that the solution $B$ satisfies the same boundary conditions as the corresponding coarse-grained magnetic field. Therefore, it remains to construct a coarse-grained solution $\bar{B}$ with the desired behavior at $t=0$, as $t\to\infty$, and on $\partial\Omega$.

Take a smooth cutoff function $\chi:[0,\infty)\to[0,1]$ with $\chi=1$ on $[0,1]$ and $\chi=0$ on $[2,\infty)$ and declare
$$
\bar{B}=\chi(t)B^\circ + (1-\chi(t))\mathscr{B}.
$$
Let us check that $\bar{B}$, $\nabla\times\bar{B}$ and $n$ are linearly dependent on $\partial\Omega$ computing their mixed product.  Since $\nabla\times B^\circ\parallel n$ and $\nabla\times\mathscr{B}\parallel\mathscr{B}$, we have
$$
(\bar{B}\times n)\cdot\nabla\times\bar{B}
=\chi(1-\chi)(B^\circ\times n)\cdot\nabla\times\mathscr{B}=\lambda \chi(1-\chi) (B^\circ\times n)\cdot \mathscr{B}=0,
$$
where we have used that  $\nabla\times\mathscr{B}=\lambda\mathscr{B}$, and that $B^\circ$, $\mathscr{B}$ and $n$ are linearly dependent by assumption. Thus 
the boundary flux for $\bar B$ vanishes. Then Theorem~\ref{thm:hprinciple} yields the desired solutions with the claimed properties. 
\end{proof}

\section{Geometric transport equation}\label{sec:GTE}

In this section we extend the results of Sections~\ref{sec:idealdynamos} and~\ref{sec:fastdynamos} to the \eqref{GTE}. For the sake of conciseness, we explain the construction, highlighting the main differences, and present the estimates for the dangerous terms, while briefly recalling the harmless ones. We refer the reader to \cite{Federer} for a comprehensive treatment of multilinear algebra and currents, and to \cite{Krantz-Parks} for a more beginner-friendly introduction.

\subsection{Multilinear algebra}
We first recall some basic notions and notation from multilinear algebra. We define a multi-index of order $k$ among $d$ elements (shortly, $k$-index) any ordered subset of $\{1,\ldots,d\}$ with $k$ elements, and we indicate it by $\alpha=(\alpha_1,\ldots,\alpha_k)$, where $\alpha_i\in\{1\ldots,d\}$ and $\alpha_1<\ldots<\alpha_k$. We indicate the space of all $k$-indices on $d$ elements by $I(d,k)$. We define the space of $k$-covectors in $\R^d$ as the span of all elements $d x_\alpha:=dx_{\alpha_1}\wedge\ldots\wedge dx_{\alpha_k}$ among all multi-indices $\alpha\in I(d,k)$, where $\wedge$ denotes the wedge product. We denote the corresponding space by $\bigwedge^k\R^d$ (superscript $k$). A differential $k$-form is an element of the type $\omega(x)=\sum_{\alpha\in I(d,k)} \omega_\alpha(x) dx_\alpha$, where $\omega_\alpha(x)$ are real-valued functions on $\R^d$ (the coefficients of $\omega$).

Similarly, we define the space of $k$-vectors in $\R^d$ as the span of all elements $e_\alpha:=e_{\alpha_1}\wedge \ldots \wedge e_{\alpha_k}$ among all multi-indices $\alpha\in I(d,k)$, and we denote the corresponding space by $\bigwedge_k\R^d$ (subscript $k$). The spaces $\bigwedge_k\R^d$ and $\bigwedge^k\R^d$ are naturally dual to each other, where the duality pairing is given as follows on the basis vectors and then extended by linearity: for every pair of multi-indices $\alpha,\beta\in I(d,k)$
\[
\langle dx_\alpha,e_\beta\rangle =\begin{cases}
    1 & \text{if } \alpha=\beta\\
    0 & \text{otherwise}
\end{cases}.
\]

\subsection{The contraction operator} 
An important role in the computations that will follow is played by the contraction operator, which can be thought of as the dual to the wedge product: the contraction with $dx_i$ takes a $(k+1)$-vector $\tau$ and turns it into the $k$-vector $\tau\llcorner dx_i$ that is defined by its action on all $k$-forms as follows
\[
\langle  \beta, \tau \llcorner dx_i \rangle:=\langle   dx_i \wedge \beta,\tau  \rangle
\]
for every $(k+1)$-vector $\tau$ and every $k$-covector $\alpha$, and every $i=1,\ldots, d$. More concretely, one can think of the contraction with $dx_i$ as an operation that removes the $i$-th components from the $(k+1)$-vector $\tau$, up to a sign that counts the number of swaps to get to that component. As a concrete example we have the following: given a multi-index $\alpha=(\alpha_1,\ldots, \alpha_k)$
\[
(e_{\alpha_1}\wedge\ldots\wedge e_{\alpha_k})\llcorner dx_i=\begin{cases}
    (-1)^{j-1}e_{\alpha_1}\wedge\ldots\wedge \widehat{e_{\alpha_j}}\wedge\ldots\wedge e_{\alpha_k} & \text{if $i=\alpha_j$ for some $j$}\\
    0 & \text{otherwise}
\end{cases}
\]
The contraction of a general $(k+1)$-vector with a a general $1$-covector can be defined by multi-linearity writing both in the respective bases $\{e_\alpha\}_{\alpha\in I(d,k+1)}$ and $\{dx_i\}_{i=1,\ldots,d}$.

\subsection{Multi-index notation}
Given a vector $v \in \mathbb{R}^d$ and a multi-index $i\in I(d,m) $ of size $0 \leq m \leq d$, we denote
$$
v_i=(v_{i_1},\ldots,v_{i_m}),
\quad\quad
v^i=(v_{j_1},\ldots,v_{j_{d-m}}),
$$
where $j$ denotes the complement of $i$. For instance, in three dimensions we have $v_{(2,3)} = (v_2,v_3)$ and $v^{(2,3)} = v_1$. By convention, we set $v_i = 1$ if $m = 0$, that is, if $i = \emptyset$, and $v^i = 1$ if $m = d$.

\subsection{Currents}
The space of $k$-currents is defined as the dual of the space of compactly supported smooth differential $k$-forms, with the appropriate topology (analogous to the topology that one puts on test functions to define distributions). We will only use currents with \textit{finite mass}, in which case one can think of a current $T$ as a finite measure which takes values in the space of $k$-vectors: $T=\tau\nu$, where $\nu$ is a scalar measure and $\tau\in L^1_\nu(\R^d,\bigwedge_k\R^d)$. In this case the action of $T=\tau \nu$ on a differential $k$-form $\omega$ is given by
\[
\langle T,\omega\rangle =\int \langle \omega(x),\tau(x)\rangle d\nu(x)
\]
where $\langle\omega(x),\tau(x)\rangle$ denotes the duality between $k$-covectors and $k$-vectors. One should keep in mind that $k$-currents represent a generalization of oriented smooth $k$-dimensional manifolds. Indeed given such a manifold $M$ one can define an associated $k$-current $T_M$ by 
\[
\langle T_M,\omega\rangle:=\int_M \omega
\]
where the integration is the classical integration of a differential $k$-form on an oriented smooth $k$-dimensional manifold. 

\subsection{The boundary operator}
With the previous discussion in mind, the definition of the boundary operator as the dual of the exterior derivative is very natural:
$$
\langle\partial T,\omega\rangle :=
\langle T,d\omega\rangle
$$
for any $k$-current $T$ and $(k-1)$-form $\omega$, and $k=1,\ldots,d$. Notice that $\partial T$ is a $(k-1)$-current, while $d\omega$ is a $k$-form. Indeed, by Stokes' theorem this definition of boundary is consistent with the classical notion of boundary for smooth manifolds, or in other words, $\partial T_M=T_{\partial M}$ (up to a sign that might depend on Stokes' theorem or on how the orientation on $\partial M$ is defined). By convention $\partial T=0$ for $0$-currents. In particular, for the \eqref{transport} equation the condition $\partial\rho =0$ holds trivially. For a current $T=\tau\nu$ we indicate by $T\llcorner dx_j$ the $(k-1)$-current defined by $(\tau\llcorner dx_j)\nu$.

From \cite[4.1.7]{Federer} we obtain the following formulas: if $T$ is a $k$-current, with $1\le k\le d$, then
\begin{equation}\label{eq:partial_T_Federer}
\partial T=-\sum_{j=1}^d (\partial_j T)\llcorner dx_j.
\end{equation}
Moreover, if $f:\R^d\to \R$ is a sufficiently regular function ($C^1$ is enough) and $T$ is a $k$-current, then
\begin{equation}\label{eq:partial_fT_Federer}
\partial(fT)=f\partial T-\sum_{j=1}^d (\partial_j f)\wedge(T\llcorner dx_j),
\end{equation}
and if $u:\R^d\to\R^d$ is a sufficiently regular vector field, then
\begin{equation}\label{eq:partial_bT_Federer}
\partial(u\wedge T)=-\sum_{j=1}^d \big((\partial_j u)\wedge T+u\wedge (\partial_j T)\big)\llcorner dx_j.
\end{equation}

\subsection{Building blocks}

We take a smooth cutoff function $\varphi:(0,1)^{d-(k+1)}\to\R$ satisfying
$$
\int\varphi=0,
\quad\quad
\int\varphi^2=1.
$$
For instance, we could simply take $\varphi(y)=\prod_{j=1}^{d-(k+1)}\phi(y_j)$.
For any $\mu\geq 1$ we define the concentration $\varphi_\mu(y)=\varphi(\mu y)$ for $y\in (0,1)^{d-(k+1)}$. Next, we extend $\varphi_\mu$ periodically into $\T^{^{d-(k+1)}}$. Then, for any $\lambda\in\N$ we define the oscillation $(\varphi_\mu)_\lambda(y)=\varphi_\mu(\lambda y)$ where now $y\in\T^{d-(k+1)}$.
For any given multi-index
$$\alpha=(\alpha_1,\ldots ,\alpha_{k+1})\in I (d,k+1),$$
we define the spatial-intermittent shear flows
\begin{subequations}\label{eq:def_space_intermittency}
\begin{align}
\mathbf{u}_{\alpha}
&=\mu^{\frac{d-(k+1)}{p}}\varphi_\mu(x^\alpha)e_{\alpha_1},\label{eq:def_space_intermittency:1}\\
\mathbf{T}_{\alpha}
&=\mu^{\frac{d-(k+1)}{q}}\varphi_\mu(x^\alpha)e_{\alpha^1}.\label{eq:def_space_intermittency:2}
\end{align}
\end{subequations}
where recall that $\alpha^1=(\alpha_2,\ldots,\alpha_{k+1})$. 
Notice that $e_{\alpha_1}\wedge e_{\alpha^1}=e_\alpha$.

Next, we deduce the following property, which generalizes Proposition \ref{prop:steadydynamo}.
\begin{prop}
The pair $(\mathbf{u}_{\alpha},\mathbf{T}_{\alpha})$ is a stationary solution to the GTE, that is
\begin{align*}
\partial(\mathbf{u}_{\alpha}\wedge\mathbf{T}_{\alpha})&=0,\\
\nabla\cdot\mathbf{u}=0,\,
\partial\mathbf{T}_{\alpha}&=0.
\end{align*}
Moreover, it satisfies
$$
\dashint_{\T_\ell^d}\mathbf{u}_{\alpha}\dif x =
0,
\quad\quad
\dashint_{\T_\ell^d}\mathbf{T}_{\alpha}\dif x 
=0,\quad\quad
\dashint_{\T_\ell^d}\mathbf{u}_{\alpha}\wedge \mathbf{T}_{\alpha}\dif x
=e_{\alpha}.
$$
\end{prop}

\begin{proof}
    To prove the first part we use formulas \eqref{eq:partial_T_Federer} and \eqref{eq:partial_bT_Federer}: if $j\in \alpha$ then $\partial_j \mathbf{u}_\alpha=0$ and $\partial_j \mathbf{T}_\alpha=0$, since $\mathbf{u}_\alpha$ and $\mathbf{T}_\alpha$ only depend on the variables $x^\alpha$. If on the other hand $j\not\in\alpha$, then $e_\alpha\llcorner dx_j=0$ and $e^{\alpha_0}\llcorner dx_j=0$. In any case all the terms are zero.

    The zero-average property of $\mathbf{T}_\alpha$ descends from the same property for $\varphi$ (and thus for $\varphi_\mu$). Finally,
    \[
    \mathbf{u}_\alpha\wedge \mathbf{T}_\alpha=\mu^{d-(k+1)}\varphi_\mu(x^\alpha)^2 e_\alpha
    \]
    and after a change of variables we use that $\int \varphi^2=1$ to obtain the thesis.
\end{proof}

\subsection{The potential}\label{sec:potential}
As anticipated in Section~\ref{sec:localized:GTE}, we now introduce the general potential used to correct both the incompressibility and the oscillatory error.

\begin{prop}\label{prop:potential:GTE}
	Given $g:\T^{d-(k+1)}\to\R$ with $\int g=0$ we define $$g_\alpha(x)=g(x^\alpha)$$
	and also $G_{\alpha}=\mathrm{div}^{-1}(g_\alpha)$, that is,
	$$
	G_{\alpha,j}(x)
	=-i\sum_{\xi\neq 0}\hat{g}_\alpha(\xi)\frac{\xi_j}{|\xi|^2}e^{i\xi\cdot x}.
	$$
	Let $\beta\subset\alpha$ and consider
	$$U=f (g_\alpha)_\lambda e_\beta.$$
	Then, the pair
	$$
	V=\frac{1}{\lambda}\sum_{i,j=1}^d\partial_i f (G_{\alpha,j})_\lambda (e_\beta\wedge e_j)
	\mathbin{\llcorner} d x_i
	$$
	and
	$$
	W=\frac{1}{\lambda}\sum_{j=1}^d
	f(G_{\alpha,j})_\lambda(e_\beta\wedge e_j)
	$$
	satisfy
	$$U+V=\partial W.$$
\end{prop}
\begin{proof} 
	We compute 
	\begin{subequations}
		\begin{align}
			\partial W
			&=\frac{1}{\lambda}\sum_{i,j=1}^d
			\partial_i(f(G_{\alpha,j})_\lambda)(e_\beta\wedge e_j)
			\mathbin{\llcorner} d x_i\\
			&=\sum_{i,j=1}^d
			f(\partial_iG_{\alpha,j})_\lambda(e_\beta\wedge e_j)
			\mathbin{\llcorner} d x_i\label{eq:potential:1}\\
			&+\frac{1}{\lambda}\sum_{i,j=1}^d
			\partial_if(G_{\alpha,j})_\lambda(e_\beta\wedge e_j)
			\mathbin{\llcorner} d x_i.\label{eq:potential:2}
		\end{align}  
	\end{subequations}
	Notice that the second term \eqref{eq:potential:2} equals $V$. It remains to show that the first term \eqref{eq:potential:1} equals $U$.  
	Firstly, we notice that
	\begin{align*}
		\hat{g}_\alpha(\xi)
		&=\int g_\alpha(x)e^{-ix\cdot\xi}\dif x\\
		&=\int g(x^\alpha)e^{-ix^\alpha\cdot\xi^\alpha}\dif x^\alpha
		\int e^{-ix_\alpha\cdot\xi_\alpha}\dif x_\alpha\\
		&=\hat{g}(\xi^\alpha)\delta_{\xi_\alpha}.
	\end{align*}
	It follows that $G_{\alpha,j}$ only depends on $x^\alpha$.
	Moreover, we have $\partial_i G_{\alpha,j}=G_{\alpha,j}=0$ for $i,j\in\alpha$. For $i,j\notin\alpha$ we have
	$$(e_\beta\wedge e_j)
	\mathbin{\llcorner} d x_i=e_\beta\delta_{ij}.$$
	Therefore,
	$$
	\eqref{eq:potential:1}
	=f
	\left(\sum_{j=1}^d\partial_j G_{\alpha,j}
	\right)_\lambda e_\beta =U.
	$$
\end{proof}

\begin{Rem}
    For simplicity, Proposition~\ref{prop:potential:GTE} is stated for $g_\alpha$ defined only in terms of the oscillation parameter $\lambda\in\mathbb{N}$. However, an analogous statement holds when including a concentration parameter $\mu\ge1$, namely for $((g_\alpha)_\mu)_\lambda$. Here, the innermost subscript $\alpha$ denotes the multi-index specifying the variables on which the function depends, while $\mu$ and $\lambda$ are the concentration and oscillation parameters, respectively. The only difference is that the factor $\frac{1}{\lambda}$ in front of $V$ and $W$ is replaced by $\frac{1}{\lambda\mu}$, since the additional concentration produces an extra factor $\mu$ when differentiating.
\end{Rem}

\subsection{The new $(u_1,T_1,R_1)$}\label{sec:newT1}
Analogously to Section \ref{sec:newB1}, 
given a solution $(u_0,T_0,R_0)$ of the \eqref{RGTE},
we define 
\begin{equation}\label{def:newT1u1}
    u_1:=u_0 + \underbrace{\tilde{u} + \tilde{u}_c}_{\partial\tilde{\psi}},
	\quad\quad
    T_1:=T_0 + \underbrace{\tilde{T} + \tilde{T}_c}_{\partial\tilde{S}}.
\end{equation}
The new defect $R_1$ must satisfy the \eqref{GTE}, which can be written in terms of the decomposition \eqref{def:newT1u1} as follows
\begin{subequations}\label{R1:GTE}
	\begin{align}
		\partial R_1
		&=\partial_t T_1 - \partial(u_1\wedge T_1)\nonumber\\
		&= \partial(R_0-\tilde{u}\wedge\tilde{T})\label{term:quad:GTE}\\
		&+ \partial(\partial_t\tilde{S}-\tilde{T}\wedge u_0 - T_0\wedge\tilde{u})\label{term:lin:GTE}\\
		&-\partial(\tilde{u}_c\wedge (T_0+\tilde{T})
		+(u_0+\tilde{u})\wedge \tilde{T}_c
		+\tilde{u}_c\wedge \tilde{T}_c)\label{term:cor:GTE}.
	\end{align}
\end{subequations}
Recall that in the first line we have applied the hypothesis that $(u_0,T_0,R_0)$ is a solution.
We decompose $R_1$ into three terms depending on each line of \eqref{R1:GTE}
$$R_1:=R^{\text{quad}}
+R^{\text{lin}}
+R^{\text{cor}}.$$

\subsection{The mean term $(\tilde{T},\tilde{u})$}\label{sec:mean:GTE}

Analogously to Section \ref{sec:mean}, 
the new pair $(\tilde{u},\tilde{T})$ is obtained from  the old defect $R_0$, the spatial building blocks $(\mathbf{u}_\alpha,\mathbf{T}_{\alpha})$ defined in \eqref{def:intermittentshearflows}, and the cutoff $\chi$ as follows 
\begin{equation}\label{eq:defTildeT}
	\tilde{u}
	:=\frac{1}{\eta}\sum_{\alpha}\underbrace{\chi_{\alpha}|R_{0,k}|^{1/q}}_{\tilde{u}_\alpha}(\mathbf{u}_{\alpha})_\lambda,
	\quad\quad
    \tilde{T}:=\eta\sum_{\alpha}\underbrace{\chi_{\alpha}\operatorname{sgn}(R_{0,\alpha})|R_{0,\alpha}|^{1/p}}_{\tilde{T}_{\alpha}}(\mathbf{T}_\alpha)_\lambda.
\end{equation}
where $\alpha\in I (d,k+1)$,
\begin{equation}\label{Rdecomposition:GTE} 
	R_0(t,x)=\sum_{\alpha}R_{0,\alpha}(t,x)e_\alpha,
\end{equation}
and
$$
\chi_{\alpha}(t,x)
:=\chi\left(\frac{|R_{0,\alpha}(t,x)|}{\delta/N}\right).
$$
The parameter $N = N(d,k)$ is chosen sufficiently large so that \eqref{eq:cutoff:N} holds.

By applying the improved H\"older inequality (Proposition~\ref{prop:IHI}), we generalize Lemma~\ref{lem:Lestimates}: 
There exist $M>0$ and $C=C(p,\delta,\|R_0\|_{C^1})>0$ such that
\begin{align*}
\|\tilde{u}\|_{L^p}
&\leq \frac{M}{\eta}\left(\|R_{0}\|_{L^1}^{1/p} + \frac{C}{\lambda^{1/p}}\right),\\
\|\tilde{T}\|_{L^q}
&\leq M\eta\left(\|R_{0}\|_{L^1}^{1/q} + \frac{C}{\lambda^{1/q}}\right).
\end{align*}

\subsection{The quadratic error}

By definition we have
\begin{align*}
	R_0-\tilde{u}\wedge\tilde{T}
	&=\sum_{\alpha}
	\chi_{\alpha}^2R_{0,\alpha}(e_\alpha-\mathbf{u}_{\alpha}\wedge\mathbf{T}_{\alpha})_\lambda
	&&=:R^{\text{osc}}_0\\
	&+\sum_{\alpha}(\chi_{\alpha}^2-1)
	R_{0,k}e_k
	&& =: R^\chi\\
	&+\sum_{\alpha\neq\alpha'}\tilde{u}_{\alpha}\tilde{T}_{\alpha'} (\mathbf{u}_{k}\wedge \mathbf{T}_{\alpha'})_\lambda
	&&=:R^{\text{int}}
\end{align*}
We define
\begin{equation}\label{def:Rquad:GTE}
	R^{\text{quad}}
	:=R^{\text{osc}}_1
	+R^\chi
	+R^{\text{int}}.
\end{equation}

\subsubsection{The oscillation error}
Notice that
$$
e_\alpha-\mathbf{u}_{\alpha}\wedge\mathbf{T}_{\alpha}
=g_\alpha e_\alpha,
$$
where $g=(1-\mu^{d-(k+1)}\varphi_\mu^2)$ has zero mean. Hence, we can apply Proposition~\ref{prop:potential:GTE} with  $\beta=\alpha$, $f=\chi_{\alpha}^2R_{0,\alpha}$ and $g$ to obtain a smaller oscillatory error. Namely, we define
$$
R_1^{\text{osc}}
=-\frac{1}{\lambda}\sum_{i,j=1}^d\partial_i(\chi_{\alpha}^2R_{0,\alpha})(G_{\alpha,j})_\lambda (e_\alpha\wedge e_j)
	\mathbin{\llcorner} d x_i,
$$
where we recall that $G_{\alpha}=\mathrm{div}^{-1}(g_\alpha)$. It follows that
$$
\|R^{\text{osc}}_1\|_{L^1}
\leq\frac{C}{\lambda},
$$
for some $C=C(\delta,\|R_0\|_{C^1})>0$. This generalizes Lemmas~\ref{def:Rosc} and~\ref{lemma:Rosc}.

\subsubsection{The cutoff error}
Arguing exactly as in Lemma~\ref{lemma:Rcutoff}, we obtain
\begin{equation}\label{eq:cutoff:N}
\|R^\chi\|_{L^1}\leq\delta/3.
\end{equation}

\subsubsection{The intersection error}

Recall that $R^{\text{int}}$ arises from the intersection of the mixed terms. 
By shifting the building blocks, we can ensure that $R^{\text{int}}=0$ whenever $2(d-(k+1))>d$. 
This is the typical strategy used with Mikado flows in 3D. 
Otherwise, we can make $R^{\text{int}}$ small by taking the concentration $\mu$ sufficiently large. 
More precisely, as in the proof of Lemma~\ref{lemma:Rint}, we have
$$
\|\varphi_\mu(x^\alpha)\varphi_\mu(x^{\alpha'})\|_{L^1}
=\mu^{-(2a+b)}\|\varphi^2\|_{L^1}^a
\|\varphi\|_{L^1}^b,
$$
where $a\ge 1$ is the number of common variables and $a+b=d-(k+1)$. 
Thus we deduce that
$$
\|R^{\text{int}}\|_{L^1}
\le
\frac{M}{\mu}\|R_0\|_{C^0},
$$
for some universal constant $M>0$.

Therefore,
we can take $(\lambda,\mu)$ large enough, in terms of $(\delta,\|R_0\|_{C^1})$, such that
$$
\|R^{\text{quad}}\|_{L^1}\leq\delta/2.
$$

\subsection{The incompressibility correctors}
Recall that 
$$
\mathbf{T}_\alpha = g_\alpha e_{\alpha^1},
$$
where $g=\mu^{\frac{d-(k+1)}{q}}\varphi_\mu$ has zero mean. Hence, we can apply Proposition~\ref{prop:potential:GTE} with $\beta=\alpha^1$, $f=\eta\tilde{T}_\alpha$  and $g$ to obtain a $\partial$-corrector
$$
\tilde{T}_c
=\frac{\eta}{\lambda}\sum_{i,j=1}^d\partial_i\tilde{T}_\alpha (G_{\alpha,j})_\lambda (e_{\alpha^1}\wedge e_j)
	\mathbin{\llcorner} d x_i,
$$
where we recall that $G_\alpha=\mathrm{div}^{-1}(g_\alpha)$. Indeed, there is a potential 
$$
\tilde{S}=\frac{\eta}{\lambda}\sum_{j=1}^d
	\tilde{T}_\alpha(G_{\alpha,j})_\lambda(e_{\alpha^1}\wedge e_j)
$$
such that $\tilde{T}+\tilde{T}_c=\partial\tilde{S}$.

For the velocity field $\tilde{u}$, which can be identified with a 1-current,  Proposition~\ref{prop:potential:GTE} implies the existence of a 1-current $\tilde{u}_c$ and a 2-current $\tilde{\psi}$ such that $\tilde{u}+\tilde{u}_c=\partial\tilde{\psi}$.
Recalling that $\partial\circ\partial=0$, and that $\partial$ corresponds to minus the divergence for $1$-currents, this implies that the velocity field corresponding to $\tilde u + \tilde u_c$ is divergence free.

It follows that  
\begin{align*}
\|\tilde{u}_c\|_{L^p}\leq\eta\frac{C}{\lambda\mu},\\
\|\tilde{T}_c\|_{L^p}\leq\frac{1}{\eta}\frac{C}{\lambda\mu},
\end{align*}
for some $C(\delta,\|R_0\|_{C^2})>0$.
This generalizes Lemmas \ref{lemma:incompressibilitycorrectors} and \ref{lem:Lestimates:cor}.

\subsection{The linear and corrector errors}
The terms
\begin{align*}
R^{\text{lin}}&=\partial_t\tilde{S}-\tilde{T}\wedge u_0 - T_0\wedge\tilde{u},\\
R^{\text{cor}}&=-(\tilde{u}_c\wedge (T_0+\tilde{T})
		+(u_0+\tilde{u})\wedge \tilde{T}_c
		+\tilde{u}_c\wedge \tilde{T}_c)
\end{align*}
can be bounded as in Lemma~\ref{lemma:Rlin} and Lemma~\ref{lemma:Rcor}.
In particular, we can take $(\lambda,\mu)$ large enough, in terms of $(\delta,\eta,\|R_0\|_{C_{t,x}^1},\|R_0\|_{C^2},\|B_0\|_{C^0},\|u_0\|_{C^0})$, such that
$$
\|R^{\text{lin}}\|_{L^1},\,
\|R^{\text{cor}}\|_{L^1}\leq\delta/8.
$$

\subsection{Sobolev regularity}
By applying Lemma~\ref{lemma:buildingblock} and the Leibniz rule, we obtain the corresponding Sobolev estimates as in Lemma~\ref{lemma:Sobolev}, namely
\begin{align*}
\|\tilde{\psi}\|_{\dot{W}^{m,r}}
&\leq\frac{C}{\eta}(\lambda\mu)^{m-1}\mu^{(d-(k+1))\left(\frac{1}{p}-\frac{1}{r}\right)},\\[0.2cm]
\|\tilde{S}\|_{\dot{W}^{m,r}}
&\leq C\eta(\lambda\mu)^{m-1}\mu^{(d-(k+1))\left(\frac{1}{q}-\frac{1}{r}\right)}.
\end{align*}
for some $C(\delta,R_0)>0$.

\subsection{Proof of Theorems \ref{thm:GTE} and \ref{thm:DGTE}}\label{sec:proof:thm:GTE}
Analogously to Section~\ref{sec:proof:Lebesgue}, the previous estimates allow us to prove a version of the inductive estimates (Propositions~\ref{Prop:Lebesgue} and \ref{Prop:Sobolev}) for the \eqref{GTE}, for exponents in the range of Theorem~\ref{thm:GTE}. This, in turn, allows us to conclude a version of the $H$-principles (Theorems~\ref{Thm:lebesgue} and \ref{Thm:Sobolev}) for the \eqref{GTE}. We omit the corresponding statements for brevity, as they are essentially obtained by replacing $B$ with $T$. This completes the proof of Theorem~\ref{thm:GTE}.

As explained in Section~\ref{sec:DGTE}, the diffusive term can be absorbed into the defect in the range \eqref{d>k+1+p}, which allows us to prove the analogous inductive estimates and $H$-principles for the \eqref{DGTE}, in particular Theorem~\ref{thm:DGTE}.

\subsection{The diffusion GTE equation}

We proceed analogously to Section \ref{sec:fastdynamos}. For the sake of brevity, we only recall the estimates for the dangerous terms. The regularity of the perturbations is given by
$$
\|\partial\tilde{\psi}_{\text{dyn}}\|_{L_t^1 W^{1,p}}
\leq M\frac{\lambda\mu^{1-\frac{d-(k+1)}{p}}}{\mu_0^{1-\frac{1}{p_0}}},
$$
and
$$
\|\partial\tilde{S}_{\text{dyn}}\|_{L_t^\gamma C^m}
\leq C\frac{(\lambda\mu)^{m}\mu^{d-(k+1)}}{{\mu_0}^{\tfrac{1}{\gamma}-\tfrac{1}{q_0}}}.
$$
The acceleration and diffusion errors satisfy the same bound:
$$
\|R^{\text{acc}}\|_{L_{t,x}^1}
\leq C \frac{\lambda_0\mu_0^{1/q_0}}{\lambda\mu},
$$
and
$$\|\partial^\dagger\partial\tilde{S}_{\textup{dyn}}\|_{L_ {t,x}^1}
\leq C
\frac{\lambda\mu}{{\mu_0}^{1/p_0}}.$$
All the remaining terms are harmless. These estimates allow us to choose the parameters so as to close the scheme as in Section~\ref{Proof:prop:dynamo}.

\subsection{Proof of Theorem \ref{thm:DGTE:intermittent}}\label{sec:proof:thm:DGTE}
Analogously to Section~\ref{Proof:thm:hprinciple}, this choice of parameters allows us to prove a version of the inductive estimate (Proposition~\ref{prop:dynamo}), as well as a version of the $H$-principle (Theorem~\ref{thm:hprinciple}), for the \eqref{DGTE}. In particular, we deduce Theorem~\ref{thm:DGTE:intermittent}. 
As before, we omit the corresponding statements for brevity, as they are essentially obtained by replacing $B$ with $T$.

\subsection{The codimension 1 case}\label{sec:k=d-1}
All results concerning the \eqref{GTE} have been stated under the condition $0\le k\le d-2$, and we now spend a few words commenting on this. As already mentioned in the introduction, in the case $k=d$ the class of boundaryless $d$-dimensional currents in $\R^d$ is trivial, since it includes only constants. This leaves out the only remaining interesting case $k=d-1$. The reason why we can not currently include this case in our results is due to the fact that, in the construction of the intermittent shear flows in \eqref{eq:def_space_intermittency:1} and \eqref{eq:def_space_intermittency:2}, we need to build (non-trivial) functions that depend on $d-(k+1)$ variables only. When $k=d-1$ this would entail working with functions that depend on 0 variables, i.e., constants, and this is too restrictive. The extension of the results to the case $k=d-1$ is thus left for future investigations.

\section{Conclusion}\label{sec:conclusion}

In this work, we have investigated the potential of convex integration for the \eqref{GTE}, and in particular for the \eqref{dynamo1} equation. The outcomes of convex integration should be interpreted in two complementary ways. On the one hand, they highlight the limits of well-posedness theory for various partial differential equations, thereby complementing the results of \cite{BDR-JFA,BDR-TAMS,BonicattoDelNinpp} on the well-posedness of the \eqref{GTE}. On the other hand, the solutions we obtain often reproduce key features of turbulent flows.

From this perspective, there is an ongoing debate about which convex integration solutions may have a genuine physical interpretation. In this regard, our scheme is highly explicit, which raises the intriguing possibility that such constructions could, at least in part, be replicated experimentally in the laboratory to build an artificial dynamo. It is therefore natural to compare the numerical simulations in \cite{ADN02,BouyaDormy13}, see  Figure~\ref{fig:4},  with the structure of the dynamo solutions presented here, namely an intermittent superposition of Daneri–Sz\'ekelyhidi Mikado flows \cite{DaneriSzekelyhidi17}.

From a mathematical standpoint, it would be of great interest to explore alternative convex integration schemes, such as those conserving helicity \cite{FLSz21,EPPpp,GiardiSzekelyhidipp}, or to combine our approach with other models of turbulent flows (e.g., \cite{ModenaSattig20,BCKpp,BCKpp2}). Another promising direction is the combination with homogenization, as in \cite{AmstrongVicol26,BSWpp,BSWpp2}, which could lead to universal dynamos, namely velocity fields that amplify all magnetic fields. At the same time, our scheme has certain limitations, notably the lack of universality and the presence of ``bad times'', which should be addressed in future work.

Overall, the theory of the kinematic dynamo appears to have reached a stage where it can be fruitfully studied using modern mathematical tools developed over the past 15 years for the analysis of turbulent flows.

\begin{figure}[h!]
	\centering
	\includegraphics[height=4.0cm]{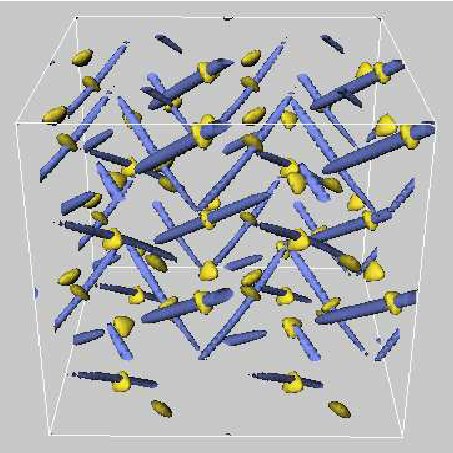}
	\quad\quad
	\includegraphics[height=4.0cm]{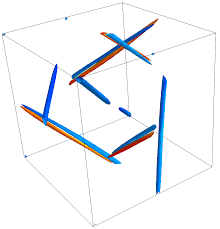}
	\caption{Magnetic field strength isosurfaces. Left figure reproduced from~\cite{ADN02} with permission of the authors; right figure reproduced from~\cite{BouyaDormy13} (available at \href{https://doi.org/10.1063/1.4795546}{https://doi.org/10.1063/1.4795546}) with permission of AIP Publishing.}
    \label{fig:4}
	\end{figure}

\appendix

\section{An eigenbasis of the magnetostatic operator} \label{sec:magnetostatic}
We consider the magnetostatic eigenvalue problem
\begin{equation} \label{e:Magnetostatic problem 2}
\nabla \times (\nabla \times B) = \lambda B, \quad \nabla \cdot B = 0, \quad B \cdot n|_{\partial \Omega} = 0, \quad (\nabla \times B) \times n|_{\partial \Omega} = 0
\end{equation}
for $H^1$-integrable fields. We first need to explain how the boundary condition $(\nabla \times B) \times n = 0$ is defined (since $B \in H^1_\sigma(\Omega)$ alone is not enough to make sense of the trace $(\nabla \times B) \times n$). If $B \in C^\infty(\bar{\Omega}) \cap H^1_\sigma(\Omega)$, then \eqref{e:Magnetostatic problem 2} is equivalent to the condition that
\begin{equation} \label{e:Weak formulation}
\int_\Omega \nabla \times B \cdot \nabla \times \varphi = \lambda \int_\Omega B \cdot \varphi \qquad \text{for every } \varphi \in H^1_\sigma(\Omega).
\end{equation}
We take \eqref{e:Weak formulation} as the weak formulation of \eqref{e:Magnetostatic problem 2} for $B \in H^1_\sigma(\Omega)$. In Proposition \ref{p:Magnetostatic proposition} below we present a regularity result on solutions of \eqref{e:Magnetostatic problem 2}, and for that recall the following special case of~\cite[Corollary 2.15]{ABDG}:
\begin{lemma} \label{l:ABDG}
Let $m \in \N$. Then the function spaces
\begin{align*}
& \{v \in L^2(\Omega): \nabla \times v \in H^{m-1}(\Omega), \ \nabla \cdot v \in H^{m-1}(\Omega), \ v \times n \in H^{m-\nicefrac{1}{2}}(\partial \Omega)\}, \\
& \{v \in L^2(\Omega): \nabla \times v \in H^{m-1}(\Omega), \ \nabla \cdot v \in H^{m-1}(\Omega), \ v \cdot n \in H^{m-\nicefrac{1}{2}}(\partial \Omega)\}
\end{align*}
are continuously embedded in $H^m(\Omega)$.
\end{lemma}

\begin{prop} \label{p:Magnetostatic proposition}
Suppose $B \in H^1_\sigma(\Omega)$ satisfies \eqref{e:Weak formulation}. Then $B \in \cap_{m \in \N} H^m(\Omega)$ and $B$ satisfies \eqref{e:Magnetostatic problem 2}. Furthermore, $\norm{B}_{H^m(\Omega)} \leq C_{\Omega,m} (1+\lambda^{\nicefrac{m}{2}}) \norm{B}_{L^2(\Omega)}$ for every $m \in \N$.
\end{prop}

\begin{proof}
Assuming that \eqref{e:Weak formulation} holds, we first show that $\nabla \times (\nabla \times B) = \lambda B$ in the sense of distributions. Let $\varphi \in C_c^\infty(\Omega)$ and consider a Helmholtz decomposition $\varphi = \psi + \nabla g$. Integrating by parts,
\[\int_\Omega B \cdot \nabla \times (\nabla \times \varphi) = \int_\Omega \nabla \times B \cdot \nabla \times \varphi = \int_\Omega \nabla \times B \cdot \nabla \times \psi = \lambda \int_\Omega B \cdot \psi = \lambda \int_\Omega B \cdot \varphi.\]
We then show that $B \in H^2(\Omega)$. By Lemma \ref{l:ABDG}, it is enough to show that $\nabla \times B \in H^1(\Omega)$. Since $\nabla \cdot (\nabla \times B) = 0$ and $\nabla \times (\nabla \times B) = \lambda B \in L^2(\Omega)$, the tangential trace $(\nabla \times B) \times n \in H^{-1/2}(\partial \Omega)$ is well-defined and by Lemma \ref{l:ABDG}, it suffices to show that $(\nabla \times B) \times n = 0$. Given $\varphi \in H^1(\Omega)$, we compute $\int_{\partial \Omega} (\nabla \times B) \times n \cdot \varphi = \int_\Omega (\nabla \times B \cdot \nabla \times \varphi - \nabla \times (\nabla \times B) \cdot \varphi) = \int_\Omega (\nabla \times B \cdot \nabla \times \varphi - \lambda B \cdot \varphi) = 0$.

We then prove the bound $\norm{B}_{H^m(\Omega)} \leq C_{\Omega,m} (1+\lambda^{\nicefrac{m}{2}}) \norm{B}_{L^2}$ for $m \geq 2$. First note that by \eqref{e:Weak formulation} we have $\norm{\nabla \times B}_{L^2} = \sqrt{\lambda} \norm{B}_{L^2}$. Thus, using Lemma \ref{l:ABDG},
\begin{align*}
    \norm{B}_{H^m}
&\lesssim_{\Omega,m} \norm{B}_{L^2} + \norm{\nabla \times B}_{H^{m-1}} \\
&\lesssim_{\Omega,m} \norm{B}_{L^2} + \norm{\nabla \times B}_{L^2} + \norm{\nabla \times (\nabla \times B)}_{H^{m-2}} \\
&\lesssim_{\Omega,m} (1+\sqrt{\lambda}) \norm{B}_{L^2} + \lambda \norm{B}_{H^{m-2}},
\end{align*}
and the claim follows by using induction.
\end{proof}

In Proposition \ref{p:Basis} below we recall a statement about an eigenbasis of the magnetostatic operator from~\cite[\textsection A]{FL20}.
We denote the set of Neumann fields on $\Omega$ by \[L^2_H(\Omega) = \tp{span}(\{H_1,\ldots,H_b\}) = \{v \in L^2_\sigma(\Omega): \nabla \times v = 0\}\]
(where $b$ is the first Betti number of $\Omega$ and $\{H_1,\ldots,H_b\}$ is an orthonormal system in $L^2_\sigma(\Omega)$). The Neumann fields are described e.g. in~\cite[Proposition 3.14]{ABDG}. Thus, $H_1,\ldots,H_b$ satisfy \eqref{e:Magnetostatic problem 2} with $\lambda = 0$. 

We endow $H^1_\sigma(\Omega)$ with the norm $\norm{\cdot}_{H^1_\sigma}$ induced by the inner product
\[\langle v, w \rangle_{H^1_\sigma} \defeq \int_\Omega \nabla \times v \cdot \nabla \times w + \sum_{i=1}^b \gamma_i \langle v \cdot n, 1 \rangle_{\Sigma_i} \langle w \cdot n, 1 \rangle_{\Sigma_i}\]
(see~\cite{ABDG} for an explanation of the cuts $\Sigma_i$; $\gamma_i>0$ are chosen so that $\norm{H_i}_{H^1_\sigma}=1$ for $i=1,\ldots,N\}$). The norm $\norm{\cdot}_{H^1_\sigma}$ is equivalent to the one inherited from $H^1(\Omega)$~\cite[Corollary 3.16]{ABDG}.

We also denote
\[H^1_\Sigma(\Omega) \defeq H^1_\sigma(\Omega) \ominus L^2_H(\Omega).\]
For every $v \in H^1_\Sigma(\Omega)$ we have $\langle v \cdot n, 1 \rangle_{\Sigma_i} = 0$ for $i=1,\ldots,N$. In particular, $H^1_\Sigma(\Omega)$ has a Poincar\'e-type inequality:
\begin{equation} \label{e:Poincare}
\norm{v}_{H^1(\Omega)}^2 \le C_\Omega \norm{\nabla \times v}_{L^2(\Omega)}^2 \quad \text{for all } v \in H^1_\Sigma(\Omega).
\end{equation}
By \eqref{e:Poincare}, the positive eigenvalues of the magnetostatic operator are bounded away from zero. We can now state an existence theorem for the magnetostatic eigenbasis of $L^2_\sigma(\Omega)$~\cite[Lemma A.6]{FL20}:

\begin{prop} \label{p:Basis}
$L^2_\sigma(\Omega)$ has an orthonormal basis $\{\psi_j\}_{j \in \N}$ with the following properties:
\begin{enumerate}
\item $\{\psi_1,\ldots,\psi_b\} = \{H_1,\ldots,H_b\}$. In particular, $\psi_1,\ldots,\psi_b$ satisfy \eqref{e:Magnetostatic problem 2} for $\lambda = 0$.

\item $\psi_j \in H^1_\Sigma(\Omega) \defeq H^1_\sigma(\Omega) \ominus L^2_H(\Omega)$ for every $j > b$.

\item Every $\psi_j$, $j > b$, satisfies \eqref{e:Magnetostatic problem 2} for some $\lambda \geq c_\Omega >0$.
\end{enumerate}
\end{prop}

We state a classical min-max characterization of the eigenvalues of the magnetostatic operator and present a proof for the reader's convenience. We also present a max-min characterization for completeness. We use the standard fact that, for every $k \in \N$,
\begin{equation} \label{e:Eigenvalues}
\lambda_k = \frac{\int_\Omega \abs{\nabla \times \psi_k}^2}{\int_\Omega \abs{\psi_k}^2} = \min_{\psi \perp \psi_1,\ldots,\psi_{k-1}} \frac{\int_\Omega \abs{\nabla \times \psi}^2}{\int_\Omega \abs{\psi}^2} = \max_{\psi \in \tp{span} \{\psi_1,\ldots,\psi_k\}} \frac{\int_\Omega \abs{\nabla \times \psi}^2}{\int_\Omega \abs{\psi}^2}.
\end{equation}
Here and below, $\perp$ refers to orthogonality in $L^2_\sigma(\Omega)$, and $\psi$ and $\phi_i$ belong to $H^1_\sigma(\Omega)$.

\begin{prop}\label{p:minimax}
The eigenvalues of the magnetostatic operator can be characterised by
\begin{align*}
\lambda_k
&= \min_{\tp{dim(span)}\{\phi_1,\ldots,\phi_k\}= k} \max_{\psi \in \tp{span}\{\phi_1,\ldots,\phi_k\}} \frac{\int_\Omega \abs{\nabla \times \psi}^2}{\int_\Omega \abs{\psi}^2} \\
&= \max_{\phi_1,\ldots,\phi_{k-1}} \min_{\psi \perp \phi_1,\ldots,\phi_{k-1}} \frac{\int_\Omega \abs{\nabla \times \psi}^2}{\int_\Omega \abs{\psi}^2}.
\end{align*}
\end{prop}

\begin{proof}
Let $k \in \N$, and denote
\[\gamma_k \defeq \sup_{\phi_1,\ldots,\phi_{k-1}} \min_{\psi \perp \phi_1,\ldots,\phi_{k-1}} \frac{\int_\Omega \abs{\nabla \times \psi}^2}{\int_\Omega \abs{\psi}^2}.\]
We first show that $\lambda_k \leq \gamma_k$. By selecting $\phi_1 =\psi_1,\ldots,\phi_{k-1}=\psi_{k-1}$ and using \eqref{e:Eigenvalues} we obtain $\lambda_k = \min_{\psi \perp \phi_1,\ldots,\phi_{k-1}} \int_\Omega \abs{\nabla \times \psi}^2/\int_\Omega \abs{\psi}^2 \leq \gamma_k$. We next show that $\gamma_k \leq \lambda_k$. Suppose $\phi_1,\ldots,\phi_{k-1} \in H^1_\sigma(\Omega)$. Since $\psi_1,\ldots,\psi_k$ are orthogonal, we can select $0 \neq \psi_0 \in \tp{span}\{\psi_1,\ldots,\psi_k\} \cap \{\phi_1,\ldots,\phi_{k-1}\}^\perp$, so that
\[\min_{\psi \perp \phi_1,\ldots,\phi_{k-1}} \frac{\int_\Omega \abs{\nabla \times \psi}^2}{\int_\Omega \abs{\psi}^2} \leq \frac{\int_\Omega \abs{\nabla \times \psi_0}^2}{\int_\Omega \abs{\psi_0}^2}.\]
Furthermore, normalizing $\psi_0$ to have $L^2$-norm $1$, so that $\psi_0 = \sum_{j=1}^k \alpha_j \psi_j$ with $\sum_{j=1}^k \alpha_j^2 = 1$, we have
\[\int_\Omega \abs{\nabla \times \psi_0}^2 = \sum_{j=1}^k \alpha_j^2 \int_\Omega \abs{\nabla \times \psi_j}^2 = \sum_{j=1}^k \alpha_j^2 \lambda_j \leq \max_{1 \leq \ell \leq k} \lambda_\ell \sum_{j=1}^k \alpha_j^2 = \lambda_k.\]
Therefore, $\gamma_k \leq \lambda_k$.

We then prove the min-max characterization. Denote
\[\rho_k \defeq \inf_{\tp{dim(span)}\{\phi_1,\ldots,\phi_k\}= k} \max_{\psi \in \tp{span}\{\phi_1,\ldots,\phi_k\}} \frac{\int_\Omega \abs{\nabla \times \psi}^2}{\int_\Omega \abs{\psi}^2}.\]
We get $\rho_k \leq \lambda_k$ as follows. Select $\phi_1=\psi_1,\ldots,\phi_k=\psi_k$. If $\psi \in \tp{span}\{\psi_1,\ldots,\psi_k\}$, we repeat the argument for $\psi_0$ that we gave above, and so $\int_\Omega \abs{\nabla \times \psi}^2/\int_\Omega \abs{\psi}^2 \leq \lambda_k$.

It remains to show that $\rho_k \geq \lambda_k$. Let $\tp{dim(span)}\{\phi_1,\ldots,\phi_k\}= k$. Our aim is to find $\psi \in \tp{span}\{\phi_1,\ldots,\phi_k\} \setminus \{0\}$ with $\frac{\int_\Omega \abs{\nabla \times \psi}^2}{\int_\Omega \abs{\psi}^2} \geq \lambda_k$. Since $\{\phi_1,\ldots,\phi_k\}$ is linearly independent, there exists $\psi \in (\tp{span}\{\phi_1,\ldots,\phi_k\} \cap \{\psi_1,\ldots,\psi_{k-1}\}^\perp) \setminus \{0\}$. But now $\frac{\int_\Omega \abs{\nabla \times \psi}^2}{\int_\Omega \abs{\psi}^2} \geq \lambda_k$ by \eqref{e:Eigenvalues}.
\end{proof}

\begin{prop}\label{prop:Weyl}For any smooth bounded domain $\Omega \subset \mathbb{R}^3$,
let $\lambda_k$ be the ordered sequence of  eigenvalues for the magnetostatic problem. Then there exists a constant $C_\Omega$ such that 

\[ \lambda_k \le C_\Omega k^{\frac{2}{3}}. \]

\end{prop}
\begin{proof}
 A quick way to obtain the  inequality  is to apply Dirichlet-Neumann bracketing philosophy \`a la Courant-Hilbert. Namely, because of the min-max characterization of $\lambda_k$ (see Proposition \ref{p:minimax}) we can compare the magnetostatic eigenvalue problem with that of the eigenvalues relative to the Stokes problem 
 with Dirichlet boundary conditions in an open ball $B(a,r) \subset \Omega$. This is simply because it is classical
  that the eigenvalues for the Stokes problem with Dirichlet boundary
 conditions, which we denoted by $\lambda^D_k$, are characterized by a min-max problem  over the space  $H^1_{0,\sigma}(B(a,r))$. 
 Indeed it can  proved following the proof of Proposition \ref{p:minimax}.
 Since 
 $H^1_{0,\sigma}(B(a,r)) \subset
 H^1_\sigma(\Omega)$, the min-max characterizations of both problems  yield that  $\lambda_k(\Omega) \le \lambda^D_k(B(a,r))$.
 
 For the unit ball the growth of the Stokes  eigenvalues is $O(k^{\frac{2}{3}})$. We  sketch a quick  proof of this well-known result:
We consider toroidal fields of the form $u = \nabla \times (\psi x)$, where, in spherical coordinates, $\psi(r,\theta,\phi) = f(r) Y_m^\ell(\theta,\phi)$ and $Y_m^\ell$ is a spherical harmonic, $\ell \in \N$ and $m \in \{-\ell,-\ell+1,\ldots,\ell-1,\ell\}$. Denoting $T \defeq (\nabla_{S^2} Y_m^\ell) \times x$ we get
\[\Delta u = (\Delta_r f) T + f \Delta_{S^2} T = \left( f''(r) + \frac{2}{r} f'(r) - \frac{\ell(\ell+1)}{r^2} f(r) \right) T.\]
Now $\Delta u = \lambda u$ if $f(r) = J_{\ell+1/2}(\sqrt{\lambda} r)$, where $J_{\ell+1/2}$ is a Bessel function of the first kind. The boundary condition $u|_{S^2} = 0$ holds if $J_{\ell+1/2}(\sqrt{\lambda}) = 0$. Thus, denoting the $k$th zero of $J_{\ell+1/2}$ by $j_{\ell,k}$, the square $(j_{\ell,k})^2$ is an eigenvalue of the Stokes operator, and its multiplicity is $2\ell+1$. We then obtain an upper bound on the zeros $j_{k,\ell}$. First, by Lommel's theorem, $J_{1/2}$ has a zero in each open interval $((m-1/2)\pi,m\pi)$, $m \in \N$ (see~\cite[p. 478]{Watsonbook}). Then, by the \emph{interlacing property} of zeros of Bessel functions of the first kind (see~\cite[p. 479]{Watsonbook}), we have
$j_{l,k} < j_{l-1,k+1} < \ldots < j_{0,k+l} < (k+l) \pi.$
Thus, a simple summation argument shows that that for each $N \in \N$, the amount of Stokes eigenvalues $\lambda_{k,\ell} = (j_{\ell,k})^2 < N\pi$ (counted with multiplicity) is at least of the order of $N^{3/2}$. This gives the sought bound $C k^{2/3}$ on the $k$th eigenvalue of the Stokes operator on the unit ball.

By rescaling the eigenfields, one sees that the eigenvalues on $B(a,r)$ grow like $C\frac{1}{r^2}k^{\frac{2}{3}}$. This is an upper bound for the growth of the magnetostatic problem by the minmax discussion above. Since $r$ only depends on the geometry of $\Omega$ the claim is proved. Of course, this quick elementary argument yields a non-optimal constant.
 \end{proof}
 \begin{cor}\label{corappendix:magneticfields} Let $B_k$ a unit eigenfield of the magnetostatic operator corresponding to $\lambda_k$. Then, 
 \begin{equation}\label{eqappendix:eigenfields:Hs}
\|B_k\|_{H^m} \leq C_{\Omega,m}(1+\lambda_k^{\frac{m}{2}}) \leq C'_{\Omega,m}k^\frac{m}{3}
\end{equation} 
for all $m,k \in \N$.

 \end{cor}

\noindent {\bf Acknowledgments.} 
The authors thank L. Sz\'ekelyhidi Jr. for valuable conversations and suggestions, as well as the hospitality of the Max Planck Institute for Mathematics in the Sciences in Leipzig.
D.F., F.M., and S.L. acknowledge the hospitality of the Institute for Advanced Study in Princeton during the Special Year on h-Principle and Flexibility in Geometry and PDEs. 
D.F acknowledges the excellent master thesis of Katerina Vergiopoulou
on topics related to this paper.
G.D.N. and F.M. acknowledge support from the Max Planck Institute for Mathematics in the Sciences in Leipzig.
D.F., S.L., and F.M. acknowledge support from QUAMAP, ERC Grant 834728, and by PID2024-158664NB-C22. D.F. acknowledges support from the Severo Ochoa Programme for Centres of Excellence (Grant CEX2019-000904-S), funded by MCIN/AEI/10.13039/501100011033.
D.F. and F.M. acknowledge support from Grants RED2022-134784-T and RED2018-102650-T, funded by MCIN/AEI/10.13039/501100011033. S.L. acknowledges support from the Finnish Centre of Excellence in Randomness and Structures.
F.M.~acknowledges support from grant RYC2023-045748-I, from grants PID2024-158418NB-I00 funded by  MCIN/AEI/10.13039/501100011033, and also from the IMUS–María de Maeztu Unit of Excellence grant CEX2024-001517-M (Apoyo a Unidades de Excelencia María de Maeztu) funded by MICIU/AEI/10.13039/501100011033.

\bibliographystyle{abbrv}
\bibliography{Turbulent_dynamos_and_GTE}

@article {ABDG,
    AUTHOR = {Amrouche, C. and Bernardi, C. and Dauge, M. and Girault, V.},
     TITLE = {Vector potentials in three-dimensional non-smooth domains},
   JOURNAL = {Math. Methods Appl. Sci.},
  FJOURNAL = {Mathematical Methods in the Applied Sciences},
    VOLUME = {21},
      YEAR = {1998},
    NUMBER = {9},
     PAGES = {823--864},
      ISSN = {0170-4214,1099-1476},
   MRCLASS = {35J25 (35Q30 65N30 76D07)},
  MRNUMBER = {1626990},
MRREVIEWER = {Juha\ H.\ Videman},
       DOI =
              {10.1002/(SICI)1099-1476(199806)21:9<823::AID-MMA976>3.0.CO;2-B},
       URL =
              {https://doi.org/10.1002/(SICI)1099-1476(199806)21:9<823::AID-MMA976>3.0.CO;2-B},
}

@article {AmstrongVicol26,
	AUTHOR = {S Amstrong and Vlad Vicol},
	TITLE = {Anomalous Diffusion by Fractal Homogenization},
	JOURNAL = {Annals of PDE},
	YEAR = {2025},
}

@article {ADN02,
	AUTHOR = {V. Archontis and S. B. F. Dorch and A. Nordlund},
	TITLE = {Numerical simulations of kinematic dynamo action},
	JOURNAL = {Astronomy \& Astrophysics},
	FJOURNAL = {Astronomy \& Astrophysics},
    VOLUME = {397},
	YEAR = {2002},
    NUMBER = {2},
    PAGES = {393--399},
}

@book {Arn04,
    AUTHOR = {Arnold, Vladimir I.},
     TITLE = {Arnold's problems},
   EDITION = {revised},
      NOTE = {With a preface by V. Philippov, A. Yakivchik and M. Peters},
 PUBLISHER = {Springer-Verlag, Berlin; PHASIS, Moscow},
      YEAR = {2004},
     PAGES = {xvi+639},
      ISBN = {3-540-20614-0},
   MRCLASS = {58-02 (00A07 01A72 37-02 53-02 57-02)},
  MRNUMBER = {2078115},
}

@book {ArnoldKhesin98,
    AUTHOR = {Arnold, Vladimir I. and Khesin, Boris A.},
     TITLE = {Topological methods in hydrodynamics},
    SERIES = {Applied Mathematical Sciences},
    VOLUME = {125},
   EDITION = {Second},
 PUBLISHER = {Springer, Cham},
      YEAR = {2021},
     PAGES = {xx+455},
      ISBN = {978-3-030-74277-5; 978-3-030-74278-2},
   MRCLASS = {58-02 (35Q30 57Z05 58B25 58D05 76-02 76M30)},
  MRNUMBER = {4268535},
       DOI = {10.1007/978-3-030-74278-2},
       URL = {https://doi.org/10.1007/978-3-030-74278-2},
}

@article {BSWpp,
	AUTHOR = {Burzcak, Jan  and Sz\'ekelyhidi, Jr., L\'aszl\'o and Wu, Wian},
	TITLE = {Anomalous {D}issipation and {E}uler {F}lows},
	JOURNAL = {arxiv:2310.02934},
	YEAR = {2023},
}

@article {BSWpp2,
	AUTHOR = {Burzcak, Jan  and Sz\'ekelyhidi, Jr., L\'aszl\'o and Wu, Wian},
	TITLE = {Scalar anomalous dissipation and optimal regularity via iterated homogenization},
	JOURNAL = {arXiv:2604.13912},
	YEAR = {2026},
}

@article {GiardiSzekelyhidipp,
	AUTHOR = {Giardi, Matteo and Sz\'ekelyhidi, Jr., L\'aszl\'o},
	TITLE = {${C}^{1/5^-}$ {C}onvex {I}ntegration {S}olutions of {I}deal {MHD}},
	JOURNAL = {arXiv:2604.12091},
	YEAR = {2026},
}

@article {BBV20,
	AUTHOR = {Beekie, Rajendra and Buckmaster, Tristan and Vicol, Vlad},
	TITLE = {Weak solutions of ideal {MHD} which do not conserve magnetic helicity},
	JOURNAL = {Ann. PDE},
	FJOURNAL = {Annals of PDE. Journal Dedicated to the Analysis of Problems  from Physical Sciences},
	VOLUME = {6},
	YEAR = {2020},
	NUMBER = {1},
	PAGES = {Paper No. 1, 40},
	ISSN = {2524-5317},
	MRCLASS = {76W05 (35D30 35Q35)},
	MRNUMBER = {4105741},
	DOI = {10.1007/s40818-020-0076-1},
	URL = {https://doi.org/10.1007/s40818-020-0076-1},
}

@article {Ber84,
    AUTHOR = {Berger, M. A.},
     TITLE = {Rigorous new limits on magnetic helicity dissipation in the solar corona},
   JOURNAL = {Geophys. Astrophys. Fluid Dyn.},
  FJOURNAL = {Geophysical and Astrophysical Fluid Dynamics},
    VOLUME = {30},
      YEAR = {1984},
     PAGES = {79--104}
}

@misc{Bonicatto-Frobenius,
      title={A general {F}robenius' {T}heorem via the {T}ransport of {C}urrents}, 
      author={Bonicatto, Paolo},
      year={2025},
      journal={arxiv:2510.21478}, 
}

@article {BDR-JFA,
    AUTHOR = {Bonicatto, Paolo and Del Nin, Giacomo and Rindler, Filip},
     TITLE = {Existence and uniqueness for the transport of currents by
              {L}ipschitz vector fields},
   JOURNAL = {J. Funct. Anal.},
  FJOURNAL = {Journal of Functional Analysis},
    VOLUME = {286},
      YEAR = {2024},
    NUMBER = {7},
     PAGES = {Paper No. 110315, 24},
      ISSN = {0022-1236,1096-0783},
   MRCLASS = {49Q15 (35Q49)},
  MRNUMBER = {4693224},
       DOI = {10.1016/j.jfa.2024.110315},
       URL = {https://doi.org/10.1016/j.jfa.2024.110315},
}

@article {BDR-TAMS,
    AUTHOR = {Bonicatto, Paolo and Del Nin, Giacomo and Rindler, Filip},
     TITLE = {Transport of currents and geometric {R}ademacher-type
              theorems},
   JOURNAL = {Trans. Amer. Math. Soc.},
  FJOURNAL = {Transactions of the American Mathematical Society},
    VOLUME = {378},
      YEAR = {2025},
    NUMBER = {6},
     PAGES = {4011--4075},
      ISSN = {0002-9947,1088-6850},
   MRCLASS = {49Q15 (35Q49)},
  MRNUMBER = {4907949},
       DOI = {10.1090/tran/9312},
       URL = {https://doi.org/10.1090/tran/9312},
}

@misc{Bonicatto-Rindler,
      title={Homogenization of elasto-plastic evolutions driven by the flow of dislocations}, 
      author={Bonicatto, Paolo and Rindler, Filip},
      year={2025},
      journal={arxiv:2410.02906},
}

@article {BouyaDormy13,
	AUTHOR = {Ismael Bouya and Emmanuel Dormy},
	TITLE = {Revisiting the ABC flow dynamo},
	JOURNAL = {Physics of Fluids},
	FJOURNAL = {Physics of Fluids},
	VOLUME = {25},
    YEAR = {2013},
    NUMBER = {3},
    PAGES = {037103},
}

@article {BLL15,
    AUTHOR = {Bronzi, Anne C. and Lopes Filho, Milton C. and Nussenzveig
              Lopes, Helena J.},
     TITLE = {Wild solutions for 2{D} incompressible ideal flow with passive
              tracer},
   JOURNAL = {Commun. Math. Sci.},
  FJOURNAL = {Communications in Mathematical Sciences},
    VOLUME = {13},
      YEAR = {2015},
    NUMBER = {5},
     PAGES = {1333--1343},
      ISSN = {1539-6746,1945-0796},
   MRCLASS = {35Q35 (35D30 76B03 76W05)},
  MRNUMBER = {3344429},
       DOI = {10.4310/CMS.2015.v13.n5.a12},
       URL = {https://doi.org/10.4310/CMS.2015.v13.n5.a12},
}

@article {BDS16,
    AUTHOR = {Buckmaster, Tristan and De Lellis, Camillo and Sz\'ekelyhidi,
              Jr., L\'aszl\'o},
     TITLE = {Dissipative {E}uler flows with {O}nsager-critical spatial
              regularity},
   JOURNAL = {Comm. Pure Appl. Math.},
  FJOURNAL = {Communications on Pure and Applied Mathematics},
    VOLUME = {69},
      YEAR = {2016},
    NUMBER = {9},
     PAGES = {1613--1670},
      ISSN = {0010-3640,1097-0312},
   MRCLASS = {35Q31 (35B10 35D30 76B03)},
  MRNUMBER = {3530360},
MRREVIEWER = {Franck\ Sueur},
       DOI = {10.1002/cpa.21586},
       URL = {https://doi.org/10.1002/cpa.21586},
}

@article {BDSV19,
	AUTHOR = {Buckmaster, Tristan and De Lellis, Camillo and Sz\'{e}kelyhidi, Jr., L\'{a}szl\'{o} and Vicol, Vlad},
	TITLE = {Onsager's conjecture for admissible weak solutions},
	JOURNAL = {Comm. Pure Appl. Math.},
	FJOURNAL = {Communications on Pure and Applied Mathematics},
	VOLUME = {72},
	YEAR = {2019},
	NUMBER = {2},
	PAGES = {229--274},
	ISSN = {0010-3640},
	MRCLASS = {76B03 (35Q31)},
	MRNUMBER = {3896021},
	MRREVIEWER = {Luigi Carlo Berselli},
	DOI = {10.1002/cpa.21781},
	URL = {https://doi.org/10.1002/cpa.21781},
}

@article {BuckmasterVicol19,
    AUTHOR = {Buckmaster, Tristan and Vicol, Vlad},
     TITLE = {Nonuniqueness of weak solutions to the {N}avier-{S}tokes
              equation},
   JOURNAL = {Ann. of Math. (2)},
  FJOURNAL = {Annals of Mathematics. Second Series},
    VOLUME = {189},
      YEAR = {2019},
    NUMBER = {1},
     PAGES = {101--144},
      ISSN = {0003-486X,1939-8980},
   MRCLASS = {35Q30 (35Q31 35Q35 76D05 76F02)},
  MRNUMBER = {3898708},
MRREVIEWER = {Isabelle\ Gruais},
       DOI = {10.4007/annals.2019.189.1.3},
       URL = {https://doi.org/10.4007/annals.2019.189.1.3},
}

@article {BCKpp2,
	AUTHOR = {Bru\'{e}, Elia and Colombo, Maria and Kumar, Anuj},
	TITLE = {Flexibility of {T}wo-{D}imensional {E}uler {F}lows with {I}ntegrable {V}orticity},
	JOURNAL = {arXiv:2408.07934},
	FJOURNAL = {arXiv:2408.07934},
	YEAR = {2024},
    NOTE = {To appear in Duke Math. J.},
}

@article {CFM19,
    AUTHOR = {Castro, \'A. and Faraco, D. and Mengual, F.},
     TITLE = {Degraded mixing solutions for the {M}uskat problem},
   JOURNAL = {Calc. Var. Partial Differential Equations},
  FJOURNAL = {Calculus of Variations and Partial Differential Equations},
    VOLUME = {58},
      YEAR = {2019},
    NUMBER = {2},
     PAGES = {Paper No. 58, 29},
      ISSN = {0944-2669,1432-0835},
   MRCLASS = {35Q35},
  MRNUMBER = {3921353},
MRREVIEWER = {Beno\^it\ P.\ Desjardins},
       DOI = {10.1007/s00526-019-1489-0},
       URL = {https://doi.org/10.1007/s00526-019-1489-0},
}

@article {CCF21,
    AUTHOR = {Castro, A. and C\'ordoba, D. and Faraco, D.},
     TITLE = {Mixing solutions for the {M}uskat problem},
   JOURNAL = {Invent. Math.},
  FJOURNAL = {Inventiones Mathematicae},
    VOLUME = {226},
      YEAR = {2021},
    NUMBER = {1},
     PAGES = {251--348},
      ISSN = {0020-9910,1432-1297},
   MRCLASS = {35Q35 (76S05)},
  MRNUMBER = {4309495},
MRREVIEWER = {Alberto\ Valli},
       DOI = {10.1007/s00222-021-01045-1},
       URL = {https://doi.org/10.1007/s00222-021-01045-1},
}

@article {CFM22,
    AUTHOR = {Castro, \'A. and Faraco, D. and Mengual, F.},
     TITLE = {Localized mixing zone for {M}uskat bubbles and turned
              interfaces},
   JOURNAL = {Ann. PDE},
  FJOURNAL = {Annals of PDE. Journal Dedicated to the Analysis of Problems
              from Physical Sciences},
    VOLUME = {8},
      YEAR = {2022},
    NUMBER = {1},
     PAGES = {Paper No. 7, 50},
      ISSN = {2524-5317,2199-2576},
   MRCLASS = {35Q35 (76S05)},
  MRNUMBER = {4406890},
       DOI = {10.1007/s40818-022-00121-w},
       URL = {https://doi.org/10.1007/s40818-022-00121-w},
}

@article {CheskidovLuo21,
    AUTHOR = {Cheskidov, Alexey and Luo, Xiaoyutao},
     TITLE = {Nonuniqueness of weak solutions for the transport equation at
              critical space regularity},
   JOURNAL = {Ann. PDE},
  FJOURNAL = {Annals of PDE. Journal Dedicated to the Analysis of Problems
              from Physical Sciences},
    VOLUME = {7},
      YEAR = {2021},
    NUMBER = {1},
     PAGES = {Paper No. 2, 45},
      ISSN = {2524-5317,2199-2576},
   MRCLASS = {35A02 (35D30 35Q35)},
  MRNUMBER = {4199851},
MRREVIEWER = {Jens\ Wirth},
       DOI = {10.1007/s40818-020-00091-x},
       URL = {https://doi.org/10.1007/s40818-020-00091-x},
}

@article {CheskidovLuo22,
    AUTHOR = {Cheskidov, Alexey and Luo, Xiaoyutao},
     TITLE = {Sharp nonuniqueness for the {N}avier-{S}tokes equations},
   JOURNAL = {Invent. Math.},
  FJOURNAL = {Inventiones Mathematicae},
    VOLUME = {229},
      YEAR = {2022},
    NUMBER = {3},
     PAGES = {987--1054},
      ISSN = {0020-9910,1432-1297},
   MRCLASS = {35Q30 (76D05)},
  MRNUMBER = {4462623},
MRREVIEWER = {Xiaojing\ Dong},
       DOI = {10.1007/s00222-022-01116-x},
       URL = {https://doi.org/10.1007/s00222-022-01116-x},
}

@article {CheskidovLuo24,
    AUTHOR = {Cheskidov, Alexey and Luo, Xiaoyutao},
     TITLE = {Extreme temporal intermittency in the linear {S}obolev
              transport: almost smooth nonunique solutions},
   JOURNAL = {Anal. PDE},
  FJOURNAL = {Analysis \& PDE},
    VOLUME = {17},
      YEAR = {2024},
    NUMBER = {6},
     PAGES = {2161--2177},
      ISSN = {2157-5045,1948-206X},
   MRCLASS = {35A02 (35D30 35Q35 35Q49)},
  MRNUMBER = {4776296},
MRREVIEWER = {A.\ El Hajj},
       DOI = {10.2140/apde.2024.17.2161},
       URL = {https://doi.org/10.2140/apde.2024.17.2161},
}

@article {BCKpp,
	AUTHOR = {Bru\'e, Elia  and Colombo, Maria and Kumar, Anuj},
	TITLE = {Sharp {N}onuniqueness in the {T}ransport {E}quation with {S}obolev {V}elocity {F}ield},
	JOURNAL = {arxiv:2405.01670},
	YEAR = {2024},
    NOTE = {To appear in J. Eur. Math. Soc.}
}

@article {CSVpp,
	AUTHOR = {Coti Zelati, Michele and Sorella, Massimo and Villringer, David},
	TITLE = {Alpha-unstable flows and the fast dynamo problem},
	JOURNAL = {arxiv:2504.00855},
	YEAR = {2025},
}

@article{Cortopassipp,
      title={A current based approach for the uniqueness of the continuity equation}, 
      author={Tommaso Cortopassi},
      year={2024},
      JOURNAL={arXiv:2402.10719},
}

@article {CSVpp2,
	AUTHOR = {Coti Zelati, Michele and Sorella, Massimo and Villringer, David},
	TITLE = {Fast dynamo action on the 3-torus for pulsed-diffusions},
	JOURNAL = {arXiv:2603.09861},
	YEAR = {2026},
}

@article {CZNF,
author = {Coti Zelati, M. and Navarro-Fern\'andez, V.},
title =  {{Three-dimensional exponential mixing and ideal kinematic dynamo with randomized ABC flows}},
journal = {arXiv:2407.18028},
year = 2024,
}

@article {SorellaVillingerpp,
	AUTHOR = { Sorella, Massimo and Villringer, David},
	TITLE = {A limsup fast dynamo on $\mathbb{R}^3$},
	JOURNAL = {arXiv:2511.23024},
	YEAR = {2025},
}

@article {CZZpp,
author = {Cheskidov, A and Zeng, Z and Zhang, D },
title =  {Global dissipative solutions of the 3{D} {N}avier {S}tokes equation and {MHD}},
journal = {arXiv:2503.05692},
year = 2025,
}

@book {ChildressGilbert03,
    AUTHOR = {Childress, S. and Gilbert, A.},
     TITLE = {Stretch, Twist, Fold: The Fast Dynamo},
    SERIES = {Lecture Notes in Physics Monographs},
    VOLUME = {37},
 PUBLISHER = {Springer},
      YEAR = {1995},
     PAGES = {xi+408},
      ISBN = {978-3-662-14014-7; 978-3-540-44778-8},
       DOI = {10.1007/978-3-540-44778-8},
       URL = {https://doi.org/10.1007/978-3-540-44778-8},
}

@article {DaneriSzekelyhidi17,
    AUTHOR = {Daneri, Sara and Sz\'ekelyhidi, Jr., L\'aszl\'o},
     TITLE = {Non-uniqueness and h-principle for {H}\"older-continuous weak
              solutions of the {E}uler equations},
   JOURNAL = {Arch. Ration. Mech. Anal.},
  FJOURNAL = {Archive for Rational Mechanics and Analysis},
    VOLUME = {224},
      YEAR = {2017},
    NUMBER = {2},
     PAGES = {471--514},
      ISSN = {0003-9527,1432-0673},
   MRCLASS = {35Q31 (35A02 35B65 35D30 76B03)},
  MRNUMBER = {3614753},
MRREVIEWER = {Jean\ C.\ Cortissoz},
       DOI = {10.1007/s00205-017-1081-8},
       URL = {https://doi.org/10.1007/s00205-017-1081-8},
}

@article {DeLellisSzekelyhidi09,
	AUTHOR = {De Lellis, Camillo and Sz\'{e}kelyhidi, Jr., L\'{a}szl\'{o}},
	TITLE = {The {E}uler equations as a differential inclusion},
	JOURNAL = {Ann. of Math. (2)},
	FJOURNAL = {Annals of Mathematics. Second Series},
	VOLUME = {170},
	YEAR = {2009},
	NUMBER = {3},
	PAGES = {1417--1436},
	ISSN = {0003-486X},
	MRCLASS = {35Q31 (34A60 35D30 76B03)},
	MRNUMBER = {2600877},
	MRREVIEWER = {Fr\'{e}d\'{e}ric Charve},
	DOI = {10.4007/annals.2009.170.1417},
	URL = {https://doi.org/10.4007/annals.2009.170.1417},
}

@article {DeLellisSzekelyhidi10,
	AUTHOR = {De Lellis, Camillo and Sz\'{e}kelyhidi, Jr., L\'{a}szl\'{o}},
	TITLE = {On admissibility criteria for weak solutions of the {E}uler equations},
	JOURNAL = {Arch. Ration. Mech. Anal.},
	FJOURNAL = {Archive for Rational Mechanics and Analysis},
	VOLUME = {195},
	YEAR = {2010},
	NUMBER = {1},
	PAGES = {225--260},	
	ISSN = {0003-9527},
	MRCLASS = {35Q31 (35A02 35L65 76N15)},
	MRNUMBER = {2564474},
	MRREVIEWER = {Stefano Bianchini},
	DOI = {10.1007/s00205-008-0201-x},	
	URL = {https://doi.org/10.1007/s00205-008-0201-x},
}

@article {DeLellisSzekelyhidi13,
    AUTHOR = {De Lellis, Camillo and Sz\'ekelyhidi, Jr., L\'aszl\'o},
     TITLE = {Dissipative continuous {E}uler flows},
   JOURNAL = {Invent. Math.},
  FJOURNAL = {Inventiones Mathematicae},
    VOLUME = {193},
      YEAR = {2013},
    NUMBER = {2},
     PAGES = {377--407},
      ISSN = {0020-9910,1432-1297},
   MRCLASS = {35Q31 (35A01 35B10 35B65 35D30 76B03)},
  MRNUMBER = {3090182},
MRREVIEWER = {Francesco\ Fanelli},
       DOI = {10.1007/s00222-012-0429-9},
       URL = {https://doi.org/10.1007/s00222-012-0429-9},
}

@article {EPP25,
    AUTHOR = {Enciso, Alberto and Pe\~nafiel-Tom\'as, Javier and
              Peralta-Salas, Daniel},
     TITLE = {An extension theorem for weak solutions of the 3d
              incompressible {E}uler equations and applications to singular
              flows},
   JOURNAL = {Forum Math. Pi},
  FJOURNAL = {Forum of Mathematics. Pi},
    VOLUME = {13},
      YEAR = {2025},
     PAGES = {Paper No. e21, 84},
      ISSN = {2050-5086},
   MRCLASS = {35Q31 (35D30)},
  MRNUMBER = {4965553},
       DOI = {10.1017/fmp.2025.10012},
       URL = {https://doi.org/10.1017/fmp.2025.10012},
}

@article {EPPpp,
	AUTHOR = {Enciso, Alberto and Pe\~nafiel-Tom\'as, Javier and
              Peralta-Salas, Daniel},
	TITLE = {H\"older continuous dissipative solutions of ideal {MHD} with nonzero helicity},
	JOURNAL = {arxiv:2507.23749},
	FJOURNAL = {	arXiv:2508.23749},
	YEAR = {2025},
}

@article {Eyi15,
    AUTHOR = {Eyink, Gregory L.},
     TITLE = {Turbulent General Magnetic Reconnection},
   JOURNAL = {Astrophys. J.},
    VOLUME = {807},
      YEAR = {2015},
     PAGES = {29 pp.},
       DOI = {10.1088/0004-637X/807/2/137},
       URL = {https://doi.org/10.1088/0004-637X/807/2/137},
}

@article {FL20,
    AUTHOR = {Faraco, Daniel and Lindberg, Sauli},
     TITLE = {Proof of {T}aylor's conjecture on magnetic helicity
              conservation},
   JOURNAL = {Comm. Math. Phys.},
  FJOURNAL = {Communications in Mathematical Physics},
    VOLUME = {373},
      YEAR = {2020},
    NUMBER = {2},
     PAGES = {707--738},
      ISSN = {0010-3616,1432-0916},
   MRCLASS = {35Q35 (76W05 78A25)},
  MRNUMBER = {4056647},
MRREVIEWER = {Steven\ David\ London},
       DOI = {10.1007/s00220-019-03422-7},
       URL = {https://doi.org/10.1007/s00220-019-03422-7},
}

@article {FLMV21,
    AUTHOR = {Faraco, Daniel and Lindberg, Sauli and MacTaggart, David and
              Valli, Alberto},
     TITLE = {On the proof of {T}aylor's conjecture in multiply connected
              domains},
   JOURNAL = {Appl. Math. Lett.},
  FJOURNAL = {Applied Mathematics Letters. An International Journal of Rapid
              Publication},
    VOLUME = {124},
      YEAR = {2022},
     PAGES = {Paper No. 107654, 7},
      ISSN = {0893-9659,1873-5452},
   MRCLASS = {76F02},
  MRNUMBER = {4317735},
MRREVIEWER = {Matteo\ Caggio},
       DOI = {10.1016/j.aml.2021.107654},
       URL = {https://doi.org/10.1016/j.aml.2021.107654}
}

@article {FLSz21,
    AUTHOR = {Faraco, Daniel and Lindberg, Sauli and Sz\'ekelyhidi, Jr.,
              L\'aszl\'o},
     TITLE = {Bounded solutions of ideal {MHD} with compact support in
              space-time},
   JOURNAL = {Arch. Ration. Mech. Anal.},
  FJOURNAL = {Archive for Rational Mechanics and Analysis},
    VOLUME = {239},
      YEAR = {2021},
    NUMBER = {1},
     PAGES = {51--93},
      ISSN = {0003-9527,1432-0673},
   MRCLASS = {76W05 (35Q60)},
  MRNUMBER = {4198715},
MRREVIEWER = {Paola\ Trebeschi},
       DOI = {10.1007/s00205-020-01570-y},
       URL = {https://doi.org/10.1007/s00205-020-01570-y},
}

@article {FLSz24,
    AUTHOR = {Faraco, Daniel and Lindberg, Sauli and Sz\'ekelyhidi, Jr.,
              L\'aszl\'o},
     TITLE = {Magnetic helicity, weak solutions and relaxation of ideal
              {MHD}},
   JOURNAL = {Comm. Pure Appl. Math.},
  FJOURNAL = {Communications on Pure and Applied Mathematics},
    VOLUME = {77},
      YEAR = {2024},
    NUMBER = {4},
     PAGES = {2387--2412},
      ISSN = {0010-3640,1097-0312},
   MRCLASS = {76W05},
  MRNUMBER = {4705296},
       DOI = {10.1002/cpa.22168},
       URL = {https://doi.org/10.1002/cpa.22168},
}

@book {Federer,
    AUTHOR = {Federer, Herbert},
     TITLE = {Geometric measure theory},
    SERIES = {Die Grundlehren der mathematischen Wissenschaften},
    VOLUME = {Band 153},
 PUBLISHER = {Springer-Verlag New York, Inc., New York},
      YEAR = {1969},
     PAGES = {xiv+676},
   MRCLASS = {28.80 (26.00)},
MRREVIEWER = {J.\ E.\ Brothers},
}

@article {ForsterSzekelyhidi18,
    AUTHOR = {F\"orster, Clemens and Sz\'ekelyhidi, Jr., L\'aszl\'o},
     TITLE = {Piecewise constant subsolutions for the {M}uskat problem},
   JOURNAL = {Comm. Math. Phys.},
  FJOURNAL = {Communications in Mathematical Physics},
    VOLUME = {363},
      YEAR = {2018},
    NUMBER = {3},
     PAGES = {1051--1080},
      ISSN = {0010-3616,1432-0916},
   MRCLASS = {35Q35 (35D30 35R35 76S05)},
  MRNUMBER = {3858828},
MRREVIEWER = {Beno\^it\ P.\ Desjardins},
       DOI = {10.1007/s00220-018-3245-2},
       URL = {https://doi.org/10.1007/s00220-018-3245-2},
}

@article {GKN23,
	AUTHOR = {Vikram Giri and Hyunju Kwon and Matthew Novack},
	TITLE = {The ${L}^3$-based strong {O}nsager theorem},
	JOURNAL = {arXiv:2305.18509},
	FJOURNAL = {arXiv:2305.18509},
	YEAR = {2023},
}

@article {GKS21,
    AUTHOR = {Gebhard, Bj\"{o}rn and Kolumb\'{a}n, J\'{o}zsef J. and
              Sz\'{e}kelyhidi, L\'{a}szl\'{o}},
     TITLE = {A new approach to the {R}ayleigh-{T}aylor instability},
   JOURNAL = {Arch. Ration. Mech. Anal.},
  FJOURNAL = {Archive for Rational Mechanics and Analysis},
    VOLUME = {241},
      YEAR = {2021},
    NUMBER = {3},
     PAGES = {1243--1280},
      ISSN = {0003-9527,1432-0673},
   MRCLASS = {76E17 (35Q31 76F25)},
  MRNUMBER = {4284526},
MRREVIEWER = {Beno\^{i}t\ P.\ Desjardins},
       DOI = {10.1007/s00205-021-01672-1},
       URL = {https://doi.org/10.1007/s00205-021-01672-1},
}

@article {GebhardKolumban22a,
	AUTHOR = {Gebhard, Bj\"{o}rn and Kolumb\'{a}n, J\'{o}zsef J.},
	TITLE = {Relaxation of the {B}oussinesq system and applications to the {R}ayleigh-{T}aylor instability},
	JOURNAL = {NoDEA Nonlinear Differential Equations Appl.},
	FJOURNAL = {NoDEA. Nonlinear Differential Equations and Applications},
	VOLUME = {29},
	YEAR = {2022},
	NUMBER = {1},
	PAGES = {Paper No. 7, 38},
	ISSN = {1021-9722},
	MRCLASS = {35Q35 (76B03 76F25)},
	MRNUMBER = {4353502},
	DOI = {10.1007/s00030-021-00739-y},
	URL = {https://doi.org/10.1007/s00030-021-00739-y},
}

@article {Gilbert88,
	AUTHOR = {A.D. Gilbert},
	TITLE = {Fast dynamo action in the Ponomarenko dynamo},
	JOURNAL = {Geophysical \& Astrophysical Fluid Dynamics},
	FJOURNAL = {Geophysical \& Astrophysical Fluid Dynamics},
	YEAR = {1988},
	VOLUME = {44},
	PAGES = {241--258},
    publisher = {Taylor \& Francis},
    doi = {10.1080/03091928808208888},
    URL = {https://doi.org/10.1080/03091928808208888},
    eprint = {https://doi.org/10.1080/03091928808208888},
}

@incollection {Gilbert03,
    AUTHOR = {Gilbert, Andrew D.},
     TITLE = {Dynamo theory},
 BOOKTITLE = {Handbook of mathematical fluid dynamics, {V}ol. {II}},
     PAGES = {355--441},
 PUBLISHER = {North-Holland, Amsterdam},
      YEAR = {2003},
      ISBN = {0-444-51287-X},
   MRCLASS = {76W05 (35Q35 35Q60 78A25 85A15 86A25)},
  MRNUMBER = {1984156},
MRREVIEWER = {Iuliana\ Oprea},
       DOI = {10.1016/S1874-5792(03)80011-3},
       URL = {https://doi.org/10.1016/S1874-5792(03)80011-3},
}

@article {HitruhinLindberg24,
    AUTHOR = {Hitruhin, Lauri and Lindberg, Sauli},
     TITLE = {Relaxation of the kinematic dynamo equations},
   JOURNAL = {Proc. Amer. Math. Soc.},
  FJOURNAL = {Proceedings of the American Mathematical Society},
    VOLUME = {152},
      YEAR = {2024},
    NUMBER = {12},
     PAGES = {5265--5278},
      ISSN = {0002-9939,1088-6826},
   MRCLASS = {35Q35 (76W05)},
  MRNUMBER = {4855882},
       DOI = {10.1090/proc/16992},
       URL = {https://doi.org/10.1090/proc/16992},
}

@article {Isett18,
    AUTHOR = {Isett, Philip},
     TITLE = {A proof of {O}nsager's conjecture},
   JOURNAL = {Ann. of Math. (2)},
  FJOURNAL = {Annals of Mathematics. Second Series},
    VOLUME = {188},
      YEAR = {2018},
    NUMBER = {3},
     PAGES = {871--963},
      ISSN = {0003-486X,1939-8980},
   MRCLASS = {35Q31 (35A02 35D30 76B03 76F02 76F05)},
  MRNUMBER = {3866888},
MRREVIEWER = {Benedetta\ Ferrario},
       DOI = {10.4007/annals.2018.188.3.4},
       URL = {https://doi.org/10.4007/annals.2018.188.3.4},
}

@book {Kampschulte,
    AUTHOR = {Kampschulte, Malte},
    TITLE = {Gradient flows and a generalized Wasserstein distance in the space of Cartesian currents},
    YEAR = {2017},
    PUBLISHER = {PhD Thesis, RWTH Aachen University},
}

@article {Kap25,
    AUTHOR = {K\"apyl\"a, P.J.},
     TITLE = {Connecting mean-field theory with dynamo simulations},
   JOURNAL = {Living Rev Sol Phys},
    VOLUME = {22},
      YEAR = {2025},
    NUMBER = {3},
     PAGES = {77 pp.},
       DOI = {10.1007/s41116-025-00042-3},
       URL = {https://doi.org/10.1007/s41116-025-00042-3},
}

@book {KR80,
    AUTHOR = {Krause, F. and R\"adler, K.-H.},
     TITLE = {Mean-Field Magnetohydrodynamics and Dynamo Theory},
 PUBLISHER = {Pergamon Press},
      YEAR = {1980},
     PAGES = {271},
      ISBN = {978-0-08-025041-0},
       DOI = {10.1016/C2013-0-03269-0},
       URL = {https://doi.org/10.1016/C2013-0-03269-0},
}

@book {Krantz-Parks,
    AUTHOR = {Krantz, Steven G. and Parks, Harold R.},
     TITLE = {Geometric integration theory},
    SERIES = {Cornerstones},
 PUBLISHER = {Birkh\"auser Boston, Inc., Boston, MA},
      YEAR = {2008},
     PAGES = {xvi+339},
      ISBN = {978-0-8176-4676-9},
   MRCLASS = {49Q15 (26-01 26B20 28-01 49Q05 58A25 58C35)},
  MRNUMBER = {2427002},
MRREVIEWER = {Andreas\ Bernig},
       DOI = {10.1007/978-0-8176-4679-0},
       URL = {https://doi.org/10.1007/978-0-8176-4679-0},
}

@article {LZZpp,
    AUTHOR = {Li, Yachun and Zeng, Zirong and Zhang, Deng},
     TITLE = {Non-uniqueness of weak solutions to 3{D} magnetohydrodynamic
              equations},
   JOURNAL = {J. Math. Pures Appl. (9)},
  FJOURNAL = {Journal de Math\'ematiques Pures et Appliqu\'ees. Neuvi\`eme
              S\'erie},
    VOLUME = {165},
      YEAR = {2022},
     PAGES = {232--285},
      ISSN = {0021-7824,1776-3371},
   MRCLASS = {35A02 (35Q30 76D05 76W05)},
  MRNUMBER = {4470114},
       DOI = {10.1016/j.matpur.2022.07.009},
       URL = {https://doi.org/10.1016/j.matpur.2022.07.009},
}

@article {Mengual22,
    AUTHOR = {Mengual, Francisco},
     TITLE = {H-principle for the 2-dimensional incompressible porous media
              equation with viscosity jump},
   JOURNAL = {Anal. PDE},
  FJOURNAL = {Analysis \& PDE},
    VOLUME = {15},
      YEAR = {2022},
    NUMBER = {2},
     PAGES = {429--476},
      ISSN = {2157-5045,1948-206X},
   MRCLASS = {35Q35 (76F25 76S05)},
  MRNUMBER = {4409883},
MRREVIEWER = {Luisa\ da Cunha e Costa Consiglieri},
       DOI = {10.2140/apde.2022.15.429},
       URL = {https://doi.org/10.2140/apde.2022.15.429},
}

@article {MengualSzekelyhidi23,
    AUTHOR = {Mengual, Francisco and Sz\'{e}kelyhidi, Jr., L\'{a}szl\'{o}},
     TITLE = {Dissipative {E}uler flows for vortex sheet initial data
              without distinguished sign},
   JOURNAL = {Comm. Pure Appl. Math.},
  FJOURNAL = {Communications on Pure and Applied Mathematics},
    VOLUME = {76},
      YEAR = {2023},
    NUMBER = {1},
     PAGES = {163--221},
      ISSN = {0010-3640,1097-0312},
   MRCLASS = {76B47},
  MRNUMBER = {4544797},
MRREVIEWER = {Alessandro\ Morando},
}

@article {MiaoYe24,
    AUTHOR = {Miao, Changxing and Ye, Weikui},
     TITLE = {On the weak solutions for the {MHD} systems with controllable
              total energy and cross helicity},
   JOURNAL = {J. Math. Pures Appl. (9)},
  FJOURNAL = {Journal de Math\'ematiques Pures et Appliqu\'ees. Neuvi\`eme
              S\'erie},
    VOLUME = {181},
      YEAR = {2024},
     PAGES = {190--227},
      ISSN = {0021-7824,1776-3371},
   MRCLASS = {35A02 (35D30 35Q30 76D05 76W05)},
  MRNUMBER = {4677154},
       DOI = {10.1016/j.matpur.2023.12.010},
       URL = {https://doi.org/10.1016/j.matpur.2023.12.010},
}

@article {MYYpp,
    AUTHOR = {Miao, Changxing and Nie, Yao and Ye, Weikui},
     TITLE = {On {O}nsager-type conjecture for the {E}ls\"asser energies of
              the ideal {MHD} equations},
   JOURNAL = {Ann. PDE},
  FJOURNAL = {Annals of PDE. Journal Dedicated to the Analysis of Problems
              from Physical Sciences},
    VOLUME = {11},
      YEAR = {2025},
    NUMBER = {2},
     PAGES = {Paper No. 31, 77},
      ISSN = {2524-5317,2199-2576},
   MRCLASS = {35Q35 (35A02 76W05)},
  MRNUMBER = {4991127},
       DOI = {10.1007/s40818-025-00224-0},
       URL = {https://doi.org/10.1007/s40818-025-00224-0},
}

@book {MD19,
    AUTHOR = {Moffatt, Keith and Dormy, Emmanuel},
     TITLE = {Self-exciting fluid dynamos},
    SERIES = {Cambridge Texts in Applied Mathematics},
 PUBLISHER = {Cambridge University Press, Cambridge},
      YEAR = {2019},
     PAGES = {xviii+520},
      ISBN = {978-1-108-71705-2; 978-1-107-06587-1},
   MRCLASS = {85A30 (76W05 86A25)},
  MRNUMBER = {3930621},
       DOI = {10.1017/9781107588691},
       URL = {https://doi.org/10.1017/9781107588691},
}

@article {ModenaSattig20,
	AUTHOR = {Modena, Stefano and Sattig, Gabriel},
	TITLE = {Convex integration solutions to the transport equation with full dimensional concentration},
	JOURNAL = {Ann. Inst. H. Poincar\'{e} Anal. Non Lin\'{e}aire},
	FJOURNAL = {Ann. Inst. H. Poincar\'{e} Anal. Non Lin\'{e}aire},
	VOLUME = {37(5)},
	YEAR = {2020},
	PAGES = {1075–1108},
}

@article {ModenaSzekelyhidi18,
	AUTHOR = {Modena, Stefano and Sz\'{e}kelyhidi, Jr., L\'{a}szl\'{o}},
	TITLE = {Non-uniqueness for the transport equation with {S}obolev vector fields},
	JOURNAL = {Ann. PDE},
	FJOURNAL = {Annals of PDE. Journal Dedicated to the Analysis of Problems from Physical Sciences},
	VOLUME = {4},
	YEAR = {2018},
	NUMBER = {2},
	PAGES = {Paper No. 18, 38},
	ISSN = {2524-5317},
	MRCLASS = {35F50 (35A02 35Q35)},
	MRNUMBER = {3884855},
	MRREVIEWER = {Rodica Luca},
	DOI = {10.1007/s40818-018-0056-x},
	URL = {https://doi.org/10.1007/s40818-018-0056-x},
}

@article {NavarroFernandezVillinger25,
author = {Navarro-Fern\'andez, Victor and Villinger, David},
title =  {{Spectral instability in the smooth Ponomarenko dynamo}},
journal = {arXiv:2509.19201},
year = 2025,
}

@article {GR23,
    AUTHOR = {Giri, Vikram and Radu, R\u azvan-Octavian},
     TITLE = {The {O}nsager conjecture in 2{D}: a {N}ewton-{N}ash iteration},
   JOURNAL = {Invent. Math.},
  FJOURNAL = {Inventiones Mathematicae},
    VOLUME = {238},
      YEAR = {2024},
    NUMBER = {2},
     PAGES = {691--768},
      ISSN = {0020-9910,1432-1297},
   MRCLASS = {35Q31 (35D30 76B03)},
  MRNUMBER = {4809443},
MRREVIEWER = {Qingtian\ Zhang},
       DOI = {10.1007/s00222-024-01291-z},
       URL = {https://doi.org/10.1007/s00222-024-01291-z},
}

@article {NunezSanz00,
	AUTHOR = {
        M. Nuñez and J. Sanz},
	TITLE = {Uniform growth rates for the magnetic field in a kinematic dynamo},
JOURNAL = {Journal of Physics},
    VOLUME = {33},
      YEAR = {2000},
     PAGES = {3605--3611},
}

@article {Ponomarenko73,
	AUTHOR = {Y. B. Ponomarenko},
	TITLE = {Theory of the hydromagnetic generator},
	JOURNAL = {J. Appl. Mech. Tech. Phys.},
	FJOURNAL = {J. Appl. Mech. Tech. Phys.},
    VOLUME = {14},
	YEAR = {1973},
    PAGES = {775--778},
}

@article {Soward87,
	AUTHOR = {A.M. Soward},
	TITLE = {Fast dynamo action in a steady flow},
	JOURNAL = {J. Fluid Mech.},
	FJOURNAL = {J. Fluid Mech.},
	YEAR = {1987},
    volume={180},
    pages={267–-295},
    DOI={10.1017/S0022112087001800},
}

@article {Rindler-Space-time,
    AUTHOR = {Rindler, Filip},
     TITLE = {Space-time integral currents of bounded variation},
   JOURNAL = {Calc. Var. Partial Differential Equations},
  FJOURNAL = {Calculus of Variations and Partial Differential Equations},
    VOLUME = {62},
      YEAR = {2023},
    NUMBER = {2},
     PAGES = {Paper No. 54, 31},
      ISSN = {0944-2669,1432-0835},
   MRCLASS = {49Q15 (28A75 49Q20 53Z05 74C05)},
  MRNUMBER = {4525735},
MRREVIEWER = {Xiangyu\ Liang},
       DOI = {10.1007/s00526-022-02332-2},
       URL = {https://doi.org/10.1007/s00526-022-02332-2},
}

@article {Row25,
	AUTHOR = {Rowan, Keefer},
	TITLE = {A subsequentially fast dynamo on $\mathbb{T}^3$},
	JOURNAL = {arXiv:2505.23936},
	YEAR = {2025},
}

@article {Tobias21,
    AUTHOR = {Tobias, S. M.},
     TITLE = {The turbulent dynamo},
   JOURNAL = {J. Fluid Mech.},
  FJOURNAL = {Journal of Fluid Mechanics},
    VOLUME = {912},
      YEAR = {2021},
     PAGES = {Paper No. P1, 76},
      ISSN = {0022-1120,1469-7645},
   MRCLASS = {76W05 (76F99 86A25)},
  MRNUMBER = {4218307},
       DOI = {10.1017/jfm.2020.1055},
       URL = {https://doi.org/10.1017/jfm.2020.1055},
}

@article {Zeldovich57,
	AUTHOR = {Zeldovich, Y. B.},
	TITLE = {The magnetic field in the two-dimensional motion of a conducting turbulent fluid},
	JOURNAL = {Sov. Phys.},
	FJOURNAL = {Sov. Phys.},
    VOLUME = {51},
	YEAR = {1957},
    NUMBER = {3},
    PAGES = {493--497},
}

@article {ZMRS84,
	AUTHOR = {Zeldovich, Y. B. and Ruzmaikin, A. A. and Molchanov, S. A. and Sokoloff, D. D.},
	TITLE = {Kinematic dynamo problem in a linear velocity field},
	JOURNAL = {J. Fluid Mech.},
	FJOURNAL = {J. Fluid Mech.},
	VOLUME = {144},
	YEAR = {1984},
	PAGES = {1-11},
}

@article {BonicattoDelNinpp,
	AUTHOR = {Bonicatto, Paolo and  Del Nin, Giacomo},
	TITLE = {Well-posedness of the transport of normal currents by time-dependent vector fields},
	JOURNAL = {arXiv:2504.15974},
	FJOURNAL = {arXiv:2504.15974},
	YEAR = {2025},
}

@article {DeLellisSzekelyhidi17,
	AUTHOR = {De Lellis, Camillo and Sz\'{e}kelyhidi, Jr., L\'{a}szl\'{o}},
	TITLE = {High dimensionality and h-principle in {PDE}},
	JOURNAL = {Bull. Amer. Math. Soc. (N.S.)},
	FJOURNAL = {American Mathematical Society. Bulletin. New Series},
	VOLUME = {54},
	YEAR = {2017},
	NUMBER = {2},
	PAGES = {247--282},
	ISSN = {0273-0979},
	MRCLASS = {35Q31 (35A01 35D30 53A99 53C21 76F02)},
	MRNUMBER = {3619726},
	MRREVIEWER = {Benedetta Ferrario},
	DOI = {10.1090/bull/1549},
	URL = {https://doi.org/10.1090/bull/1549},
}

@article{SattigSzekelyhidi2023,
  author    = {Gabriel Sattig and L\'aszl\'o Sz\'ekelyhidi Jr.},
  title     = {The Baire Category Method for Intermittent Convex Integration},
  journal   = {Acta Math. Hungar.},
  volume    = {171},
  number    = {1},
  pages     = {88--106},
  year      = {2023},
  doi       = {10.1007/s10474-023-01380-0},
  eprint    = {arXiv:2305.04644}
}

@article{FazekasKolumban2025,
  author    = {Borb{\'a}la Fazekas and József J. Kolumb{\'a}n},
  title     = {Estimating the convex relaxation of the ideal magnetohydrodynamics equations},
  journal   = {arXiv:2505.10230},
  year      = {2025},
}

@book{Markfelder21,
    AUTHOR = {Markfelder, Simon},
     TITLE = {Convex integration applied to the multi-dimensional
              compressible {E}uler equations},
    SERIES = {Lecture Notes in Mathematics},
    VOLUME = {2294},
 PUBLISHER = {Springer, Cham},
      YEAR = {[2021] \copyright 2021},
     PAGES = {x+239},
      ISBN = {978-3-030-83784-6; 978-3-030-83785-3},
   MRCLASS = {76N10 (35Q31 76-02)},
  MRNUMBER = {4385531},
       DOI = {10.1007/978-3-030-83785-3},
       URL = {https://doi.org/10.1007/978-3-030-83785-3},
}

@article{ChiodaroliFeireislKreml2015,
  author    = {Elisabetta Chiodaroli and Eduard Feireisl and Ond{\v r}ej Kreml},
  title     = {On the weak solutions to the equations of a compressible heat conducting gas},
  journal   = {Annales de l’I.H.P. Analyse non linéaire},
  volume    = {32},
  number    = {1},
  pages     = {225--243},
  year      = {2015},
  doi       = {10.1016/j.anihpc.2013.11.005},
  eprint    = {arXiv:1307.0640},
  mrnumber  = {3303948},
  zbl       = {1315.35160}
}

@book{Watsonbook,
    AUTHOR = {Watson, G. N.},
     TITLE = {A {T}reatise on the {T}heory of {B}essel {F}unctions},
 PUBLISHER = {Cambridge University Press, Cambridge; The Macmillan Company,
              New York},
      YEAR = {1944},
     PAGES = {vi+804},
   MRCLASS = {33.0X},
  MRNUMBER = {10746},
MRREVIEWER = {G.\ Szeg\"o},
}

@article {DEIJ22,
    AUTHOR = {Drivas, Theodore D. and Elgindi, Tarek M. and Iyer, Gautam and
              Jeong, In-Jee},
     TITLE = {Anomalous dissipation in passive scalar transport},
   JOURNAL = {Arch. Ration. Mech. Anal.},
  FJOURNAL = {Archive for Rational Mechanics and Analysis},
    VOLUME = {243},
      YEAR = {2022},
    NUMBER = {3},
     PAGES = {1151--1180},
      ISSN = {0003-9527,1432-0673},
   MRCLASS = {76F02 (76R50)},
  MRNUMBER = {4381138},
MRREVIEWER = {Anna\ L.\ Mazzucato},
       DOI = {10.1007/s00205-021-01736-2},
       URL = {https://doi.org/10.1007/s00205-021-01736-2},
}

@article {VishikFriedlander93,
    AUTHOR = {Vishik, Misha M. and Friedlander, Susan},
     TITLE = {Dynamo theory methods for hydrodynamic stability},
   JOURNAL = {J. Math. Pures Appl. (9)},
  FJOURNAL = {Journal de Math\'ematiques Pures et Appliqu\'ees. Neuvi\`eme
              S\'erie},
    VOLUME = {72},
      YEAR = {1993},
    NUMBER = {2},
     PAGES = {145--180},
      ISSN = {0021-7824},
   MRCLASS = {76E99 (35Q35 76C99 76D05 76E25 76W05)},
  MRNUMBER = {1216094},
}

@article{Tay74,
  title = {Relaxation of Toroidal Plasma and Generation of Reverse Magnetic Fields},
  author = {Taylor, J. B.},
  journal = {Phys. Rev. Lett.},
  volume = {33},
  issue = {19},
  pages = {1139--1141},
  numpages = {0},
  year = {1974},
  month = {Nov},
  publisher = {American Physical Society},
  doi = {10.1103/PhysRevLett.33.1139},
  url = {https://link.aps.org/doi/10.1103/PhysRevLett.33.1139}
}

\begin{flushleft}
	\quad\\
	Giacomo Del Nin\\
	\textsc{Max Planck Institute for Mathematics in the Sciences, 
		04103 Leipzig, Germany}\\
	\textit{E-mail address:} giacomo.delnin@mis.mpg.de
\end{flushleft}

\begin{flushleft}
	\quad\\
	Daniel Faraco\\
	\textsc{Departamento de Matem\'aticas, Universidad Aut\'onoma de Madrid, Instituto de
Ciencias Matem\'aticas, CSIC-UAM-UC3M-UCM, 
		28049 Madrid, Spain}\\
	\textit{E-mail address:} 
daniel.faraco@uam.es 
\end{flushleft}

\begin{flushleft}
	\quad\\
	Sauli Lindberg\\
	\textsc{Department of Mathematics and Statistics, University of Helsinki, 
	00014 Helsingin yliopisto, Finland}\\
	\textit{E-mail address:} 
sauli.lindberg@helsinki.fi
\end{flushleft}

\begin{flushleft}
	\quad\\
	Francisco Mengual\\
	\textsc{Universidad de Sevilla, IMUS, 
		41012 Sevilla, Spain}\\
	\textit{E-mail address:} fmengual@us.es
\end{flushleft}

\end{document}